\documentclass[11pt]{article}

\usepackage{amsmath}
\usepackage{amsfonts}
\usepackage{times}
\usepackage{txfonts}
\usepackage{graphicx}
\usepackage{epsfig}
\usepackage{multirow}
\usepackage{color} 
\usepackage{lscape}
\usepackage{adjustbox}
\usepackage{tabls}
\usepackage{rotating}
\usepackage{verbatim}
\usepackage{url}
\usepackage{float}
\usepackage{color}
\usepackage{comment}
\usepackage[title]{appendix}
\usepackage{makecell}
\usepackage{xcolor}
\usepackage{cases}
\usepackage{empheq}

\usepackage{enumerate}
\usepackage{placeins}
\usepackage{qtree}
\usepackage{alphalph}
\usepackage{tabulary}
\usepackage{array}
\usepackage{paralist}
\usepackage{graphics}
\usepackage{caption}
\usepackage{subcaption}
\usepackage{cases}
\usepackage{empheq}
\usepackage{endnotes}
\let\footnote=\endnote

\usepackage{nccmath}

\usepackage{bookmark}
\usepackage{cleveref}
\usepackage{geometry}
\usepackage{algpseudocode}
\usepackage{algorithm}
\usepackage{tabularx,booktabs}

\usepackage{bibentry}
\usepackage{caption}

\newcolumntype{P}[1]{>{\centering\arraybackslash}p{#1}}

\usepackage{etoolbox}
\preto\subequations{\ifhmode\unskip\fi}

\newcommand{\MLB}[1]{{\color{blue}{#1}}}

\newcommand{\ML}[1]{{\color{black}{#1}}}

\newcommand{\WMR}[1]{{\color{black}{#1}}}
\newcommand{\WMRMin}[1]{{\color{blue}{#1}}}

\usepackage[giveninits=true, maxbibnames = 20, maxcitenames = 2,
citestyle=numeric-comp]{biblatex}
\DefineBibliographyExtras{english}{%
}
\renewbibmacro{in:}{}
\DeclareFieldFormat
  [article,inbook,incollection,inproceedings,patent,thesis,unpublished]
  {title}{#1\isdot}
\AtEveryBibitem{
  \clearfield{doi}
  \clearfield{number}
  \clearfield{month}
  }

\usepackage[normalem]{ulem}

\def\Rset{{\mathbb R}}
\def\bbbone{{\mathchoice {\rm 1\mskip-4mu l} {\rm 1\mskip-4mu l}
{\rm 1\mskip-4.5mu l} {\rm 1\mskip-5mu l}}}
\def\R{{\mathcal R}}
\def\Z{{\mathcal Z}}

\newenvironment{proof}{\textbf{Proof:}}

\algdef{SE}[SUBALG]{Indent}{EndIndent}{}{\algorithmicend\ }%
\algtext*{Indent}
\algtext*{EndIndent}

\newcounter{SAME}

\newtheorem{thm}[SAME]{Theorem}%[section]

\newtheorem{prop}[SAME]{Proposition}

\newtheorem{rem}{Remark}
\DeclareFieldFormat{pages}{#1}
\DeclareFieldFormat[article]{volume}{\bibstring{jourvol}\addnbspace #1}

\newlength{\Oldarrayrulewidth}

\title{\vspace{-1.095in} 
Drone-Delivery Network for Opioid Overdose -  Nonlinear Integer Queueing-Optimization Models and Methods}

\author {Miguel A. Lejeune \footnote{George Washington University, Washington, DC, USA; {\tt mlejeune@gwu.edu}. M. Lejeune acknowledges the support of the National Science Foundation (grant ECCS-2114100) and the Office of Naval Research (grant N00014-22-1-2649).}\ , 
Wenbo Ma \footnote{George Washington University, Washington, DC, USA; {\tt wenboma2011@gmail.com}}}

\vspace{1ex}

\date{}
 \DeclareUnicodeCharacter{0301}{\'{e}}
\begin{document}
\vspace{-0.50in}
\maketitle

\vspace{-0.43in}
\begin{abstract}
\noindent 
We propose a new stochastic emergency network design model that uses a fleet of drones to quickly deliver \ML{naloxone} in response to opioid overdoses. The network is represented as a collection of $M/G/K$ queueing systems in which the capacity $K$ of each system is a decision variable and the service time is modelled as a decision-dependent random variable.
The model is an % bilocation-allocation
optimization-based queueing problem which locates fixed (drone bases) and mobile (drones) servers and determines the drone dispatching decisions, and takes the form of a  nonlinear integer problem intractable in its original form. 
We develop an efficient reformulation and algorithmic framework. 
%to solve efficiently large problem instances. 
Our approach reformulates the multiple nonlinearities (fractional, polynomial, exponential, factorial terms) to give a mixed-integer linear programming (MILP) formulation. We demonstrate its~generalizablity and show that the problem of minimizing the average response time of a \ML{collection of $M/G/K$ queueing systems} with unknown capacity $K$ is always MILP-representable. We design \ML{an outer approximation branch-and-cut algorithmic framework which is computationally efficient} and scales well. 
%Besides computational efficiency, 
%are used to show the importance of the delivery mode on the response time and probability of survival of overdose victims. 
The analysis based on real-life data reveals that drones can in Virginia Beach: 1)~decrease the response time by \ML{82\%}, %almost seven %6.82
2) increase the survival chance by more than \ML{273\%}, 
3) save up to \ML{33} additional lives per year, and 4) provide annually up to \ML{279} additional quality-adjusted life years. 
%could be saved annually in Virginia Beach and that the an additional quality-adjusted 
%(QALY).% in Virginia Beach.
%for the design of a drone-based network for the emergency delivery of naloxone in case of opioid overdose. 

\vspace{0.05in}
\noindent
Key words: opioid overdose, drone delivery,  optimization-based queueing model, survival chance, mixed-integer nonlinear programming, stochastic network design
\end{abstract}

% ----------------------------------------------------------------

\thispagestyle{plain}
\markboth{LEJEUNE-MA}{OVERDOSE}

\def\Rset{{\mathbb R}}
\def\bbbone{{\mathchoice {\rm 1\mskip-4mu l} {\rm 1\mskip-4mu l}
{\rm 1\mskip-4.5mu l} {\rm 1\mskip-5mu l}}}
\def\R{{\mathcal R}}
\def\Z{{\mathcal Z}}
\def\eps{{\varepsilon}}
% ----------------------------------------------------------------{
\baselineskip=2\normalbaselineskip

%\date{}
\vspace{-0.15in}

\baselineskip=1.3\normalbaselineskip
\setlength{\parskip}{0.01\baselineskip}
%%%%%%%%%%%%%%%%%%%%%%%%%%%%%%%%%%%%%%%%%%%%%%%%%%%%%%%%%%%%%%%%%%%%%%%%%%%%%%%%%%%%%%%%%%%%%%%%%%%%%%

\vspace{-0.0875in}
\section{Introduction}\label{sec_intro}
\vspace{-0.1in}

From 1999 to 2019, almost 500,000 people died from overdose in the USA \cite{cdc19992019} which is the leading cause of death for 25- to 64-year old US citizens. 
%and opioid use disorder concerns two million Americans \cite{dezfulian2021opioid}. 
% Annual US opioid-related mortality increased from 9489 to 47 600 deaths from 2001 to 20171,4 and is now the leading cause of death among adults 25 to 64 years of age. 
%Opioid use disorder concerns about two million Americans and is estimated to cost \$78.5 billion in annual medical expenses \cite{dezfulian2021opioid}.
%\cite{orkin2017out}: Drug overdose causes approximately 50,000 deaths in Canada and the United States combined, and 183,000 deaths worldwide annually.In the United States, United Kingdom, and Australia, overdose is responsible for more deaths each year than motor vehicle collisions.
Opioid overdose is a life-threatening condition that, if not treated within minutes, can lead to severe neurological damage and death \cite{BucklandDesignConsider}. 
% it’s a race against time (Dr. Ornato).	When it comes to rescuing someone who just overdosed on an opiate, every second counts. By the time they get the call, dispatch a unit, weave an ambulance through congested roads, arrive at the scene, unload their gear and initiate care, it’s often too late. “Time is the enemy,” Ornato, medical director of the Richmond Ambulance Authority.
The chance of survival of an opiate overdose victim decreases by 10\% with each minute passing before resuscitation is attempted \cite{ornato2020feasibility}.
Although emergency medical services (EMS) are located to optimize access to the population, the median arrival time of US EMS is between 7 and 8 minutes and can go beyond 14 minutes in rural, geographically- challenged, or high-traffic urban areas
%Currently, one in ten patients in rural areas waits nearly 30 minutes for EMS arrival.
\cite{Johnson2021impact}.
In that respect,  \citeauthor{gao2020dynamic} \cite{gao2020dynamic} report that lessening the response time by one minute traditionally requires adding ambulances, each costing around  \$200,000, whereas a \$10,000 drone can decrease the response time by two minutes. Consequently, medical drones are viewed as a   promising way to reduce response times and to enhance EMS performance \cite{PULSIRI19}.
This motivates us to design a drone-based network to provide timely medical treatment and to study its potential to increase the chance of survival of overdose victims.

%%%%%%%%%%%%%%%%%%%%%%%%%%%%%%%%%%%%%%%%%%%
\vspace{-0.1in}
\subsection{Background }\label{BACK}
\vspace{-0.085in}
The opioid overdose epidemic is ramping up as the number of incidents has quadrupled in the last 15 years. 
%While there were over 33,000 opioid-related deaths in the USA in 2015 \cite{LynnOverdoseStats}, this number has reached 50,000 in 2019 \cite{nihopioid}. 
In the midst of this crisis, there are about 130 opioid overdose deaths each day \cite{ornato2020feasibility}. 
%G2: Gao et al. \cite{gao2020dynamic}: More than 72,000 Americans died from a drug overdose with opioids, which is more than the lives lost in car crashes or gun-related incidents.
The first wave of overdoses began in 1990s and was associated to prescription opioids. The 2010 second wave was due to heroin while the third one in 2013 was caused by synthetic opioids, in particular the illicitly manufactured fentanyl \cite{cdcopioid}.
% With increase in prescription overdose in the 1990s and 2000s and non-prescribed opioids in the 2000s and 2010s, opioid overdose, as a leading cause of preventable injury and death, is a major contributor to worsening overall survival among middle-age white Americans and an increase causing of mortality among all racial and age categories. \cite{ToddAlexOpioid}. 
The ongoing COVID-19 pandemic has further  exacerbated the opioid health crisis.  %\cite{SLAVOVA}. 
According to the Centers for Disease Control and Prevention (CDC) \cite{cdc81000}, over 81,000 drug overdose deaths occurred in the USA between May 2019 and May 2020. 
% and the increases in drug overdose deaths appear to have accelerated during the COVID-19 pandemic. 
Synthetic opioids are the primary driver behind the surge in overdose deaths. 
The death counts related to the synthetic opioid increased by 38.4\% from the 12-month period ending in June 2019 to the 12-month period ending in May 2020. % \cite{cdcpandemic}.
In addition to the substantial human toll, opioid overdose poses a great burden on the economy. 
The Society of Actuaries \cite{soa2019} estimates that the economic cost of opioid misuse amounts to at least \$631 billions from 2015 to 2019 in the USA and includes the cost for healthcare services, premature mortality, criminal justice activity, and related educational and assistance program.
%As an illustration, 
The White House Council of Economic Advisers (CEA) estimates the 2015 total cost of opioid
overdoses to \$504 billion with \$431.7 billion due to mortality costs and \$72.3 billion due to health, productivity, and criminal justice costs \cite{2018Brill}.

Many overdose-related deaths are preventable if naloxone, an overdose-reversal medication, can be delivered within minutes following the breath cessation of the overdose victim \cite{BucklandDesignConsider}. Naloxone 
%has virtually no effect on people who have not taken opioid and 
can be safely administered by a layperson witnessing an overdose incident \cite{whowebiste2020} and its increased utilization has contributed to reducing the mortality rate of opioid overdoses as stressed in \cite{LynnOverdoseStats}.
%by Lynn and J. Galinkin
%who show that a timely access and administration of naloxone can effectively reduce deaths due to overdoses.
%%%%%%%%%%%%%%%%%%%%%%%%%%%%%%
%Despite the urgency of the crisis and the ease and effectiveness of the naloxone treatment, only 8\% of the US counties had established overdose education and naloxone distribution programs by 2014, and only 13\% of the counties with high overdose mortality rates had such programs \cite{BucklandDesignConsider}. 
As opioids can slow down or even entirely stop breathing, permanent brain damage can occur within only four minutes of oxygen deprivation. This underscores the need of administering naloxone timely to prevent brain damages or death 
%\cite{miniutesmatter} 
and explains why naloxone access is of the US Department of Health and Human Services’ three priority areas to combat the opioid crisis~\cite{ToddAlexOpioid}. 
In order to mitigate the opioid crisis,
% Ornato et al. 
\citeauthor{ornato2020feasibility} \cite{ornato2020feasibility} advocate the development of a system that allows 911 dispatchers, trained as drone pilots, to swiftly deliver naloxone with drones to bystanders and to guide them to administer naloxone. % while awaiting EMS arrival
%\textcolor{blue}{need to find statistics regarding the 'golden hour' and how the mortality increases as the response time increases}. 

Drones, also known as unmanned aerial vehicles (UAVs), are an emerging technology that can accomplish time-sensitive medical delivery tasks by fastening the time-to-scene response, particularly in the areas lacking established transportation network. 
A test %In a naloxone delivery experiment 
in Detroit showed that the DJI 'Inspire 2' drones turned out to be systematically faster than ambulances -- the traditional first responders -- in delivering a naloxone nasal spray toolkit \cite{TukelTimeToScene}. 
%arrival times for each delivery recorded in the experiment 
%\ML{Another pilot study in Durham County, North Carolina, shows that a network consisting of four drone bases could reduce the response times by 4 minutes and 38 seconds compared to the average times needed by ambulances to  arrive at the scene after a 911 call \cite{ye2019optdronenetwork}.} 
In order to make drone deliveries of naloxone a reality and to leverage their full capability,
%in a large-scale and real-time environment, 
drones need to be strategically deployed 
%on a network of bases 
and assigned in a timely manner to the randomly arriving overdose-triggered requests (OTRs)  to account for the spatial and temporal uncertainty of opioid overdose incidents. 
The design of drone networks that sets up drone bases, pre-positions drones, and determines their assignment to overdose emergencies remains an open problem in the literature. In that regard, \citeauthor{BucklandDesignConsider} \cite{BucklandDesignConsider} stress the need to develop drone network optimization models for opioid overdoses so that dispatchers  “{\it understand when and where to dispatch drones}”. %This study is purported to fill this gap in the academic literature and the EMS practice. 

%%%%%%%%%%%%%%%%%%%%%%%%%%%%%%%%%%%%%%%%%%%%%%%%%%%%%%%%%%%%%%%%%%%%%%%%%%%%%%%%%%%%%%%%%%%%
%%%%%%%%%%%%%%%%%%%%%%%%%%%%%%%%%%%%%%%%%%%%%%%%%%%%%%%%%%%%%%%%%%%%%%%%%%%%%%%%%%%%%%%%%%%%
%%%%%%%%%%%%%%%%%%%%%%%%%%%%%%%%%%%%%%%%%%%%%%%%%%%%%%%%%%%%%%%%%%%%%%%%%%%%%%%%%%%%%%%%%%%%
%%%%%%%%%%%%%%%%%%%%%%%%%%%%%%%%%%%%%%%%%%%%%%%%%%%%%%%%%%%%%%%%%%%%%%%%%%%%%%%%%%%%%%%%%%%%
%%%%%%%%%%%%%%%%%%%%%%%%%%%%%%%%%%%%%%%%%%%%%%%
\vspace{-0.085in}
\subsection{Structure and Contributions} \label{STRU}
\vspace{-0.085in}

The contributions of this study are fourfold spanning across modeling, reformulation, algorithmic development, and emergency medical service practice. 

\noindent
{\bf Modeling:} 
We develop a novel network design queueing-optimization model to help EMS operators respond quickly and efficiently to an opioid overdose via the drone-based delivery of naloxone.
The proposed model is a stochastic service system design model with congestion 
\cite{BERMAN-KRASS2019} 
 and represents the drone network as a \ML{collection of $M/G/K$ queueing systems}.
The model  determines the location of drone bases (DB) and drones as well as the dispatching of drones to overdose incidents. 
The model has two distinct and novel features.
%Although some similar models (e.g. $M/M/K$) have been proposed in the literature, we note two aspects that distinguish the proposed one from the literature thus showing its novelty.
First, while some studies consider $M/M/K$ queueing systems in which $K$ is a decision variable, our optimization model is the first one -- to our knowledge -- to consider $M/G/K$ queueing systems (which generalize the  $M/M/K$ system) 
in which the number of mobile servers $K$ (system capacity) 
% in the $M/G/K$ system modelled 
is a decision variable. 
The possibility to adjust the capacity of the queueing system is a critical feature. \ML{A recent study \cite{BOCHAN} argues that “{\it an integrated location-queueing model that simultaneously optimizes the location and number of drone resources can reduce overall resource needs by more than 50\% without degrading performance}”.}
%compared with a model from the literature that optimizes drone bases and the total number of drones separately"}.}
Second,  while the travel time between facilities and demand locations is for mobile servers an integral part of the service time, which hence depends on location and assignment decisions, the extant literature assumes that the average service time and the service rate are constants known a priori. 
%\WM{They rely on this constant service rate to derive a probabilistic constraints to ensure service-level reliability}. 
Contrasting with this, our model represents the service time, and thus the queueing delay and response time, as functions of the location and assignment decisions. 
This modeling approach allows us to capture the decision-dependent uncertain nature of the service time whose expected value is calculated endogenously, and can be applied (and even be more beneficial) for
other mobile server networks (e.g., ambulances, helicopters) in which the travel time can be a larger portion of the service time. We are not aware of any such model in the literature.

\noindent
{\bf Reformulation:} 
The base formulation of the proposed drone network design model is a nonlinear integer problem which includes fractional, polynomial, exponential, and factorial terms. We derive in Theorem \ref{T2} 
%ML{and Proposition \ref{PROPO2} 
a lifted mixed-integer linear programming (MILP) reformulation. 
%of the base model. 
Theorem \ref{TH-GEN} shows that the reformulation method is highly generalizable and that optimization models minimizing the \ML{average network response time} of a series of interdependent $M/G/K$ queueing systems with unknown number of servers $K$ are always MILP-representable. 
%%%%% \ML{Proposition \ref{NESTED-PROP2} presents new valid inequalities for nested multilinear polynomial terms that tighten the continuous relaxation of the reformulations.} 

\noindent
{\bf Algorithm:} 
%On the algorithmic side, 
We design two outer approximation methods to solve the lifted MILP reformulation. The first one is an outer approximation algorithm based on the concept of lazy constraints to attenuate the challenges due to the large dimension of the constraint space. 
The second algorithm is an outer approximation branch-and-cut algorithm that involves the dynamic incorporation of valid inequalities and optimality cuts. 
The computational study shows that the second algorithm based on the reduced-size MILP reformulation is the most efficient and scalable, and solves to optimality all problem instances in one hour while {\sc Gurobi} only solves 10\% of those. 
% \begin{comment}
% To solve the lifted MILP reformulation, we design an outer approximation branch-and-cut algorithmic framework that uses lazy constraints to attenuate the challenges due to the large dimension of the constraint space, and that involves the dynamic incorporation of valid inequalities and optimality cuts. The computational study shows that the proposed method is computationally efficient and scales well. It can  solve to optimality all problem instances in one hour while {\sc Gurobi} only solves 10\% of those. 
% \end{comment}
We show that the reduced response times are not only due to drones but also to our modeling approach which significantly increases the chance of survival as compared to using greedy heuristics and a benchmark model. 
In short, this study fills in the research gap underlined by \cite{BucklandDesignConsider} who argue that 
further research is needed “{\it to develop a decision support system to aid the 911 dispatchers to efficiently assign the drones in a time-sensitive environment}” such as opioid overdoses and that “{\it such decision support requires advances in resource allocation and optimization algorithms that consider under-resourced environment}”. 

\noindent
{\bf Data-driven EMS insights:} 
Using real-life overdose data, we 
demonstrate the extent to which the drone network contributes in reducing the response time, in increasing the survival chance of overdose victims, and thereby in saving lives. The cross-validation, \ML{the comparison with another model and two greedy heuristic methods}, and the simulations analyses underline the applicability and robustness of the network and its performance. Additional tests attest the low cost of a drone network and its flexibility to adjust to the spatio-temporal uncertainty of overdose incidents. The quality-adjusted 
life year (QALY) underscores the largely increased quality of life and the low incremental QALY cost.

The remainder of this paper is organized as follows. 
The literature review of the various fields intersecting with this study is in Section \ref{litrev}. 
Section \ref{sec_model} derives the closed-form formulas for the queueing metrics used to assess the performance of the drone response network, presents the base formulation of the model, and analyzes its complexity. 
Section \ref{sec_reformulation} is devoted to the reformulation method while Section \ref{sec_ALGO} describes the algorithms. 
Section \ref{sec_TESTS} describes the numerical  study based on overdose real-life data from Virginia Beach, presents the computational efficiency of the method, and analyzes the healthcare benefits obtained with our model. Section \ref{sec_conclusion} provides concluding remarks. 
We refer the reader to Appendix \ref{sec:abbre} for a list of the acronyms used in this paper and to Appendix \ref{sec:notations} for a list of the mathematical notations.

%and presents the modeling contributions of this study.
%%%%%%%%%%%%%%%%%%%%%%%%%%%%%%%%%%%%%%%%%%%%
%%%%%%%%%%%%%%%%%%%%%%%%%%%%%%%%%%%%%%%%%%%%
%%%%%%%%%%%%%%%%%%%%%%%%%%%%%%%%%%%%%%%%%%%%
%%%%%%%%%%%%%%%%%%%%%%%%%%%%%%%%%%%%%%%%%%%%
\vspace{-0.15in}
\section{Literature Review and Gaps}\label{litrev}
\vspace{-0.12in}
This study propose contributions spanning across three fields -- drone-based response to medical emergencies,  %opioid overdose, 
network design queueing models with mobile servers, and fractional 0-1 programming - which we review below.
%This study overlaps with five fields, namely opioid overdose response optimization, facility location problems for emergency medical service (EMS), queueing theory -based location models, drone applications for EMS and mixed-integer nonlinear programming (MINLP), which we now review succinctly. 
%%%%%%%%%%%%%%%%%%%%%%%%%%%%%%%%%%%%%%%%%%%%
%%%%%%%%%%%%%%%%%%%%%%%%%%%%%%%%%%%%%%%%%%%%
\vspace{-0.15in}
\subsection{Drone-based Response to Medical Emergencies}
\vspace{-0.1075in}
%In view of the surge of opioid overdoses, there is a growing need for new approaches and studies addressing the provision of an efficient and timely response to overdoses  \cite{LynnOverdoseStats}.
EMSs struggle to provide a timely response to medical emergencies leading to   
%due, for example, to traffic congestion, hard-to-access areas, or lack of ambulances. 
patients passing away due to the delayed arrival of ground ambulances.
%\cite{HART}. 
%,NIMILAN}.
%Drones can also take on-scene imagery to support decision making in emergency operation and ambulance dispatch.
Medical drones are %increasingly 
viewed as an efficient solution to assist EMS professionals, including for overdoses. They can quickly deliver medicine to patients and take on-sight imagery to support EMS personnel and operations %decision making in emergency operation and ambulance dispatch5images and informing EMS personnel 
before the arrival of ambulances \cite{PULSIRI19}. 
%thereby, contributing to increasing the survival rate 
% [Pulsiri et al. (2019b)].
We review next the literature on the use of drones for medical emergencies. We focus on overdoses, the topic of this study, and cardiac arrests to which most drone-related studies have been devoted in the medical sphere \cite{MERMIRI}. We refer the reader to \cite{211ROSSER} for a recent review on the use of drones in health care. 
% Yet, the extant literature is fairly scant. The very few studies we are aware of are summarized below.

%%%%%%%%%%%%%%%%%%%%%%%%%%%%%%%%%%%
% \citeauthor{LynnOverdoseStats} \cite{LynnOverdoseStats}  address the questions of optimal naloxone dosage to reverse opioid-induced respiratory depression while reducing opioid withdrawal symptoms. 
% Ye et al. 
\citeauthor{ye2019optdronenetwork} \cite{ye2019optdronenetwork} propose a drone-delivery optimization model for overdoses solved with a genetic algorithm and tested on suspected opioid overdose data from Durham County, NC.
%%We created a drone geospatial network model based on current technological specifications and potential base locations. All 29 fire, paramedic and EMS administrative buildings within Durham County were considered potential drone bases. The drone network modelled response time and county coverage. Then a genetic optimization model was built to maximize county coverage by drones and the number of overdoses covered per drone base. From this model, we quantified the optimum reduction in average response time balanced with the number drone bases required.
%The average time spanning across the 911 call to ambulance arrival at scene was 10 minutes 46 seconds. 
%In order to maximize county coverage, 
The model reduces the response time to 4 minutes 38 seconds by using four drone bases
% (instead of 10 minutes 46 seconds with ambulances) 
 and provides a 64.2\% coverage of the county.
%\ML{Similarly, a delivery experiment in Detroit shows that drones are faster than ambulances to deliver intranasal naloxone \cite{TukelTimeToScene}.}
% Gao et al. 
\citeauthor{gao2020dynamic} \cite{gao2020dynamic} propose a Markov decision process model \MLB{that selects which drone should be dispatched to an overdose incident and then relocates the drone.} Confronted with the curse of dimensionality, the authors develop a state aggregation heuristic  to derive a lookup table policy. A simulation based on EMS data from Indiana reveals that the 
state aggregation approach outperforms the myopic policy used in practice, in particular when the overall request intensity gets higher, but also highlights the difficulty to estimate the value function. % used within their approach. %state aggregation approach.
\citeauthor{ornato2020feasibility} \cite{ornato2020feasibility} report that all participants in a pilot feasibility study based on 30 simulated opioid overdoses were able to 
%followed the simulated 9-1-1 dispatcher instructions accurately and 
administer the intranasal naloxone medication to the manikin within about two minutes after the 911 call and that 97\% of them felt confident that they could do it successfully in a real event.
%In general, EMS planning models can be of three types: 
%%%%%%%%%%%%%%%%%%%%%%%%%%%%%%%%%
%the EMS planning problems can be divided into three different categories (1) strategic level where decisions are made for several years, e.g., facility location problems;  (2) tactical level where decisions often hold for several months, e.g., staff scheduling; and (3) operational level where decisions are made on a daily basis or even in real time, e.g., patient transport scheduling. 

We now review the drone literature for out-of-hospital cardiac arrests (OHCA).
Scott and Scott \cite{scott2019models} propose two set covering formulations based on a tandem air-ground  response to OHCAs.
The models set up the location of the depots (where ground vehicles are stationed) and drone bases without taking into account congestion.
%and delays. 
Medical supplies are delivered with a ground-based vehicle from a depot to a drone location and the last-mile delivery to the OHCA location is carried out with a drone. 
Pulver et al. \cite{pulver2016locating} determine where to set up drone bases in order to maximize the number of OHCA requests responded to  within one minute. 
%The binary decision variables are whether a drone base is established and whether OHCA locations can be served by at least one drone base within a predetermined time. 
This coverage-type facility location model does not concern with the possible congestion of the network and assumes that the demand and the drone response times are deterministic. 
%\ML{The model does not tackle the placement of drones, nor their assignment to OHCAs.}
%%%%%%%%%%%%%%%%%%%%%%%%%%%%%
Pulver and Wei \cite{pulver2018optimizing} propose a multi-objective set covering model with backup coverage. Decisions pertaining to 
the positioning of drones and their assignment to OHCAs are taken to maximize the backup and primary coverage of the network. 
%%%%%%%%%%%%%%%%%%%%%%%%%%%%%%%%%%%%%%%%%%%%%%%%%%%%%%%%%%%%%%%%%%%%%
Boutilier et al. \cite{boutilier2017optimizing} present a two-step %decoupled 
approach to design a drone network delivering automated external defibrillators. % (AED). 
In the first step, a set covering model is solved approximately to determine the number of drone bases needed to reduce the median response time. 
%The decision variables determine which drone bases should be established and which bases will be assigned to cover each OHCA incident. The first step uses a set covering model to determine the number and location of drone bases, as well as which drone bases will cover each OHCA location.
In the second step, an $M/M/K$ queueing model determines heuristically the number of drones at each base. 
The opened stations are fixed (determined in the first stage), and the arrival rate and expected service time %of each drone 
are defined as fixed parameters determined ex-ante. %in an exogenous manner. 
Boutilier and Chan \cite{BOCHAN} propose two coverage-type models that minimize the number of drones needed to satisfy a response time target 
(i.e., decrease in expected value or conditional value-at-risk of response time).
%The two models are essentially coverage-type models as they minimize the costs (i.e., number of drones) 
%threshold ($\gamma$-second decrease in the expected value or in the conditional value-at-risk of the baseline response time). 
The models do not reward, nor incentivize faster response times. For example, if the target is to decrease the response time by, say, 20 seconds, 
\ML{the models consider a 20-second and an 80-second reduction in the response time as equally satisfactory.}
%The models locate drones at every bases and determine the proportional assignment of OHCAs to drone bases.}
The queueing models underlying the formulations assume that the service and response times are independent from the utilization of drones. 
%, that these latter are systematically available when needed, and overlook the waiting times}. 
%We also note that the proposed models define the assignment variables as continuous ones and does not determine the location of the drone bases that should be set up but instead defines the number of drones to be placed at fixed drone base locations. 
%%%%%%%%%%%%%%%%%%%%%%%%%%%%%%%%%%%%%%%%%%%%%%%%%%%%%%%%%%%%%%%%%%%%%%%%%%%%%%%%%%%%%%%%%%%%%%%%%%%%%%%%%%%%%%%%
%Wankm{\"u}ller et al. \cite{wankmuller2020optimal} focus on the use of drones for OHCAs occurring in mountainous regions and model the altitude effect on the travel time.
%The model determines the location of the drone bases and the drone dispatching policy with the objective of minimizing an objective function defined as a convex combination of number of drones and travel time.
%The network congestion and the time needed to recharged and clean a drone after service are not accounted for. 
%%%%%%%%%%%%%%%%%%%%%%%%%%%%%%%%%%%%%%%%%%%%%%%%%
Mackle et al. \cite{mackle2020data} present a genetic algorithm to calculate the number of drone bases needed to satisfy a response time threshold. 
%The authors use a simulator to estimate the potential response time benefits of the drone network.
Bogle et al. \cite{BogleQALY} use a coverage-type facility location model to determine the number and location of drone stations 
%needed in North Carolina 
so that a specified number of at-risk OHCA victims are reachable within a certain time.
Cheskes et al. \cite{cheskes2020improving} use simulation to examine the feasibility and  benefits of using drones for OHCAs in rural and remote areas. % and how drones can reduce response times. 
%Their study suggests feasibility of AED drone delivery along with improvements in response times during simulated OHCA scenarios, and the requirement for further research for optimizing the integration of drones into the emergency response system.  
%\cite{zegre2020delivery} performs an experimental simulation-based study to describe the experience of those who used AED-equipped drones in an OHCA incident in comparison with those who searched for a public access AED. Their drone recipient participants reported largely positive experiences, highlighting that drone-based AED deliveries enable them to stay with the victim and continue cardiopulmonary resuscitation. Concerns were few but included drone arrival timing and direction as well as bystander safety. 
%Bauer et al. \cite{BAUER} show the cost effectiveness of AED-equipped~drones for OHCAs in rural areas with no access to EMSs in 10 minutes. Cost effectiveness is calculated as the ratio of financial costs to additional life years gained compared with the current~EMS~practice. 

%%%%%%%%%%%%%%%%%%%%%%%%%%%%%%%%%%%%
The above models 
%do not consider jointly the location of the fixed (drone bases) and mobile (drones) servers and the assignment decisions. Second, they
assume that the average service time is fixed and independent from the drone assignment decisions.
Additionally, they are all covering models in which the objective function is to minimize the total costs or the number of drones (bases) but does not reward a quicker response. %Instead, they view all response as equally satisfactorily as long as they do not exceed a specified time threshold. 
%%%%%%%%%%%%%%%%%%%%%%%%%%%%%%%%
Our model contributes to fill in the gaps in the literature. 
Our strategic network model 
%(i) concurrently determines the location of the drone bases, the positioning of the drones, and the drone-dispatching policy to OTRs; 
(i) is a survival network design model as it minimizes the response time which is the primary driver to increase the survival chance of overdosed patients; and
(ii) represents the service time as a distribution-free decision-dependent random variable with service rate defined endogenously as a function of the drone dispatching~policy.

We explain next how our model differs significantly from the OHCA models proposed in \cite{BOCHAN} which have, as our model, a queueing underlying structure.
First, the objective in \cite{BOCHAN} to minimize the number of DBs is typical in coverage-type facility location models which, \ML{as noted in the literature}, are unable to differentiate the impact of varying response times \cite{ERKUT2008}, or to reward shorter ones and whose suitability for time-critical medical events \ML{has been extensively questioned (see, e.g., \cite{ERKUT2008,janovsikova2021coverage,MILLS}).}
%in the health care literature.
In contrast, our model explicitly rewards 
%(minimizes) 
quick response time and is hence conducive to maximizing the number of opioid overdose incident survivors. As the overdosed patients’ chance of survival rapidly deteriorates with the time-to-treatment, 
%time passing until receiving treatment, 
it is critical to use an objective function which provides a reward diminishing with the time needed to respond.
The maximization of the  QALY is an alternative objective function that could fulfill the true needs of a medical emergency response network.
%\MLR{Our approach implements this principle, 
%in contrast to the models in \cite{BOCHAN}, 
%and is in line with Buckland et al. \cite{BucklandDesignConsider} who state that the objective of the drone network is to minimize the response times and account for delays as ``{\it a dispatch delay of even a minute may lead to increased likelihood of patient death}”.}
Second, the underlying queueing models in \cite{BOCHAN} assume that the service times are exponentially distributed with fixed rates (fixed average service time regardless of drones' placement) and that the response time between any pair of drone bases and OHCAs incidents is fixed ex-ante and independent from drone utilization and network congestion. 
%Assuming that the average response time is fixed equals to assuming that the response time is independent from the congestion in the network and from the assignment decisions, i.e., the response time is invariant with the utilization of the drones.
%, which is itself impacted by the drone pre-positioning decisions. 
%\MLR{The rationale in \cite{BOCHAN} is that the drone travel time is a smaller component of service time – which is quite a questionable assumption that has been heavily challenged in the literature.}
Such assumptions do not allow for an accurate representation of the queueing system as explained in a \ML{2023 study \cite{XU}}. The authors \cite{XU} demonstrate that it is critical to track the actual utilization of facilities and servers so that the known impact of utilization can be accurately reflected on the congestion level, waiting time, and response time. Along the same line, it was previously  shown \cite{MILLS} that assuming that facilities’ workloads are known a priori % and relatively homogeneous, 
and that each facility can provide the same service level irrelevant of their actual utilization significantly hurts the network’s efficiency and reliability.
We do not resort to such simplifying assumptions and instead model the opioid overdose arrival rate, the service rate, and the average response times as unknowns dependent on the utilization of the resources and the assignment policy, and define them as decision variables whose values are determined endogenously by the model.

%%%%%%%%%%%%%%%%%%%%%%%%%%%%%%%%%%%%%%%%%%%%%%%%%%%%%%%%%%%%%%%%%%%%%%%%%%%
%%%%%%%%%%%%%%%%%%%%%%%%%%%%%%%%%%%%%%%%%%%%%%%%%%%%%%%%%%%%%%%%%%%%%%%%%%%
\vspace{-0.1205in}
\subsection{Queueing-optimization Models for Stochastic Network Design}
\vspace{-0.085in}
Queueing models can be categorized into two types: descriptive queueing models which conduct an after-the-fact analysis to analyze how a predefined system configuration has performed while optimization-based (prescriptive)  queueing  models deal with congestion and are used for decision-making purposes (e.g., number of servers, location) \cite{MARIANOV1994QueuingProb}. 
We propose a new optimization-based queueing model that belongs to the class of stochastic network design queueing models with congestion  \cite{BERMAN-KRASS2019}. This model class can be further decomposed into two sub-categories depending on whether the servers are fixed (immobile) requesting the customers to travel to the facility, or mobile and travel to the demand location. 
Since we consider the delivery of naloxone with drones, our literature review is focused on mobile servers, which is much less extensive than the one for immobile servers (see, e.g., \cite{ANJOS}). % and the references therein).
%%%%%%%%%%%%%%%%%%%%%%%%%%%%%%%  
%We refer the reader to surveys by \cite{BERMAN2004} and \cite{INGOLFSSON2013} (the latter focused on emergency systems planning). 
%%%%%%%%%%%%%%%%%%%%%%%%%%%%%%%%%%%%%%%%%% 

The first queueing optimization model with mobile server \cite{berman1985optimal} determines the optimal location of a single facility in order to either maximize the service coverage or to minimize the response cost. Modelling the network as an $M/G/1$ system, the authors consider that, if the server is busy when the demand arrives, the demand can be either rejected (demand is lost and covered by backup servers at a cost) or enter a queue managed under the first-come-first-serve discipline. 
Building on this study, Chiu and Larson \cite{CHIU1985LocANSev} consider a single facility operating as an $M/G/K$ loss queueing system % with no queue allowed. 
and show that, under a mild assumption, the optimal location reduces to a Hakimi median, which minimizes the average travel time to a client. % \cite{Hakimi1964}. 
Batta and Berman \cite{batta1989location} study an $M/G/K$ system 
%that allows for queues with fixed number of servers  
and develop an an approximation-based approach to minimize the average response time. 
%and approximate the average queueing delay using the approach presented by \cite{nozaki1978approxi}. 
While the above studies consider a single facility, 
%the subsequent literature expands its scope and propose models that locate multiple  facilities.  
the queueing probabilistic location set-covering model proposed in \cite{MARIANOV1994QueuingProb} considers several facilities, each operating as an $M/M/K$ queueing system. The authors determine the minimum number of facilities needed so that the probability of at least one server being available is greater than or equal to a certain threshold. Later, operating under the auspices of an $M/G/K$ system, Marianov and Revelle \cite{MARIANOV1996QueuingMax} study the probabilistic version of the maximal covering problem 
%, coined maximal availability location problem, 
to site a limited number of emergency vehicles with the goal of maximizing availability when a call arrives. 
%The more recent literature tackles the multiple facilities and multiple servers problems. 
% a model to site $p$ servers in a multi-facility network was proposed %\cite{Berman2007MulServerLoc} to minimize the sum of the travel time and the average time spent by customers in the server. They assume that each facility acts as an $M/M/K$ system and obtain approximate solutions using metaheuristics. 
Considering an $M/M/K$ system, \citeauthor{Aboolian2008LocAlloc} \cite{Aboolian2008LocAlloc} determine the number of servers needed to minimize 
%a cost function encompassing 
the setup costs to open facilities and operate servers, the travel costs, and the queueing delay costs. 
In the EMS context, 
%a location-allocation models is proposed in \cite{CJLT} to establish a trauma center network  in South Korea and a medical evacuation network is presented in \cite{lejeune2018aeromedical} to evacuate the most seriously injured soldiers from the the battlefield. Both studies maximize the expected number of patients that can be evacuated without delay and devise new mixed-integer nonlinear programming (MINLP) algorithmic techniques. 
\citeauthor{boutilier2017optimizing} \cite{boutilier2017optimizing} propose a decoupled approach to design a drone network for OHCAs. 
%set up the drone-based delivery of automated external defibrillators (AED). The first stage involves the heuristic solution of a set covering model to determine the number of DBs to open.
Each DB operates as an $M/M/K$ system in which the expected service time and the arrival rate are fixed parameters. %determined in an exogenous manner.
The coverage-type queueing models in \cite{BOCHAN} minimize the number of drones needed to satisfy a response time target. They assume that the service rates and response times are fixed and independent from the utilization of the drones, their positioning at DBs, and the congestion in the network.
The queueing models in \cite{BOCHAN} extend the one in \cite{MarianovCoverMMd} for fixed servers and in which the objective is to minimize the number of facilities so that the waiting time or queue size does not exceed a set threshold.
% Marianov and Serra(2002) propose a set covering type of model to minimize the number of facilities established, while ensuring the probability of either a long waiting time or long queues is smaller than a pre-specified figure. 
%Given the study focused on fixed servers, they assume that the service rate of each server is a known constant and does not depend on the assignment variable. Boutilier and Chen (2022) use the same assumption in a mobile server context by justifying that drone travel time represents a smaller component of service time . In contrast to the these two studies, our model release that assumption by modeling the service time as a function depending on an assignment decision variable (drone’s travel time between its base and demand location). This advantage makes our model generally applicable to other mobile servers such as ambulances where service time accounts for a substantial portion of service time.
An ongoing study \cite{LEMA22} considers an $M/G/1$ queueing system for responding to cardiac arrests and proposes new optimality-based bound tightening techniques.
%%%%%%%%%%%%%%%%%%%%%%%%%%%%%
In contrast to the above health care studies,
%used in the medical emergency context,
%, (see, e.g., \cite{boutilier2017optimizing}), 
this work considers that the service times, the delays, the response times, and the arrival process of (overdose) requests assigned to drones are random variables whose parameters (mean) are endogenized. In particular, the response times and the service rates are decision variables whose values are determined by the model and depend on assignments. 

\vspace{-0.1in}
\subsection{Fractional Nonlinear 0-1 Programming}
\vspace{-0.1in}
As it will be shown in Section \ref{sec_formu}, the optimization problem studied 
%is an integer nonlinear programming problem. It 
is a nonlinear fractional integer problem %\cite{STANCU}. 
in which the objective function is a ratio of two nonlinear functions each depending on integer variables.
The class of problems closest to ours is that of fractional linear 0/1 (binary) problems for which significant improvements have been made over the last years \cite{BGP16,MGP}. 
Such problems arise in multiple applications, such as chemical engineering, clinical trials, facility location \cite{LIN}, product assortment \cite{BRONT}, and revenue management (see \cite{BGP17} for a review). We also note a few fractional mixed-integer queueing-optimization problems (see, e.g., \cite{ANJOS,ELHEDLI,HAN}. 
%Fractional binary optimizations arise naturally in many contexts that involve optimization of efficiency measures (e.g., maximizing the ratio of return/investment or profit/time and minimizing the ratio of cost/time, see [10,27,32,34]), averages, probabilities and percentages, among others. Fractional optimization models can be found in diverse application areas including problems in data mining stochastic service systems [15], finding alternative solutions to binary linear programs [38], clinical trials [7], and so on. For an overview of applications and solution methods for FPs we refer to a recent survey in [10].
Such problems pose serious computational challenges due to the pseudo-convexity and the combinatorial nature of the fractional objective.
They minimize a single or a sum of ratios in which the denominator and numerator are linear functions of binary variables. When more than one ratio term is in the objective function, the problem, even unconstrained, is NP-hard \cite{28}. %{29}. 

%Generally speaking, problems of the form (1) can be viewed as a class of linearly constrained pseudo-boolean (set-function) optimization problems (see a detailed survey in [17]), and as such, they can be transformed to quadratic 0–1 programming problems (see a survey in [64]). 
%Fractional 0–1 programs are also related to a class of fractional optimization problems, where the numerator and denominator are affine functions of the decision variables that are not necessarily required to be binary, see comprehensive treatments of these problems in [39,99,108].
%%%%%%%%%%%%%%%%%%%%%%%%%%%%%%%%%%%%%%%%%
One approach to tackle fractional linear 0-1 programs is to move the fractional terms to the constraint set and to then 
%, for the resulting problem, to 
derive MILP reformulations, which 
%using, for example, the  McCormick inequalities. % or the RLT method. 
requires the introduction of continuous auxiliary variables and big-M constraints  (see, e.g., \cite{22}).
%To solve fractional binary programs, several mixed-integer linear programming (MILP) reformulations of FPs have been proposed  consisting of linearizing bilinear terms by introducing additional O(nm) continuous variables and big-M constraints. 
\ML{However, MILP solvers have difficulties} -- 
in particular when the number of ratio terms increases -- due to the significant lifting of the decision and constraint spaces and the looseness of the continuous relaxation induced by the big-M constraints. 
%the large number of added variables and constraints. 
Within this family, the MILP reformulation proposed in \cite{BGP16} and based on the binary expansion of the integer-valued expressions in the ratio terms \ML{contains many less bilinear terms and hence less linearization variables and constraints.} 
This formulation scales much better but can generate weak continuous relaxations, which hurts the convergence of the branch-and-bound algorithm. 
Mixed-integer second-order cone reformulations have also been proposed (e.g., \cite{33}). They are based on the submodularity concept and the derivation of extended polymatroid cuts \cite{5} to obtain tighter continuous relaxations. 
%to exploit the submodular structure and strengthen the formulations. Both the aforementioned conic quadratic reformulations result in stronger convex relaxations than the standard MILP counterparts, as the latter requires linearization of bilinear terms with big-M constraints. Furthermore, thanks to recent advances in commercial MICQP optimization softwares such as CPLEX [21] and Gurobi [18], small- and medium-sized FPs can be solved efficiently.
However, it is reported \cite{MGP} that state-of-the-art optimization solvers still struggle to solve moderate-size mixed-integer conic problems and that their performance degrades quickly as size increases.
Building upon the links between MILP and mixed-integer conic reformulations, \citeauthor{MGP}  \cite{MGP} 
derive tight convex 
%considered the and have integrated some of their respective features to obtain 
continuous relaxations while limiting the lifting of the decision and constraint spaces.  
%to develop formulations for generally structured fractional 0–1 programs that perform well for all instance sizes, with special focus on large instances where current methods fail. Specifically, our contribution is threefold:
%(ii) We demonstrate how to integrate MICQP and MILP formulations to obtain novel formulations that simultaneously have strong convex relaxations, and a limited number of variables and constraints.

The above methods, while providing very valuable insights, can not be directly used to the problem tackled here as the latter differs from the fractional 0-1 problem along several dimensions.
First, the denominator and the numerator of each ratio term are %nonlinear and 
nonconvex functions and each involve binary as well as general integer decision variables.  
Second, the objective function is the sum of a very large number (\ML{i.e., several thousands}) of ratio terms. 
Third, the integer variables are multiplied by fractional parameters which prohibits the use of the MILP method (see \cite{BGP16}).
Fourth, the nonlinear terms are not simply bilinear. The numerator of each ratio term includes exponential terms and higher-degree polynomial terms while each denominator includes polynomial and factorial terms.
%We refer the reader to \cite{STANCU} for a review of the fractional programming discipline.  

%%%%%%%%%%%%%%%%%%%%%%%%%%%%%%%%%%%%%%%%%%%
%%%%%%%%%%%%%%%%%%%%%%%%%%%%%%%%%%%%%%%%%%%
%%%%%%%%%%%%%%%%%%%%%%%%%%%%%%%%%%%%%%%%%%%
%%%%%%%%%%%%%%%%%%%%%%%%%%%%%%%%%%%%%%%%%%%
\vspace{-0.15in}
\section{Drone Network Design Problem for Opioid Overdose Response} 
\label{sec_model}
\vspace{-0.1in}
\subsection{Problem Description and Notations}\label{subsec_description}
\vspace{-0.085in}
The problem studied in this paper is called the Drone Network Design Problem for Opioid Overdose Response (DNDP). 
Consider an EMS provider that seeks to design a drone network to deliver naloxone, i.e., an opioid antagonist, to an opioid overdose incident with the objective to minimize the response time and thereby to maximize the chance of survival of the overdosed patient.  
Given a set of candidate locations (e.g., fire, police, EMS stations), some of them are selected to be set up as DBs where the available medical drones can be deployed. 
The DNDP model is a bilocation-allocation problem that simultaneously determines the locations of the DBs, the capacity of those (i.e., number of drones at each DB), and the response policy determined by the assignment of drones to OTRs in order to minimize the response time. The response time is defined as the sum of the queueing delay time experienced if drones are busy when requested for service and the drone flight time from a DB to an OTR location.
The network is capacitated to account for  the severely limited availability of medical resources. Considering the opioid epidemic, Buckland et al. \cite{BucklandDesignConsider} stress that the amount of medical resources (drones, ambulances, paramedics) is “{\it substantially smaller than the number of patients who require immediate care}”.

%%%%%%%%%%%%%%%%%%%%%%%%%%%

The following notational set is used in the formulation of the model (see Appendix \ref{sec:notations}). 
Let $I$ be the set of OTR locations and $J$ be the set of potential DB locations. 
Let $A_j$ be a vector of parameters $(a_j, b_j, c_j)$ representing the coordinates of DB $j$ in the earth-centered, earth-fixed coordinate system while $A_i$ is a vector of parameters $(a_i, b_i, c_i)$ representing the location $i$ of an OTR. 
% $v$ is the speed of drone, and $r$ is the maximum distance to be traveled by drone (i.e., the drone catchment area as in \cite{boutilier2017optimizing}). 
% $h_{ij}$ represents the fraction of calls coming from $i$ that are served by $j$. 
The parameter $d_{ij} = \Vert A_i - A_j \Vert$ is the euclidean distance between DB $j$ and OTR location $i$.
A number $p$ of drones with speed $v$ can be deployed at up to \ML{$q$ $(q\leq p)$} open DBs across the network.  
Due to battery and autonomy limitations, each drone has a limited coverage defined by its catchment area with radius $r$.  
A drone can only service OTRs that are within its catchment radius \cite{boutilier2017optimizing}, %,chauhan2019maximum}, 
since the drone must have %enough power (i.e., 
battery coverage to return to its base. Accordingly, we define the sets $J_i, i \in I$ (resp., $I_j, j \in J$) which include all the DBs (resp.,  OTR locations) that are within $r$ of OTR location $i$ (resp, DB $j$).
%{\bf We define the locations that can be covered by a certain drone (base) with the Boolean parameter $a_{ij}$ that takes value 1 if location $i$ is at a distance $r/2$ or less from DB $j$ and is otherwise equal to 0.} 
% The set of DB that can serve location $i \in I$ is represented by $J_i$. Similarly, the set of locations that can be served by DB $j \in J$ is $I_j$. 
%The total number of drone bases that can be opened is $q$ ($q \leq p$). 
The binary decision variable $x_j$ is 1 if a DB is set up at location $j$, and is 0 otherwise. The binary variable $y_{ij}$ takes value 1 if an OTR at $i$ is assigned to DB $j$ and value 0 otherwise. 
The general integer decision variable $K_j$ %= \sum_{m = 1}^{p} k_j^{m}$ 
is the number of drones deployed at DB $j$.
By convention, we denote the upper and lower bound of any decision variable $x$ by $\overline{x}$ and $\underline{x}$, respectively.
% For the case of a mobile service center, the total service time is comprised of two main parts: the travel time, and the non-travel time component.\cite{jamil1999stochastic,pulver2018optimizing,boutilier2017optimizing,berman1985optimal}. The travel time component considers the travel time to and from the scene. The non-travel time component accounts for the on-scene service time (e.g. the delivery unloading time)  and the drone reset time, that is the time to recharge and to load new naloxone toolkit for the next service. 

%%%%%%%%%%%%%%%%%%%%%%%%%%%%%%%%%%%%%%%%%%%%%%%
\vspace{-0.1in}
\subsection{DNDP Model: Collection of $M/G/K_j$ 
Queueing Systems with Unknown~Capacity}\label{subsec_queue}
\vspace{-0.075in}
The stochastic nature of the occurrence of overdoses and the resulting uncertainty about the arrival rates of OTRs at DBs along with the uncertain service times can cause delays, requests being queued, and
waiting times until a drone can be dispatched. 
In order to account for those, we model the network as \ML{a collection of $M/G/K_j$ queues} in which each DB $j$ operates %and is modelled 
as an $M/G/K_j$ queueing system where the capacity $K_j$ of DB $j$ is a bounded general integer decision variable representing the number of drones to be deployed at DB $j$.
The occurrence of opioid overdoses at any location $i$ follows a Poisson distribution with arrival rate $\lambda_i$. 
The service time of drones is a random variable with general distribution and with known first and second moments. 
The arrival rate $\eta_j$ of OTRs at DB $j$ can be inferred to be a Poisson process since it is a linear combination of independent Poisson variables (i.e., weighted sum of assigned OTRs; see \eqref{ARRIVAL}). 
The arrival rate of OTRs at DBs is unknown ex-ante and is defined endogenously via the solution of the optimization problem.
The same applies to the drone service time whose expected value is also endogenized. It follows that the arrival process, i.e., demand at DBs and the service times of drones are endogenous sources of uncertainty with  decision-dependent parameter (expected value) uncertainty as coined in \cite{hellemo2018decision}.
If a drone cannot be dispatched on the spot after reception of an OTR, the request is placed in a queue depleted in a first-come-first-served manner. 
After providing service, drones travel back to the DB to be cleaned, recharged, and prepared for the next trip \cite{ornato2020feasibility}.
%For medical emergencies, the response time is defined as the time between the reception of a request at the dispatch centre and the arrival of the first emergency response vehicle at the scene \cite{Bandara:2014,DEMAIO,INGO}.

%%%%%%%%%%%%%%%%%%%%%%%%%%%%%%%%%%%%%% 
%%%%%%%%%%%%%%%%%%%%%%%%%%%%%%%%%%%%%%
%%%%%%%%%%%%%%%%%%%%%%%%%%%%%%%%%%%%%%
%%%%%%%%%%%%%%%%%%%%%%%%%%%%%%%%%%%%%%
We now derive steady-state closed-form expressions for the queueing metrics - response time, queueing delay, service time -- needed to formulate the DNDP~problem. 
We use the notations $S_{j}$, $R_i$, and $Q_j$ to represent the random variables respectively denoting the total service time at DB $j$, the response time for an overdose at location $i$, and the queueing delay at DB~$j$. 

Different from immobile servers, the travel time to and from the scene are included in the service time for mobile servers
\cite{berman2007multiple}. 
The service time for any OTR at $i$ serviced by %a drone placed at 
DB $j$ \cite{berman1985optimal} is
$S_{ij} = \beta \frac{d_{ij}}{v} + \alpha_i + \epsilon_i$,
where $\beta$ is a constant that allows for different travel speeds,
%to and from the scene, 
and $\alpha_i$ and $\epsilon_i$ are independent and identically distributed (i.i.d) random variables representing the on-scene service time and the drone's reset time (to be charged and prepared for the next demand), respectively. 
Let $\xi_i = \alpha_i + \epsilon_i$,
with expected value $\mathbb{E}[\xi_i]$. 
The expected value $\mathbb{E}[S_{ij}]$  of the service time conditional to DB $j$ responding to an OTR at $i$ is:
\vspace{-0.05in}
\begin{align} \label{e_s_ij}
\mathbb{E}[S_{ij}] & = \beta \frac{d_{ij}}{v} + \mathbb{E}[\xi_i]  \  .
\end{align} 
The expected value of the total response time for $j$ is 
the sum of the expected service times for all the OTRs serviced (i.e., $y_{ij}=1$) by the drones stationed at $j$ and depends on the dispatching variables $y_{ij}$
\vspace{-0.05in}
\begin{align} \label{e_s_j} 
\mathbb{E}[S_j] & = \frac{\sum_{i \in I_j} \lambda_i y_{ij}\mathbb{E}[S_{ij}]}{\sum_{i \in I_j} \lambda_i y_{ij}}  
\end{align}
while the second moment of the total service time at DB $j$ is:

\begin{align} \label{e_s_j_sqr}
\mathbb{E}[S_{j}^2] & = \frac{\sum_{i \in I_j} \lambda_i y_{ij}\mathbb{E}[S_{ij}^2]}{\sum_{i \in I_j} \lambda_i y_{ij}}  \ .
\end{align}

%%%%%%%%%%%%%%%%%%%%%%%%%%%%%%%%%%%%
The service time $S_j$ follows a general unspecified distribution with known first and second moments and the opioid overdoses occur at each location $i$ according to a Poisson process with rate $\lambda_i$.
Accordingly, each DB $j$ is modelled as an $M/G/K_j$ queueing system in which a variable and upper bounded number $K_j$ of drones can be stationed and whose expected queueing delay can be approximated 
\cite{nozaki1978approxi, ROSS2014IntroToProb} as
\vspace{-0.06in}
\begin{align} \label{e_q_j}
\mathbb{E}[Q_j] \approx \frac{\eta_j^{K_j} \mathbb{E}[S_j^2] \mathbb{E}[S_j]^{K_j-1}}{2(K_j-1)! (K_j-\eta_j \mathbb{E}[S_j])^2 \big[ \sum_{n = 0}^{K_j-1} \frac{(\eta_j \mathbb{E}[S_j])^n}{n!} + \frac{(\eta_j \mathbb{E}[S_j])^{K_j}}{(K_j - \eta_j \mathbb{E}[S_j]}\big]} 
\end{align}
which represents the expected waiting time until a drone is ready to be dispatched after reception of an OTR. 
The arrival rate of OTRs at DB $j$ is given by
\vspace{-0.06in}
 \begin{equation}
\label{ARRIVAL}
\eta_j = \sum_{i \in I_j} \lambda_i y_{ij} \ , \ j \in J 
\vspace{-0.105in}
\end{equation} 
which shows that the arrival rate is unknown ex-ante and depends on the assignment decisions $y_{ij}$, thereby highlighting the decision-dependent parameter uncertainty of the OTR arrivals at DBs. 

The expected response time \ML{in steady-state} for an OTR at location $i$ is:
\vspace{-0.073in}
\begin{align} \label{e_r_i}
\mathbb{E}[R_i] = \sum_{j \in J_i} \left(\mathbb{E}[Q_j] + \frac{d_{ij}}{v}\right)y_{ij} \ .
\vspace{-0.107in}
\end{align}
\vspace{-0.245in}
\MLB{
\begin{rem}
Note that the response and service time differ in the healthcare context. For example, Bandara et al. \cite{Bandara:2014} define the response time as “{\it the time between the receipt of a call at the dispatch centre and the arrival of the first emergency response vehicle at the scene}" while their definition of the service time is “{\it the time required for an ambulance to return back to its original station after leaving the station to attend the call, which includes the transportation time of the patient to the hospital if needed}”.
\end{rem}
}
\vspace{-0.105in}
\noindent
The metric used in our model is the average (over all overdoses) of the expected response time.~We~denote it by $\bar{R}$ and refer to it as the average response time. The proof of Proposition \ref{thm_avg_resp} is in Appendix \ref{AVE-PR1}.
\vspace{-0.05in}
\begin{prop}\label{thm_avg_resp}
The functional form of the average response time is a fractional expression with nonlinear numerator and denominator:
\vspace{-0.06in}
\begin{equation}
\bar{R} \approx \sum_{i \in I} \sum_{j \in J_i} \left[ \frac{\eta_j^{K_j} \mathbb{E}[S_j^2] \mathbb{E}[S_j]^{K_j-1}}{2(K_j-1)! (K_j-\eta_j \mathbb{E}[S_j])^2 \big[ \sum_{n = 0}^{K_j-1} \frac{(\eta_j \mathbb{E}[S_j])^n}{n!} + \frac{(\eta_j \mathbb{E}[S_j])^{K_j}}{(K_j - 1)!(K_j - \eta_j \mathbb{E}[S_j]}\big]} + \frac{d_{ij}}{v}\right]\frac{\lambda_i y_{ij}}{\sum_{l \in I} \lambda_l} \ .
\end{equation}
\end{prop}

%%%%%%%%%%%%%%%%%%%%%%%%%%%%%%%%%%%%%%%
%%%%%%%%%%%%%%%%%%%%%%%%%%%%%%%%%%%%%%%
%%%%%%%%%%%%%%%%%%%%%%%%%%%%%%%%%%%%%%%
%%%%%%%%%%%%%%%%%%%%%%%%%%%%%%%%%%%%%%%
\vspace{-0.2in}
\subsection{DNDP Base Model: Fractional Nonlinear Integer Programming Problem}\label{sec_formu}
\vspace{-0.1in}
We present now the base formulation of the DNDP problem, which belongs to the family of nonlinear fractional integer problems, and analyze its complexity. 
Before presenting its formulation, we recall the most distinctive features of the proposed DNDP model.
% which considers jointly the location of DBs and drones and the dispatching policy of the latter. 
First, the minimization of the average response time is a survival objective function as it contributes to increasing the survival chance of the victims of opioid overdoses. 
The chance of survival is a monotone decreasing function of the response time (see, e.g., \cite{Bandara:2014,DEMAIO}). By minimizing response time and rewarding quick care, our model effectively maximizes  
the chance of survival of overdosed victims, which is viewed in the health care literature 
as a survival optimization model. Second, the queueing-optimization model accounts for congestion and decision-dependent uncertainty as the service and arrival rates characterizing the Poisson arrival process and service times of OTRs at DBs are determined endogenously. Third, the capacity -- number of drones positioned of each $M/G/K_j$ DB $j$ is not fixed ex-ante, but is a bounded decision variable.

The fractional nonlinear integer base formulation $\mathbf{B-IFP}$ of problem DNDP is:
\begin{subequations}\label{M-BF}
\begin{align}
\mathbf{B-IFP:}
\min  & \; \frac{1}{\sum_{l \in I}\lambda_l} 
\Bigg[
\sum_{i \in I} \sum_{j \in J_i} \frac{\lambda_i d_{ij}y_{ij}}{v} +
\notag \\
&\hspace{-2.6cm} \sum_{i \in I} \sum_{j \in J_i}   \frac{\lambda_i y_{ij}\sum_{l \in I_j} (\lambda_l y_{lj}\mathbb{E}[S_{lj}^2]) 
(\sum_{l \in I_j} \lambda_l y_{lj}
\mathbb{E}[S_{lj}])^{K_j-1}} {2(K_j-1)! (K_j- \sum_{l \in I_j} \lambda_l y_{lj}\mathbb{E}[S_{lj}]  )^2 \big[ \sum_{n = 0}^{K_j-1} \frac{\sum_{l \in I_j} \lambda_l y_{lj} \mathbb{E}[S_{lj}])^n}{n!} 
+ \frac{(\sum_{l \in I_j} \lambda_l y_{lj}\mathbb{E}[S_{lj}])^{K_j}}
{(K_j - 1)!(K_j - \sum_{l \in I_j} \lambda_l y_{lj} \mathbb{E}[S_{lj}])}\big]} \Bigg] \label{D1_obj} \\
s.to \ \ &	\sum_{i \in I_j} \lambda_i y_{ij}   < K_j \frac{\sum_{i \in I_j} \lambda_i y_{ij}}{ \sum_{i \in I_j} \lambda_i y_{ij}\mathbb{E}[S_{ij}]}, \ \ j \in J \label{steady-state}\\
& \sum_{j \in J_i} y_{ij} = 1, \ \  i  \in I \label{assignment}\\
% &	y_{ij} \le a_{ij} x_j \ \  i \in I, j \in J \label{eq_lim}\\
&	y_{ij} \le x_j \ \  i \in I, j \in J_i \label{eq_lim}\\
&	\sum_{j \in J} x_j \le q \label{eq_q}\\
&x_j \le K_j \le M x_j, \ j \in J \label{open_drone}\\
& \sum_{j \in J} K_j = p \label{eq_p}\\
&x_j, y_{ij} \in \{0,1\}, \ \ i \in I_j, j \in J \label{binary} \\
& K_j \in \mathbb{Z}_+, \ \  j \in J \label{k_j_integer} 
\end{align}
\end{subequations}
\vspace{-0.285in}

The objective function \eqref{D1_obj} minimizes the average response time of the network, a key feature given the time-critical nature of opioid overdoses. 
Given that the queueing formulas used to represent the response time 
%in an $M/G/K_j$ queue 
assume that the system is in steady-state, we introduce the nonlinear constraints \eqref{steady-state} to ensure the 
%steady-state condition and 
stability of the queueing system \cite{medhi2002stochastic} as they prevent the arrival rate at any DB from being equal or larger than the service rate.
Constraint \eqref{assignment} makes sure that  each OTR is serviced by exactly one drone while \eqref{eq_lim} requires that an OTR can only be  assigned to an open DB within the drone catchment area. 
Constraint \eqref{eq_q} limits from above the number of open DBs with the parameter $q$. %while \eqref{eq_p} defines the number $p$ of drones that can be used.
Constraint \eqref{open_drone} ensures that drones can only be placed at an opened DB and that an opened DB must have at least one drone. The parameter $M$ defines the maximum number
%(upper bound) 
of drones that can be stationed at any DB $j$.
%: $K_j\leq M,j\in J$. 
As suggested in \cite{MOSHREF} and as implemented in \cite{BAUER,pulver2018optimizing, gao2020dynamic} for medical drone networks, we set $M$ equal to two as drones are scarce resources and are typically spread to expand the coverage of the network. Additionally, our preliminary tests based on real-life data show that giving the possibility to place a third drone at a DB does not reduce the response times.
Constraint \eqref{eq_p} ensures that all available drones are deployed. 
% Constraint \eqref{num_drone_lim} ensures that drones can only be deployed on an opened base. M is a large enough quantity that is no less than the maximal number of drones can be deployed in a base.  
The binary and general integer nature of the decision variables are enforced by  \eqref{binary} %, \eqref{y_binary}, 
and \eqref{k_j_integer} with $\mathbb{Z}_+$ denoting the set of nonnegative integer numbers.
Remark \ref{PROP1} summarizes several key features of the model.
% that follow from the above discussion.
\vspace{-0.0175in}
\begin{rem} \label{PROP1}
Problem $\mathbf{B-IFP}$ is a \ML{nonconvex} %NP-hard 
optimization problem in which:
%\newline 
(i) The objective function is neither convex, nor concave: the denominator and the numerator in each ratio term of the objective function are nonconvex functions; 
(ii) The ratio terms can involve division by 0 and be indeterminate;
(iii) Any ratio term related to a 
location where no DB is set up in undefined;
(iv) There is a mix of binary and bounded general integer variables;
% \newline
% \MLR{(v) The nonlinear strict inequality constraints \eqref{steady-state} include fractional and polynomial terms, and define a nonconvex feasible area.} \WM{This seems trivial since by just canceling the term we can get a linear constraint}
(v) The continuous relaxation of $\mathbf{B-IFP}$ is a nonconvex~problem. 
%\newline
%(i) is a non-linear and non-convex optimization model with both binary and general integer variables. 
%(2) Its objective function is indeterminate.
%non-convex.
\end{rem}
\vspace{-0.05in}
The validity of the above statements is demonstrated in Appendix \ref{REM1}.
We also propose in Appendix \ref{APP-ILLU} an illustration for a small network and present the formulation of the objective function in $\mathbf{B-IFP}$.  

%%%%%%%%%%%%%%%%%%%%%%%%%%%%%%%%%%%%%%%%%%%%%%%%%
%%%%%%%%%%%%%%%%%%%%%%%%%%%%%%%%%%%%%%%%%%%%%%%%%
\vspace{-0.11in}
\section{Reformulation Framework} \label{sec_reformulation}
\vspace{-0.125in}
Remark \ref{PROP1} highlights that the base formulation $\mathbf{B-IFP}$ is a fractional nonlinear integer problem 
which is extremely difficult to solve numerically even for small problem instances. 
%of small size. 
%problem $\mathbf{R-BFP}$ is a very complex optimization problem. First, it involves a large number of integer variables, including general integer ones. Second, its continuous relaxation is nonconvex. Third, some of the ratio terms in the objective function will be indeterminate if the corresponding steady-state constraint does not hold strictly.
We derive an equivalent and computationally tractable MILP reformulation. 
%that takes the form of an MILP problem.
We first demonstrate in Theorem \ref{T2} 
%and Proposition \ref{PROPO2} 
that an MILP reformulation can be derived when the number of servers is two or less as relevant 
 to the medical drone network problem (see \cite{BAUER,MOSHREF,pulver2018optimizing}) 
 % -- and simplifies the notation in the proof -- 
 before generalizing this result 
 (Theorem \ref{TH-GEN}) to any $M/G/K_j$ queueing system with variable number of servers $K_j$ at each $j$.
 %\ML{and deriving tightening valid inequalities for nested multilinear polynomials.}
By convention, a superscripted index inside parentheses $(m)$ refers to the  $m^{th}$ component of a vector (e.g., $\gamma^{(m)}$ refers to the $m^{th}$ component~of~$\gamma$).
 
%%%%%%%%%%%%%%%%%%%%%%%%%%%%%%%%%%%%%%%%%%%%%%
%%%%%%%%%%%%%%%%%%%%%%%%%%%%%%%%%%%%%%%%%%%%%%
\vspace{-0.035in}
\ML{\subsection{MILP Reformulations} \label{MINLP-REF}}
\vspace{-0.075in}
The derivation of an MILP reformulation is relatively complex and we split it into two main steps (see Proposition \ref{T1} \MLB{and Theorem \ref{T2})} % as well as Proposition \ref{PROPO2})} 
to ease the exposition. 

Proposition \ref{T1} proposes an equivalent reformulation 
taking the form of a fractional nonlinear binary problem {\bf R-BFP} with nonconvex continuous relaxation.   
\MLB{The proof is given in Appendix \ref{ATH40}.}

\vspace{-0.1in}
\begin{prop} \label{T1}
Let $\gamma_{j}^{\WMR{(m)}} \in \{0, 1\},  j \in J, m=1,\ldots,M$.  
%\newline
The fractional nonlinear binary~problem 
\vspace{-0.075in}
%$\mathbf{R-BFP}$ 
\begin{subequations}\label{F-R-BFP}
\begin{align}
\mathbf{R-BFP}: & \; \min	    
\sum_{i \in I} \sum_{j \in J_i} \frac{y_{ij}d_{ij} \lambda_i}{v \sum_{l \in I} \lambda_l} \; + \; \sum_{i \in I} \sum_{j \in J_i} \sum_{m=1}^{M} \ \frac{y_{ij}\lambda_i}{\sum_{l \in I} \lambda_l} 
\label{OBJ2} \\ 
& \hspace{-0.8in} 
\left[ \frac{\gamma_j^{\WMR{(m)}} \sum_{l \in I_j} \lambda_l y_{lj} \mathbb{E}[S_{lj}^2] (\sum_{l \in I_j} \lambda_l y_{lj} \mathbb{E}[S_{lj}])^{m-1}}{2(m-1)! (m-\sum_{l \in I_j} \lambda_l y_{lj} \mathbb{E}[S_{lj}])^2 \big[ \sum_{n = 0}^{m-1} \frac{(\sum_{l \in I_j} \lambda_l y_{lj} \mathbb{E}[S_{lj}])^n}{n!} + \frac{(\sum_{l \in I_j} \lambda_l y_{lj} \mathbb{E}[S_{lj}])^{m}}{(m - 1)!(m - \sum_{l \in I_j} \lambda_l y_{lj} \mathbb{E}[S_{lj}]}\big]} \right] \notag\\ 
\text{s.to} \; & \eqref{assignment}- \eqref{eq_q} ; \eqref{binary} \notag \\ %-\eqref{y_binary}
& \sum_{i \in I_j} \lambda_i y_{ij}\mathbb{E}[S_{ij}]   \le \sum_{m=1}^{M} m  \gamma_j^{\WMR{(m)}} - \epsilon , \ \ j \in J \label{steady-state-2}\\
& x_j \le \sum\limits_{m=1}^{M}  m  \gamma_j^{\WMR{(m)}} \le M x_j, \ j \in J \label{open_drone-2}\\
& \sum\limits_{j \in J}\sum\limits_{m=1}^{M} m\gamma_j^{\WMR{(m)}} = p \label{eq_p-2}\\
&\sum\limits_{m=1}^{M} \gamma_j^{\WMR{(m)}} \leq 1 \ , \ j \in J  \label{NEW1} \\
& \gamma_j^{\WMR{(m)}} \in \{0,1\}, \ \ j \in J, m =1, \ldots, M \label{binary2} 
\end{align}
\end{subequations}
is equivalent to the fractional nonlinear integer problem  $\mathbf{B-IFP}$. 
\vspace{-0.105in}
\end{prop}	
The following comments are worth noting.
A first difference with $\mathbf{B-IFP}$ is that $\mathbf{R-BFP}$ has only binary decision variables and does not include any general integer variables $K_j$. 
The second difference and advantage of $\mathbf{R-BFP}$ over $\mathbf{B-IFP}$ is that the polynomial terms in the reformulation $\mathbf{R-BFP}$ are of lower degree than those in $\mathbf{B-IFP}$.
Third, in $\mathbf{R-BFP}$, there is no decision variable $K_j$ appearing in the upper limits of the summation operations, such as 
$ \sum_{n = 0}^{K_j-1} (\sum_{l \in I_j} \lambda_l y_{lj} \mathbb{E}[S_{lj}])^n$ in $\mathbf{B-IFP}$.
Fourth, there is no exponential term of form $y_{ij}^{K_j}$ in $\mathbf{R-BFP}$.
%there is no (general integer) decision variable appearing in the exponent of some expression, nor in the factorial terms. 
%\WM{For instance, $ \sum_{l \in I} \lambda_l y_{lj} \mathbb{E}[S_{lj}^2] (\sum_{l \in I} \lambda_l y_{lj}\mathbb{E}[S_{lj}])^{K_{j}-1}$ and $(K_j - 1)!$ existed in the numerator and denominator of $DN$ are replaced by $\sum_{l \in I} \lambda_l y_{lj} \mathbb{E}[S_{lj}^2] (\sum_{l \in I} \lambda_l y_{lj}\mathbb{E}[S_{lj}])^{m-1}$ and $(m - 1)!$, respectively, in \textbf{R-BFP} , where $m$ is a constant. 
Fifth, the factorial terms $(K_j-1)!$ in $\mathbf{B-IFP}$ which, besides being nonlinear, can also be undefined if $K_j=0$, are no longer present in $\mathbf{R-BFP}$ in which they are replaced by the fixed parameters $(m-1)!$. The undefined issue is resolved since $m$ is at least equal to 1. The reformulation $\mathbf{R-BFP}$ is valid for any value assigned to $M$.

The constraints 
\eqref{assignment}-\eqref{eq_q}, \eqref{binary}, and %-\eqref{y_binary}, and
\eqref{steady-state-2}-\eqref{binary2} representing the feasible area of problem \textbf{R-BFP} define collectively a linear binary feasible set which, to ease the notation, is thereafter referred to as:
%with the notation $\mathcal{B}$:
%To further simplify the notation, we define 
\vspace{-0.025in}
\begin{equation}
\label{FEASIBLESET2}
\mathcal{B} = \left\{(x,y,\gamma)  \in \{0,1\}^{|J|+\sum_{i\in I} |J_i|+M|J|} : \eqref{assignment}-\eqref{eq_q}; \eqref{binary};\eqref{steady-state-2}-\eqref{binary2}\right\} \ .
\vspace{-0.025in}
\end{equation}
We give in Appendix  \ref{APP-ILLU} the objective function of problem $\mathbf{R-BFP}$ for a small network.
Table \ref{T01} in Appendix \ref{SIZE} specifies the number of constraints and integer decision variables of each type in the base formulation $\mathbf{B-IFP}$ and its reformulation $\mathbf{R-BFP}$. 

%%%%%%%%%%%%%%%%%%%%%%%%%%%%%%%%%%%%% 
%%%%%%%%%%%%%%%%%%%%%%%%%%%%%%%%%%%%% 
%%%%%%%%%%%%%%%%%%%%%%%%%%%%%%%%%%%%% 
%%%%%%%%%%%%%%%%%%%%%%%%%%%%%%%%%%%%%
%\vspace{-0.015in}
%\ML{\subsection{MILP Reformulations} \label{MILP-REF}}
%\vspace{-0.05in}
Although simpler than $\mathbf{B-IFP}$, $\mathbf{R-BFP}$ is still a challenging optimization problem as it involves polynomial and fractional terms %of binary variables 
and its continuous relaxation is nonconvex. 
We shall now demonstrate in Theorem \ref{T2} that the MILP problem $\mathbf{R-MILP}$ is equivalent to $\mathbf{R-BFP}$ and $\mathbf{B-IFP}$. 
The proof is split into two main steps that starts with the moving of the fractional terms from the objective function into the constraint set and continues with the linearization of the polynomial terms. 
To shorten the notations in the proof, we replace $\mathbf{E}[S_{lj}]$ and $\mathbf{E}[S_{lj}^{2}]$  by $\tilde{S}_{lj}$ and $\tilde{S}^{2}_{lj}$, respectively.
%First, with a slight abuse of notation to save space, let $S_{lj} = \mathbb{E}[S_{lj}]$ and $S_{lj}^{2} = \mathbb{E}[S_{lj}^{2}]$.
We first consider the medical drone network design problem DNDP in which $M=2$ (i.e., up to two servers per DB).  
Theorem \ref{TH-GEN} will later show that an MILP reformulation can be derived for any value of the parameter $M$, or, in other words, for any $M/G/K$ queueing system in which the capacity of each system is unknown and variable.

We recall two well-known linearization methods 
%\citeauthor{FORTET} and \citeauthor{mccormick1976computability} 
for bilinear terms used in the proof of Theorem \ref{T2}:
\vspace{-0.1025in}
\begin{itemize}
\item Product of binary variables \cite{FORTET} : \\
%\begin{defn}
%Let $y_1,y_2 \in \{0,1\}^2$.
The polynomial set 
$\{(y_1,y_2,z) \in \{0,1\}^2 \times [0,1]: z = y_1 y_2 \}$ 
is equivalent to the MILP set: % defined by:
\vspace{-0.05in}
\begin{equation}
\label{FOR}
\mathcal{F}_y^z=\Big\{ (y_1,y_2,z): 
z \geq 0, \;
z \geq y_1+y_2 - 1, \; 
z \leq  y_i, \;  i=1,2   \Big\} \ .
\vspace{-0.1205in}
\end{equation}
\item Product of a bounded continuous variable by a binary one  \cite{mccormick1976computability}: \\
%\begin{defn}
%Let $y_1,y_2 \in \{0,1\}^2$.
Let $0 \leq \underline{x} < \bar{x}$.
The polynomial set 
$\{(y,x,z) \in \{0,1\} \times [\underline{x},\bar{x}] \times [0,\bar{x}]: z =yx \}$ 
can be equivalently represented with the MILP set: \vspace{-0.075in}
\begin{equation}
\label{R-FOR}
\mathcal{M}_{yx}^z=\Big\{ (y,x,z): 
z \geq \underline{x} y, \;
z \geq \bar{x}(y-1)+x, \; 
z \leq  \bar{x}y, \: z \leq \underline{x}(y-1)+x   \Big\} \ .
\end{equation}

\end{itemize}
\vspace{-0.25in}
\begin{thm} \label{T2}
Define the index sets 
$D_j = \{(l,t): l,t \in I_j, l <t \}, j\in J$. 
Let 
$z_{jlt} \in [0,1]$,
$\mu^{(m)}_{ij} \in [0,\bar{U}_j^{(m)}]$, 
$\tau_{jlt} \in [0,\bar{U}^{(2)}_j]$, and
$\omega^{(m)}_{ij} \in [0,\bar{\mu}_{ij}^{(m)}]$
$i,l,t \in I_j, (l,t) \in D_j, j\in J, m=1,\ldots,M$
be continuous auxiliary variables used for the linearizization of bilinear terms.
The MILP problem $\mathbf{R-MILP}$ 

\begin{subequations}
\label{RDN3}
\begin{align}
\min & \ \sum_{i \in I} \sum_{j \in J_i} \frac{y_{ij}d_{ij} \lambda_i}{v \sum_{l \in I} \lambda_l} 
+  \sum_{i \in I} \sum_{j \in J_i} 
\Bigg[ \frac{\omega_{ij}^{\WMR{(1)}}}{2} + \frac{\omega_{ij}^{\WMR{(2)}}}{2} \Bigg] 
\frac{\lambda_i}{\sum_{l \in I} \lambda_l} & \label{obj_lin} \\ 
\text{s.to} \ & (x,y,\gamma) \in \mathcal{B} & \label{COM1} \\ 
& U_j^{(1)} = \sum_{l \in I_j}\lambda_l \mu^{(1)}_{lj} \tilde{S}_{lj} + \sum_{l \in I_j}\lambda_l y_{lj} \tilde{S}_{lj}^{2} & j \in J  \label{U}  \\
&4U^{(2)}_{j} = 
\sum_{l \in I_j} \lambda_l^2 \tilde{S}_{lj} (\mu^{(2)}_{lj}\tilde{S}_{lj} + y_{lj}  \tilde{S}_{lj}^2) 
+ 2 \sum_{(l,t) \in D_j} \lambda_l \lambda_t \tilde{S}_{tj} (z_{jlt} \tilde{S}_{lj}^{2} +  \tau_{jlt}^{}\tilde{S}_{lj}) & j \in J \label{V} \\
&(y_{lj},y_{tj},z_{jlt}) \in \mathcal{F}_y^z & (l,t) \in D_j, j \in J \label{COM2} \\ 
&(y_{lj},U^{(m)}_j,\mu^{(m)}_{lj}) \in  \mathcal{M}_{yU}^{\mu} & \hspace{-2cm} l \in I_j, j \in J, {m}=1,\ldots,M \label{COM3}\\ 
&(z_{jlt},U^{(2)}_j,\tau_{jlt}) \in \mathcal{M}_{zU}^{\tau}& (l,t) \in D_j, j \in J  \label{COM4} \\ 
%%% OLD R-MILP Formulation 
%&(\gamma^{\WMR{(m)}}_j,\mu^{\WMR{(m)}}_{ij},\omega^{\WMR{(m)}}_{ij}) \in  \mathcal{M}_{\gamma\mu}^{\omega} & \hspace{-2cm}i \in I, j \in J_i, \WMR{m}=1,\ldots,M \notag \\ 
&\MLB{(\gamma^{(1)}_j,\mu^{(1)}_{ij}, \omega^{(1)}_{ij}) \in  \mathcal{M}_{\ \gamma\mu}^{' \omega}}& \MLB{i \in I, j \in J_i} \label{COMPACT}
\end{align}
\end{subequations}
is equivalent to the nonlinear integer problems $\mathbf{B-IFP}$ and 
$\mathbf{R-BFP}$ for $M$=2.
\end{thm}
%%%%%%%%%%%%%%%%%%%%%%%%%%%%%%%%%%%%%%%
\vspace{-0.1in}
\MLB{
The proof of Theorem \ref{T2} is given in Appendix \ref{ATH41} and is decomposed into two main parts. The linearization part first demonstrates that the nonlinear problems {\bf B-IFP} and {\bf R-BFP} are MILP-representable. The compaction part shows that the number of linearization constraints and variables can be significantly reduced.} 
\begin{comment}
This result is attained by first showing that 
$\omega^{(2)}_{ij} = 
\mu^{(2)}_{ij} =  \mu^{(2)}_{ij} \cdot \gamma^{(2)}_j, i \in I, j \in J_i$, which is in turn used to prove that the following two sets of linear inequalities are equivalent:
\[
(\gamma^{(1)}_j,\mu^{(1)}_{ij},\omega^{(1)}_{ij}) \in  \mathcal{M}_{\gamma^{(1)}\mu^{(1)}}^{\omega^{(1)}}, \ i \in I, j \in J_i 
\quad \Leftrightarrow \quad
(\gamma^{(m)}_j,\mu^{(m)}_{ij},\omega^{(m)}_{ij}) \in  \mathcal{M}_{\gamma\mu}^{\omega}, \ i \in I, j \in J_i, m=1,\ldots,M \ .
\]
In Section \ref{sub_sec_compute_e}, we assess the computational benefits of the compaction approach by comparing the results with the compact MILP reformulation {\bf R-MILP} and the larger-size MILP reformulation {\bf L-MILP} 
\vspace{-0.1in}
\begin{equation}
\label{LARGE}  
{\bf L-MILP}: \max \eqref{obj_lin}: \eqref{COM1}-\eqref{COM4} ; \eqref{OLD-R-MILP}
\vspace{-0.1in}
\end{equation}
in which the set of constraints 
\vspace{-0.075in}
\begin{equation}
\label{OLD-R-MILP}
(\gamma^{(m)}_j,\mu^{(m)}_{ij},\omega^{(m)}_{ij}) \in  \mathcal{M}_{\gamma\mu}^{\omega} \; , \; i \in I, j \in J_i, m=1,\ldots,M 
\vspace{-0.075in}
\end{equation}
replaces \eqref{COMPACT} in {\bf R-MILP}.
Observe that the MILP reformulation {\bf L-MILP} contains $\sum_{i\in I} |J_i|$ 
more decision variables and $4 \ \sum_{i\in I} |J_i|$ more linearization constraints than {\bf R-MILP}
(i.e., the superscript of the constraint \eqref{OLD-R-MILP} is changed from $(m)$ to $(1)$ in \eqref{COMPACT}).
\end{comment}
\MLB{
This result is attained by first showing that 
$\omega^{(2)}_{ij} = 
\mu^{(2)}_{ij} =  \mu^{(2)}_{ij} \cdot \gamma^{(2)}_j, i \in I, j \in J_i$, which is in turn used to prove that the two sets of linear inequalities 
$(\gamma^{(1)}_j,\mu^{(1)}_{ij},\omega^{(1)}_{ij}) \in  \mathcal{M}_{\ \gamma\mu}^{' \omega}, \ i \in I, j \in J_i 
$ \eqref{COMPACT} and
\vspace{-0.075in}
\begin{equation}
\label{OLD-R-MILP}
(\gamma^{(m)}_j,\mu^{(m)}_{ij},\omega^{(m)}_{ij}) \in  \mathcal{M}_{\gamma\mu}^{\omega} \; , \; i \in I, j \in J_i, m=1,\ldots,M 
\vspace{-0.075in}
\end{equation}
are equivalent 
(i.e., the superscript in \eqref{OLD-R-MILP} is $(m), m=1,\ldots,m$ while it is $(1)$ in \eqref{COMPACT}).
In Section \ref{sub_sec_compute_e}, we assess the computational benefits of the compaction approach by comparing the results with the compact MILP reformulation {\bf R-MILP} and the larger-size MILP reformulation {\bf L-MILP} 
\vspace{-0.1in}
\begin{equation}
\label{LARGE}  
{\bf L-MILP}: \max \eqref{obj_lin}: \eqref{COM1}-\eqref{COM4} ; \eqref{OLD-R-MILP}
\vspace{-0.1in}
\end{equation}
in which \eqref{OLD-R-MILP} replaces  \eqref{COMPACT} in {\bf R-MILP}.
Observe that the MILP reformulation {\bf L-MILP} contains $\sum_{i\in I} |J_i|$ 
more decision variables and $4 \ \sum_{i\in I} |J_i|$ more linearization constraints than {\bf R-MILP}.
}

\noindent
In Appendix \ref{APP-ILLU}, we give the objective formulation of problem $\mathbf{R-MILP}$ for a small network.   

\vspace{-0.075in}
\ML{\subsection{Generalization of Reformulation Framework} \label{GENE-REF}}
\vspace{-0.05in}
We shall now demonstrate in Theorem \ref{TH-GEN} that the linearization approach proposed above for an $M/G/K$ queueing system with variable number of servers (capacity) $K$ limited from above to 2 can be extended to any $M/G/K$ system with finite variable capacity $K$. 
\begin{comment}
\ML{The two following linearization methods  \cite{adams2007} will be used in the proof of Theorem \ref{TH-GEN}:
%\MLR{
%\begin{prop}
%[Linear representation of polynomial terms with products of $n$ binary variables] 
%\label{P1}
\begin{itemize}
\item Let $y \in \{0,1\}^n$.
The polynomial set 
$\{(y,z) \in \{0,1\}^n \times [0,1]: z = \prod_{i=1}^n y_i \}$ 
can be equivalently represented with the MILP set defined by:
\begin{equation}
\label{P1}
\Big\{ (y,z): 
z \geq 0, \;
z \geq \sum_{i=1}^n y_i - n + 1, \; 
z \leq  y_i, \;  i=1,\ldots,n   \Big\} \ .
\end{equation}
%\begin{prop}
%[Linear representation of polynomial terms with one bounded continuous variable and $n$ binary variables] 
%\label{P2}
\item Let $x \in [0,\bar{x}]$, $y \in \{0,1\}^n$.
The polynomial set %of degree $n+1$
$\{(x,y,z) \in [0,\bar{x}] \times \{0,1\}^n \times [0,\bar{x}]: z = x \prod_{i=1}^n y_i  \}$ 
can be equivalently represented with the MILP set defined by:
\begin{equation}
\label{P2}    
\Big\{ (x,y,z): z \geq 0, \;
z \geq x - \bar{x} \Big(n-\sum_{i=1}^n y_i\Big), \; 
z \leq \bar{x} y_i,  \; 
z \leq x,  \; i=1,\ldots,n \Big\} \ .
\end{equation}
%\end{prop}
\end{itemize}
}
\end{comment}
\vspace{0.025in}
\begin{thm}
\label{TH-GEN}
The stochastic network design model of form $\mathbf{R-BFP}$ that minimizes the average response time of a \ML{collection} of $M/G/K_j$ queueing systems with variable and finitely bounded number of servers $K_j$ can be reformulated as
\ML{
\begin{subequations}
		\label{RDN4}
		\begin{align}
			\min & \ \sum_{i \in I} \sum_{j \in J_i} \frac{y_{ij}d_{ij} \lambda_i}{v \sum_{l \in I} \lambda_l} 
			+  \sum_{i \in I} \sum_{j \in J_i} \sum_{m=1}^{M} \ \frac{y_{ij}\lambda_i V_j^{\WMR{(m)}}}{\sum_{l \in I} \lambda_l} \\ 
			\text{s.to} \ 
			& (x,y,\gamma) \in \mathcal{B} & \notag   \\ 
			& \underbrace{V_j^{\WMR{(m)}} 
				(m-1)! (m-\sum_{l \in I_j} \lambda_l y_{lj} \tilde{S}_{lj})^2
				\sum_{n = 0}^{m-1} \frac{(\sum_{l \in I_j} \lambda_l y_{lj} \tilde{S}_{lj})^n}{n!}}_{T1} 
			+ \underbrace{V_j^{\WMR{(m)}}(m - \sum_{l \in I_j} \lambda_l y_{lj} \tilde{S}_{lj}) 
				(\sum_{l \in I_j} \lambda_l y_{lj} \tilde{S}_{lj})^m}_{T2}
			\notag \\
			= \ & \underbrace{1/2 \ \gamma_j^{\WMR{(m)}} \sum_{l \in I_j} \lambda_l y_{lj} \tilde{S}_{lj}^2 (\sum_{l \in I_j} \lambda_l y_{lj} \tilde{S}_{lj} )^{m-1}}_{T3} \; ,  \; j\in J, \ML{m=1,\ldots,M} \label{SUBST1}
		\end{align}
\end{subequations}
}
and is always MILP-representable.
%\ML{as long as the capacity is finitely bounded.}
\end{thm}
\vspace{-0.05in}
The proof of Theorem \ref{TH-GEN} is given in Appendix \ref{ATH4} and shows that the objective function and each constraint \eqref{FT2} can be linearized regardless of the value of $M$.
%%%%%%%%%%%%%%%%%%%%%%%%%%%%%%%%%%%
\MLB{
Note that T1, T2, and T3 in \eqref{SUBST1} are polynomial expressions. 
More precisely, T1 and T2  involve polynomials of degree $(m+2)$ with monomials involving the product of one single continuous bounded variable by up to $(m+1)$ binary variables while T3 includes polynomials of degree $(m+1)$ with monomials defined as products of $(m+1)$ binary variables.
%is $d=min(|I|+1, m+2)$ (see Appendix \ref{ILLU-POLY}) and 
%increases as $M$ ($m=1,\ldots,M$) grows larger.}
%(since $|I|$ is typically much larger than $M$).}
%since the number of OTRs $|I|$ is typically much larger than the  number $M$ of drones that can be stationed at a DB. 
Additionally, the polynomials in T1, T2, and T3 have a {\it nested} structure (see, e.g., \cite{CRAMA,FISCHER}) for which customized valid inequalities could be derived to strengthen the formulation.
}

\vspace{-0.1275in}
\section{Outer Approximation Algorithmic Framework} \label{sec_ALGO}
\vspace{-0.08in}
\MLB{Even though \textbf{R-MILP} is an MILP problem}, its
%and \textbf{R2-MILP} are MILP problems, their}  
solution remains nonetheless a challenge 
owing to \MLB{its} combinatorial nature and size.
%and the many additional constraints due to the %implemented %McCormick linearization approach.
Linearization methods suffer from two possible drawbacks. First, they require the introduction of many decision variables and constraints and second the resulting continuous relaxations can be very loose (i.e., some of the added constraints are big-M ones).
To palliate these issues, we use the concepts of lazy constraints, valid inequalities, and optimality cuts to devise two outer~approximation methods whose efficiency and scalability are demonstrated in Section \ref{sub_sec_compute_e}.
%%%%%%%%%%%%%%%%%%%%%%%%%%%%%%%%%%%%%%%%%%%%%%%%%%%%%%%%%%%%%%%%%%%%%%
%\vspace{-0.075in}
%\subsection{Outer Approximation Branch-and-Cut Algorithm}  \label{sub_sec_B&C}
%\vspace{-0.05in}

While instances of a special variant of model \MLB{\textbf{R-MILP}} %and \textbf{R2-MILP}} 
when the maximum number $M$ of drones at each DB is fixed and equal to 1 (i.e., each open DB is an $M/G/1$ system) can be solved, preliminary experiments reveal that  state-of-the-art solvers are however unable to solve, within 1 hour, the {\it root node} of the continuous relaxation of moderate-sized problem instances for $M\geq 2$.
To overcome this issue, we derive MILP relaxations for problems $\mathbf{R-MILP}$ using {\it lazy constraints} (see, e.g., \cite{kleinert2021,lundell2019}) and embed them in an outer approximation algorithm. The motivation is to alleviate the issue caused by the lifted decision and constraint spaces due to the linearization method. 
Supplementing the outer approximation approach via the derivation of valid inequalities and optimality cuts,
we obtain an outer approximation branch-and-cut algorithm {\tt OA-B\&C}, \MLB{which we describe next}. 

%%%%%%%%%%%%%%%%%%%%%%%%%%%%%%%%%%%%%%%%%%%%%%%%%%%%%%%%%%%%%%%%%%%%%%
\vspace{-0.075in}
\subsection{Outer Approximation with Lazy Constraints}
\label{sub_sec_lazy_c}
\vspace{-0.05in}

A lazy constraint is an integral part of the actual constraint set 
%of an optimization problem 
%(which could lead to invalid solutions in its absence) 
and is unlikely to be binding at the optimum.
%Lazy callback is mainly employed in practical implementations using an optimization solver when the number of constraints in an MIP model is extremely large. 
Instead of incorporating all lazy constraints in the formulation,
%at the root node of the branch-and-bound algorithm, 
they are grouped in a pool and are at first removed
%(before possible subsequent re-introduction) 
from the constraint set before being (possibly) iteratively and selectively reinstated on an as-needed basis. 
The targeted benefit is to obtain a  reduced-size relaxation or outer approximation problem, which is quicker to solve. 
% and tight. 
%In essence, only those constraints required for finding the optimal solution are included and only when they are needed, which gives substantial computational time savings.
%While all the constraints not binding at the optimal solution can be removed, 
One should err on the side of caution when deciding which constraints are set up as lazy.
%careful to select the constraints categorized as lazy ones and to err on the side of caution. %Two aspects regarding the setup of the lazy constraint approach are worth stressing. 
Indeed, the inspection of whether a lazy constraint is violated is carried out each time a new incumbent solution for the outer approximation problem is found and the computational overheads consecutive to the reintroduction of violated lazy constraints
%in the constraint set 
can be significant.

Within this approach, the reduced-size relaxation 
%or outer approximation of the true problem -- obtained by dropping the lazy constraints -- 
is solved at each node of the tree. 
Each time a new incumbent solution is found, a verification is made to check whether any lazy constraint is violated. 
If it is the case, the incumbent integer solution is discarded and the violated lazy constraints are (re)introduced in the constraint set of all unprocessed nodes of the tree, thereby cutting off the current~solution. %The optimal value corresponding to this solution is not feasible for the true problem and constitutes a lower upper bound on the optimal value of the true minimization problem. On the other hand, if the incumbent solution does not violate any lazy constraints, the solution if feasible for the true problem and the corresponding objective value provides an upper bound for the true  problem. The node is said to be pruned by optimality.
%The optimization process goes on until an optimal solution is found without any lazy constraint violations and all nodes are pruned. 

%%%%%%%%%%%%%%%%%%%%%%%%%%%%%%%%%%%%%%%%%%%%%%%
%During the branch and bound process, when an incumbent integer solution node is found, the validity of these relaxations is checked, in a process known as lazy constraint call back. This checking can involve simply finding violated constraints. If the checks pass, then the solution becomes the incumbent solution, as per a normal branch and bound search. If the checks fail, then new constraints must be added to cut off the solution. The solver adds these constraints to the current branch and bound node, and continues the normal branch and bound process.

%%%%%%%%%%%%%%%%%%%%%%%%%%%%%%%%%%%%%

We define here the linearization constraints 
%\eqref{MAC_z1}-\eqref{MAC_z4} and \eqref{MAC_psi1}-\eqref{MAC_psi4} 
in the sets $\mathcal{F}_y^z$ and $\mathcal{M}^{\tau}_{zU}$ as lazy constraints. The motivation for this choice is threefold  and guided by: the results of preliminary numerical tests, the very large number $\sum_{j\in J}|I_j|\cdot (|I_j|-1) / 2 $
%($\sum_{j\in J}|D_j| = \sum_{j\in J}|I_j|\cdot (|I_j|-1) / 2 $) 
of such constraints, and the fact that the inequalities in the set  %\eqref{MAC_psi1}-\eqref{MAC_psi4} 
$\mathcal{M}^{\tau}_{zU}$ are not always needed, i.e., they only play a role if two drones are placed at the same DB. 

We introduce next the notations used when using the reformulation $\mathbf{R-MILP}$ within the algorithm. The exact same procedure is applicable for \MLB{$\mathbf{L-MILP}$} and only requires changing the name of the feasible set.
%The following notations are used.
Let $\mathcal O$ denote the set of open nodes in the tree.
%We distinguish the terms node and iteration.
We call iteration a node of the branch-and-bound tree at which the optimal solution of the continuous relaxation 
%of the reduced-size relaxed problem 
is integer-feasible.
%The set of iterations is a subset of the set of nodes.
Let $\mathcal{C}$ be the entire constraint set of problem $\mathbf{R-MILP}$ (see Theorem \ref{T2}), 
$\mathcal{L}_k$
%:= \{\eqref{MAC_z1} - \eqref{MAC_z4} ; \eqref{MAC_psi1} - \eqref{MAC_psi4} \}$ 
be the set of lazy constraints at node $k$, 
 $\mathcal{V}^L_k$ be the set of violated lazy constraints at $k$,
%:= \{\eqref{MAC_z1} - \eqref{MAC_z4} ; \eqref{MAC_psi1} - \eqref{MAC_psi4} \}$ 
and  
$\mathcal{A}_k:= C \setminus \mathcal{L}_k$ be the set of active constraints at $k$, i.e., the set of constraints of the reduced-size outer approximation problem $\mathbf{OA-MILP}_k$.
The composition of the sets vary across the algorithmic process.

The outer approximation framework 
% with lazy constraints 
is designed as follows. 
At the root node ($k=0$), we have:
\vspace{-0.075in}
\begin{align}
& \mathcal{L}_0:= 
\{\mathcal{F}_y^z \cup \mathcal{M}^{\tau}_{zU}\}.
%\{\eqref{MAC_z1}-\eqref{MAC_z4} ; \eqref{MAC_psi1} - \eqref{MAC_psi4}\}. 
\label{S1} \\
%& \mathcal{L}_0:= \{(y,z,U,\tau) \in \{\eqref{MAC_z1}-\eqref{MAC_z4} ; \eqref{MAC_psi1} - \eqref{MAC_psi4}\} \}. \label{S1} \\
& \mathcal{A}_0:= \{
\mathcal{B} ; \eqref{U}-\eqref{V} ;  \mathcal{M}^{\mu}_{yU} ; \mathcal{M}^{\omega}_{\gamma \mu}  \}. 
\label{S2} \\
%& \mathcal{A}_0:= \{
%(x,y,\gamma,U,\mu,\omega) \in \mathcal{B} \cap \mathcal{M}_{\mu^m_{lj}} \cap \mathcal{M}_{\omega_{lj}^m}  \}. 
%\label{S2} \\
&\mathcal{V}_0^L:= \emptyset. \label{S3}
\end{align}
\vspace{-0.36in}

\noindent At any node $k$, % of the branch-and-bound tree, 
the reduced-size relaxation (outer approximation) problem $\mathbf{OA-MILP}_k$ is solved:
\vspace{-0.07in}
\[
\mathbf{OA-MILP}_k: \; \min \eqref{obj_lin} \quad \text{s.to} \quad (x,y,\gamma,z,U,\mu,\tau,\omega) \in \mathcal{A}_k. 
\vspace{-0.075in}
\]
%The constraints in $\mathcal{F}_{LC}$ are set aside in a lazy constraint pool. 
Two possibilities arise %cases are distinguished 
depending on the optimal solution  $X^*_k$ %=(x^*,y^*,\gamma^*,z^*,U^*,\mu^*,\tau^*,\omega^*)$
of the continuous relaxation of $\mathbf{OA-MILP}_k$:
\begin{enumerate}
\vspace{-0.075in}
    \item If  $X^*_k$ is fractional, we introduce branching linear inequalities to cut off the fractional nodal optimal solution and we continue the branch-and-bound process. 
\vspace{-0.075in}    
    \item If $X^*_k$ is an integer-valued solution with better objective value than the one of the incumbent, we check for possible violation of the current lazy constraints: 
    \begin{itemize}
\vspace{-0.075in}
    \item If some constraints in $\mathcal{L}_k$ are violated by    $X^*_k$, they are inserted in $\mathcal{V}^L_k \subseteq \mathcal{L}_k$ and $X^*_k$ is discarded.   
    The lazy and active constraint sets of each open node $o \in \mathcal {O}$ are updated as follows:  
    \vspace{-0.075in}
    \[
    \mathcal{L}_o \leftarrow \mathcal{L}_o  \setminus \mathcal{V}^L_k  \quad \text{and} \quad     
    \mathcal{A}_o \leftarrow \mathcal{A}_o  \cup \mathcal{V}^L_k.
    \]
    %so that $X^*_k$ is cut off.
    %\begin{itemize}
    %\item $\mathcal{L}_o \leftarrow \mathcal{V}^L_k  \setminus \mathcal{L}_o$.
    %\item $\mathcal{A}_o \leftarrow \mathcal{A}_o  \cup \mathcal{L}_k$.
    %\end{itemize}
\vspace{-0.35in}
    \item If no lazy constraint 
    %in $\mathcal{L}_k$ 
    is violated, % by $X^*_k$, 
    $X^*_k$ becomes the incumbent solution %for the true problem 
    and the node~is~pruned.
    \end{itemize}
\end{enumerate}
\vspace{-0.05in}
The above process terminates when all nodes are pruned. The verification of the possible violation of the lazy constraints is carried out within a callback function, which is 
%. It is worth stressing that the callback verification is 
not performed at each node of the tree, but only when a better integer-valued feasible solution is found. % at a node. 
The use of the lazy constraints within the outer approximation procedure is pivotal in the proposed method as shown in Section \ref{sub_sec_compute_e}. 

\vspace{-0.1in}
\subsection{Branch-and-Cut with Valid Inequalities and Optimality Cuts} \label{B&C}
\vspace{-0.05in}
The linearization approach for \MLB{{\bf R-MILP}}  involves the introduction of big-M constraints which typically lead to loose continuous relaxations. 
To tighten the continuous relaxation, we derive new valid inequalities and optimality cuts. 

A valid inequality does not rule out any feasible integer solutions but cuts off  fractional solutions feasible for the continuous relaxation problem. In essence, it pares away at the space between the linear and integer hulls, thereby providing a tighter %strengthened 
formulation. 
In contrast to lazy constraints, a valid inequality is inserted in the formulation if the optimal solution of the continuous relaxation at any node is fractional and violates this valid inequality.
%User cuts are those cuts that a user defines based on information already implied about the problem by the constraints; user cuts may not be strictly necessary to the problem, but they tighten the model. 
%Cuts are conventionally defined to mean those which can change the feasible space of the continuous relaxation but do not rule out any feasible integer solution that the rest of the model permits. 
We shall now derive several types of valid inequalities.
%with the objective of strengthening the continuous relaxation of \textbf{R-MILP} and of speeding up its solution.

%%%%%%%%%%%%%%%%%%%%%%%%%%%%%%%%%%%%%%%%%%%%%%%%%%%%
%%%%%%%%%%%%%%%%%%%%%%%%%%%%%%%%%%%%%%%%%%%%%%%%%%%%
%\vspace{0.05in}
%\noindent{\bf Valid Inequalities on Delay}  %\label{VI-DELAY}
%\vspace{0.05in}

%\noindent
The valid inequalities \eqref{VI3} reflect the fact that, if either one of the variables $U_j^{\WMR{(1)}}$ or $U_j^{\WMR{(2)}}$ related to the \ML{expected} delay is positive, then at least one drone is positioned at DB $j$, and that, if DB $j$ is not active, then the variables $U_j^{\WMR{(m)}}, m=1,2$ are equal to 0.
\begin{prop} \label{VALID1}
Let $U_{j}^{\WMR{(m)}} \in [0,\bar{U}_j^{\WMR{(m)}}], m=1,2$.
The linear constraints 
\label{valid_c_3}
\begin{equation} 
\label{VI3}
\gamma_j^{\WMR{(1)}} + \gamma_j^{\WMR{(2)}} \geq U_j^{\WMR{(m)}} / \bar{U}_{j}^{\WMR{(m)}} , \quad j \in J, m=1,2 
\end{equation}
%with $\bar{U}_{jm}$ denoting the $U_j^m$, 
are valid inequalities for problem \textbf{R-MILP}.
\end{prop}

The next valid inequalities \eqref{VI4} reflect that the queueing delay at a DB $j$ is a decreasing function of the number of drones positioned at this DB.
\begin{prop} \label{valid_c_6} 
\vspace{-0.075in}
The linear constraints
\vspace{-0.1in}
\begin{equation}  
\label{VI4}
U_j^{\WMR{(1)}} \geq U_j^{\WMR{(2)}}  , \quad j \in J
\end{equation}
are valid inequalities for problem \textbf{R-MILP}.
\end{prop}
The proofs of Proposition \ref{VALID1} and Proposition \ref{valid_c_6} are given in Appendix \ref{A3} and Appendix \ref{A4}.

%%%%%%%%%%%%%%%%%%%%%%%%%%%%%%%%%%%%%%%%%%%%%%%%%
%%%%%%%%%%%%%%%%%%%%%%%%%%%%%%%%%%%%%%%%%%%%%%%%%
%%%%%%%%%%%%%%%%%%%%%%%%%%%%%%%%%%%%%%%%%%%%
%%%%%%%%%%%%%%%%%%%%%%%%%%%%%%%%%%%%%%%%%%%%
%\vspace{0.05in}
%\noindent{\bf Optimality Cuts}    \label{OPT-CUTS}
%\vspace{0.05in}

%\vspace{-0.05in}
%\subsubsection{Optimality Cuts} \label{OPT-CUTS}
%\vspace{-0.05in}
%\noindent
Besides valid inequalities, we also derive optimality cuts (see Proposition \ref{valid_c_5}) which cut off integer feasible solutions that are not optimal or in a way that {\it not all} optimal solutions are removed. 
The proposed optimality cuts  \eqref{VI5} state that the  opening of a DB at location $j$ is required if and only if either one or two drones are positioned at $j$. 
The proof of Proposition \ref{valid_c_5} can be found in Appendix \ref{A5}.
\begin{prop} 
\label{valid_c_5} 
\vspace{-0.075in}
The linear constraints
\vspace{-0.1in}
\begin{equation} 
\label{VI5}
\gamma_j^{\WMR{(1)}} + \gamma_j^{\WMR{(2)}} = x_{j} , \quad j \in J
\end{equation}
are optimality-based cuts for problem \textbf{R-MILP}.
\end{prop}

%%%%%%%%%%%%%%%%%%%%%%%%%%%%%%%%%%%%%%%%%%%%%%%%%%%
%%%%%%%%%%%%%%%%%%%%%%%%%%%%%%%%%%%%%%%%%%%%%%%%%%%
\vspace{-0.1251in}
\subsection{Cut Integration in Outer Approximation Method} \label{STRUCT}
\vspace{-0.075in}
The incorporation of the valid inequalities and optimality cuts in the outer approximation branch-and-cut algorithm {\tt OA-B\&C} has a significant computational impact as shown in Section \ref{PER_EVA}. The outer approximation branch-and-cut algorithm {\tt OA-B\&C} is structured as follows.

\ML{
First, the set $\mathcal{B}$ of optimality cuts \eqref{VI5} are added to the pool of lazy constraints \eqref{S1}: 
	\vspace{-0.05in}
	\begin{equation}
		\label{S4}
		\mathcal{L}'_0:= \mathcal{L}_0 \cup \mathcal{B}.
		\vspace{-0.05in}
	\end{equation}
Second, the valid inequalities 
 %\eqref{VI3} and \eqref{VI4} 
% and the optimality-based cuts \eqref{VI5} 
are derived up-front, incorporated into a pool of user cuts, and applied and checked dynamically each time the nodal optimal solution $X^*_k$ is fractional through a user cut callback implemented in {\sc Gurobi}.
Let $\mathcal{U}_0$ be the set of user cuts (valid inequalities) at the root node: % ($k=0$):
$\mathcal{U}_0
:= \{\eqref{VI3}-\eqref{VI4} \}$ while  $\mathcal{V}^U_k$
%:= \{(\gamma,U) \in \{\eqref{VI3}-\eqref{VI4} \} \}$ 
is the set of valid inequalities violated by the fractional optimal solution of the continuous relaxation at node $k$.
	
The valid inequalities modifies the 
 %general 
 outer approximation framework as follows. 
%The user cut callback is only applied at a node $k$ at which the nodal optimal solution $X^*_k$ of the continuous relaxation is fractional. 
If $X^*_k$ is fractional, the user callback is applied, and the violated -- if any -- valid inequalities in the current user cut pool are added to the active constraint set of each open node $o \in \mathcal {O}$ (i.e., to cut off $X^*_k$)  and removed~from~$\mathcal{U}_o$:} 
% from the user cut pool:}
\vspace{-0.1in}
\[
\mathcal{U}_o \leftarrow \mathcal{U}_o  \setminus  \mathcal{V}^U_k \quad \text{and} \quad
\mathcal{A}_o \leftarrow \mathcal{A}_o  \cup \mathcal{V}^U_k.
\vspace{-0.085in}
\]
If no inequality in $\mathcal{U}_o$ is violated by $X^*_k$, then two branching constraints are entered to cut off $X^*_k$ and the next open node is processed. 
Note that, if $X^*_k$ 
%the nodal optimal solution of the continuous relaxation 
is integer feasible, the user cut callback is not applied. The algorithm stops when the set of unprocessed nodes becomes empty. 
The pseudo-code of the algorithmic method {\tt OA-B\&C} is in Appendix \ref{A-PS}.

%%%%%%%%%%%%%%%%%%%%%%%%%%%%%%%%%%%%%%%%%%%%%%%%%%%%%%%%
\vspace{-0.1in}
\MLB{
\begin{rem} \label{rem_oa}
We have implemented a simpler outer approximation {\tt OA} algorithm %with lazy constraints 
based on the general framework described in Section \ref{sub_sec_lazy_c}. It is intended to serve as a benchmark and, in particular, to assess the added value of incorporating the valid inequalities and optimality cuts described above in the outer approximation method.
For concision purposes and since, as shown in Section \ref{sub_sec_compute_e}, it does not perform as well as {\tt OA-B\&C}, we refer the reader to Appendix \ref{APP-sub_sec_lazy_c} for its detailed description. 
\end{rem}
}

%%%%%%%%%%%%%%%%%%%%%%%%%%%%%%%%%%%%%%%%%%%%
\vspace{-0.25in}
\section{Data-Driven Tests and Insights} 
\label{sec_TESTS}
\vspace{-0.15in}
To demonstrate the benefits and applicability of the proposed approach,
%and validate its computational efficiency, 
we conduct extensive numerical tests and a simulation study using the real-life 
%obtained from the city of Virginia Beach. Section \ref{sub_data} describes the 
data described in Section \ref{sub_data}. 
%and summarizes the computing environment used for the tests.
Section \ref{sub_delivery} describes the simulation framework used to carry out a comprehensive comparison of our approach with the EMS system in place in Virginia Beach, 
provides insights on response time, chance of survival, QALY, and costs, and attests the applicability and robustness of our approach through a cross-validation analysis.
Section \ref{sec_baseline} examines the added value of the network constructed with our model as compared to  those obtained with \ML{two constructive greedy heuristics and a model from the literature \cite{BOCHAN}.} 
Section \ref{sub_sec_compute_e} evaluates the computational efficiency and tractability of the reformulation and algorithmic framework.
%with respect to size and volume of opioid overdose requests, and with respect to the resources of the drone network.

%%%%%%%%%%%%%%%%%%%%%%%%%%%%%%%%%%%%%%%%%
%%%%%%%%%%%%%%%%%%%%%%%%%%%%%%%%%%%%%%%%%
%%%%%%%%%%%%%%%%%%%%%%%%%%%%%%%%%%%%%%%%%
%%%%%%%%%%%%%%%%%%%%%%%%%%%%%%%%%%%%%%%%%
\vspace{-0.096in}
\subsection{Real-life Opioid Overdose Data} \label{sub_data}
\vspace{-0.085in}
The dataset used in the tests describes all the opioid overdose requests (actual and suspected opioid incidents) in the city of Virginia Beach and is publicly available \cite{CustodioLejeune2021}.  
The data were collected through multiple sources, including the OpenVB data portal, Freedom of Information Act requests to the government of Virginia Beach, and public reports, and were  
%The data collection process was 
validated by the Virginia Beach officials. 

The dataset contains all dispatch records for OTRs from the second quarter of 2018 to the third quarter of 2019, which amounts to a total of  733 data points (overdoses). 
Each record has four fields including the time at which the request was received by the EMS, the response time, i.e, time between the reception of the request and the arrival of the EMS personnel on the scene, and the location (i.e., latitude and longitude) of the request. The dataset also provides the location of the twenty-six established EMS  facilities (i.e., fire, police, and EMS stations) that can be selected as drone bases. 
As in \cite{boutilier2017optimizing}, we consider that  drones can travel at a speed of $27.8$ meters per second  (m/s) and take 10 seconds to take off and to land. We use 25 minutes as the expected non-travel service time which includes the time to administer 
 naloxone at the overdose location and to recharge and prepare the drone for the next assignment. 

We have conducted a grid search 
%conducted on the data described in Section revealed \ref{sub_data} that
in order to identify the minimal size of the drone network to be able to respond to all OTRs. 
%Specifically, we develop various size of drone networks using each of the training sets (2018Q2.JPG, 2018Q3, 2018Q4 and 2019Q1) in Table \ref{data_summary}. We then test whether it can serve all the demands in the corresponding testing sets (2018Q3, 2018Q4, 2019Q1 and 2019Q2). 
This reveals that the drone network should include $p=11$ drones and does not require more than $q=10$ DBs. Thereafter, we refer to {\it base scenario} the case in which the drone response network allows for the opening of ten DBs and the deployment of eleven drones.

%%%%%%%%%%%%%%%%%%%%%%%%%%%%%%%%
%%%%%%%%%%%%%%%%%%%%%%%%%%%%%%%%
%%%%%%%%%%%%%%%%%%%%%%%%%%%%%%%%%%%%%%%%%
%%%%%%%%%%%%%%%%%%%%%%%%%%%%%%%%%%%%%%%%%
\vspace{-0.06in}
\subsection{Interplay Between Response Time, Survival Chance, QALY, and Delivery Mode} \label{sub_delivery}
\vspace{-0.075in}
In this section, we study the relationships between response time, survival chance, delivery mode, and compare the drone-network constructed with our approach with the ambulance-based EMS network currently in use in Virginia Beach.
To assess the two EMS systems under similar experimental settings, we use the historical real-life data for the ambulance-based EMS system in Virginia Beach and the data obtained from the simulator (presented next in Section \ref{sec_sim}) for the drone system. 
Using the results calculated numerically with our optimization model may have provided a less balanced comparison as the mathematical model may ignore some practical issues that may exist in realty. %To be self-contained, we provide in Appendix \ref{} the comparison of the ambulance EMS network in Virginia Beach with the drone one in which the performance of this latter is estimated with our optimization approach. The results 
\ML{Section \ref{sub_response} analyzes the reduction in the response time attributable to the drone network, the applicability and robustness of the proposed approach, and the spatio-temporal adjustments in the OTRs and drone network. 
Section \ref{sub_SURVIVAL} investigates the increase in the chance of survival of an overdose victim and the expected number of lives saved. Section \ref{QALY} quantifies the impact on the QALY and performs a cost analysis.
%In Appendix \ref{OPT-PERF}, we provide the metrics presented in Sections \ref{sub_response}, \ref{sub_SURVIVAL}, and \ref{QALY} when calculated directly with our optimization model. 
%\MLR{The results accentuate the benefits of our network over the ambulance-based EMS system in Virginia Beach.}
}

%%%%%%%%%%%%%%%%%%%%%%%%%%%%%%%%
%%%%%%%%%%%%%%%%%%%%%%%%%%%%%%%%
\vspace{-0.1208in}
\subsubsection{Drone Network Simulator} \label{sec_sim}
\vspace{-0.095in}
To evaluate the performance of the drone network, 
%and compare it the ambulance-based EMS system in place in Virginia Beach, 
we develop a simulation model that replicates how the network constructed with the proposed model responds to emergency opioid overdose requests. 
The simulation framework captures the dynamic nature of the queueing system, and may provide a more realistic estimate of the response times than the one obtained by calculating them numerically (with the optimization model and objective function). 
The simulator dispatches the nearest available drone when a request for an opioid overdose is transmitted. If no drone is available within the catchment area, the request is queued until a drone becomes available.
Our simulation model is similar to the simulator developed for ambulance-based EMS systems in \cite{BoutilierOR2020} but incorporates a few adjustments. Indeed, due to its limited battery autonomy, a drone always flies back to its base after completion of the service at the scene, whereas an ambulance in \cite{BoutilierOR2020} can be either re-dispatched to a new incident or  go back to its home base. 
\ML{The implementation details and pseudo-code of the simulator are given in Appendix \ref{A-Sim}.}

%%%%%%%%%%%%%%%%%%%%%%%%%%%%%%%%%%%%%%%%%
%%%%%%%%%%%%%%%%%%%%%%%%%%%%%%%%%%%%%%%%%
%%%%%%%%%%%%%%%%%%%%%%%%%%%%%%%%%%%%%%%%%
%%%%%%%%%%%%%%%%%%%%%%%%%%%%%%%%%%%%%%%%%
\vspace{-0.1in}
\subsubsection{Response Time Reduction and Network Robustness} \label{sub_response}
\vspace{-0.06in}
We analyze here how the response time, a critical metric for EMSs, can be improved by using drones. We first carry out an in-sample analysis before cross-validating the results and assessing, using out-of-sample data, the stability of the results and the robustness of the designed network. 

\noindent
{\bf In-sample analysis:} 
We consider quarterly datasets of opioid overdoses which we denote 2018Q2 (2018's second quarter), 2018Q3, 2018Q4, 2019Q1, and 2019Q2. We consider the base scenario in which one can open up to 10 DBs and deploy 11 drones ($q=10, p=11$),
To estimate the in-sample performance and response times of the drone network (see Table \ref{c_test_result}), we follow a two-step process. First, we solve the optimization problem {\bf R-MILP} %\ML{(i.e., {\bf R-MILP} and {\bf R2-MILP} are equivalent and give identical solutions)} 
and use its solution to configure the network, namely to decide where to open DBs and where to position drones. 
%We refer to the response times obtained from the solution of the {\bf R-MILP}  optimization problem as “algorithmic” response times. The algorithmic response times are not presented in Table \ref{test_result} but are analyzed later, in Section \ref{sec_baseline} and Table \ref{baseline_compare}.
Second, we run the simulation model (Section \ref{sec_sim}) based on the network configuration obtained in the first step and calculate the resulting response times, called simulated response times. All response times reported in the paper are simulated times.

%To estimate the in-sample performance and response times of the drone network, we solve the network design problem $\mathbf{R-MILP}$ associated to a quarterly training dataset, retrieving the optimal base locations and the number of drones deployed. Next, we run the simulation described in Section \ref{sec_sim} to obtain the (simulated) response time for each OTR in the training quarter, using the network obtained from $\mathbf{R-MILP}$.

% we solve the network design problem $\mathbf{R-MILP}$ associated to a quarterly training dataset, retrieve its optimal solution and value, and calculate the average response time across all opioid overdoses in the corresponding quarter. 

Table \ref{c_test_result} shows the average \ML{simulated} response times on the training sets with the proposed drone network and those obtained in Virginia Beach with their ambulance-based EMS network, which highlights the significant decrease in the average response time enabled by the drone network (i.e., \ML{1 minute and 32 seconds} versus 8 minutes and 56 seconds for the ambulance network). As compared to the current ambulance network, the drone network reduces the response time by \ML{$82.8\%$} on average. The reduction of the response time is stable across all quarters and is not due to chance.  

% \begin{table}[H]
% \centering
% \setlength\extrarowheight{1pt}
% \begin{tabular}{P{3cm}|P{3cm}|P{4.5cm}| P{3cm}}
%  \hline
% \multirow{2}{*}{Training Set} & \multicolumn{3}{c}{Response Time (minutes)}  \\ \cline{2-4}  
% \multirow{2}{*}{} & Drone Network
% & VB Ambulance Networks & Time Reduction\\
%  \hline
% 2018Q2.JPG & 2.12 & 8.77 & 75.83\%  \\
% 2018Q3 & 2.14 & 9.02 & 76.27\% \\ 
% 2018Q4 & 1.99 & 9.42 & 78.87\% \\
% 2019Q1 & 2.20 & 8.56  & 74.30\% \\
% \hline
% Quarter Average & 2.11 & 8.94 & 76.40\% \\
%  \hline
% \end{tabular}
% \caption{\label{test_result} Response Times in Quarterly Training Sets ($q = 10, p = 11$)}
% \end{table}

\begin{table}[H]
\centering
\setlength\extrarowheight{1pt}
\ML{
\begin{tabular}{P{3cm}|P{3cm}|P{4.5cm}| P{3cm}}
 \hline
\multirow{2}{*}{Training Set} & \multicolumn{3}{c}{Response Time (minutes)}  \\ \cline{2-4}  
\multirow{2}{*}{} & {\bf R-MILP}
& VB Ambulance Networks & Time Reduction\\
 \hline
2018Q2 & \textcolor{black}{1.49} & 8.77 & 83.01 \%   \\
2018Q3 & \textcolor{black}{1.58} & 9.02 & 82.48 \% \\ 
2018Q4 & \textcolor{black}{1.54} & 9.42 & 83.65 \% \\
2019Q1 & \textcolor{black}{1.50} & 8.56  & 82.48 \% \\
\hline
Quarter Average & 1.53 & 8.94 & 82.92 \% \\
 \hline
\end{tabular}
\vspace{-0.04in}
\caption{\label{c_test_result} Response Times for Quarterly Training Sets ($q = 10, p = 11$). \textcolor{blue}{The confidence intervals for the responses time obtained with {\bf R-MILP} can be found in Table \ref{ci_train_milp} in Appendix \ref{APP-CI}.}}
}
\end{table}

% 2018Q2.JPG & 1.34 \WM{(2.12)} & 8.77 & 84.72\%  \\
% 2018Q3 & 1.36 \WM{(2.14)} & 9.02 & 84.92\% \\ 
% 2018Q4 & 1.25 \WM{(1.99)} & 9.42 & 86.73\% \\
% 2019Q1 & 1.29 \WM{(2.20)} & 8.56  & 84.93\% \\

%%%%%%%%%%%%%%%%%%%%%%%%%%%%%%%%%%%%%%%%%%%%%%%%%%%%%%%%
\vspace{-0.15in}

\noindent
{\bf Out-of-sample analysis:} We now carry out a cross-validation analysis to assess how the training set-based networks perform on out-of-sample data that were not used in the design of the network.  
%To validate the performance improvement not only on the in-sample data (i.e. the data where the network is developed) but also the out-of-sample data, we split the data into four pairs.  
Using the five sets of quarterly data, we create 
%which are denoted 2018Q2.JPG (1018's second quarter), 2018Q3, 2018Q4, 2019Q1, and 2019Q2. 
%We then create 
four pairs of consecutive quarterly sets, e.g., (2018Q2 and 2018Q3), where the first one is the training set (2018Q2) and the second one is the testing set (2018Q3).   
%\ML{$(, r+1)$} (i.e., $r$ is the training quarter and $r+1$ is the testing quarter) 
Table \ref{data_summary} in Appendix \ref{app-sec-data-summary} describes the size of the datasets.  The number of OTRs %(sample size) 
in each quarterly dataset is $|I|$. 

% \vspace{-0.16in}
For each pair of training and testing sets, we use the training set only to build the network: we solve the network design problem $\mathbf{R-MILP}$ to obtain the optimal drone network configuration (i.e. locations of opened DBs and number of drones deployed at each). The average response times is calculated by running the simulator (see Section \ref{sec_sim}) using the OTRs data from the subsequent testing quarter.
% associated to the training set and retrieve its optimal solution. This provides us with the training set-based optimal configuration of the network, namely the optimal location $x^*$ of the opened DBs and the optimal number of drones (which can be inferred from the variables $\gamma^{*}$) to be stationed at each DB.
% Using the optimal location decisions for the training set-based network, we then solve an assignment problem corresponding to the testing  set, which is a "reduced" form of problem $\mathbf{R-MILP}$ in which the variables $x$ and $\gamma$ are fixed to their optimal values for the training set: $x:=x^*$ and $\gamma:= \gamma^*$. The optimal solution of the testing set-based assignment provides the optimal drone-dispatching decisions for OTRs based on the training set-based network configuration, which allows us to calculate the individual and average out-of-sample (testing set) response times to OTRs.
%%%%%%%%%%%%%%%%%%%%%%%%%%%%%%%%%%%%%%%%%%%%%%%%%%%%%%%%
The results are displayed in Table \ref{c_test_result_testing}. The cross-validation analysis confirms the results envisioned in the in-sample analysis. The network configurations for the training sets reduces in a striking manner the out-of-sample response time. Across all quarters, the average out-of-sample response time is of \ML{1 minutes and 38 seconds} while it amounts to 9 minutes and 19 seconds for the ambulance network in Virginia Beach.  This corresponds to a \ML{$82.5\%$} reduction in the average response~time.
The quarterly average response times are stable and similar to those obtained in the in-sample analysis.

\begin{table}[h]
\centering
\setlength\extrarowheight{1pt}
\ML{
\begin{tabular}{P{3.6cm}|P{3.6cm}|P{4.2cm}| P{2.8cm}}
 \hline
\multirow{2}{*}{Training/Testing Data} & \multicolumn{3}{c}{Testing  Quarter Response Time (min.) }  \\ 
\cline{2-4}  
\multirow{2}{*}{} & {\bf R-MILP}
& VB Ambulance Network & Time Reduction\\
 \hline
2018Q2 / 2018Q3 & \textcolor{black}{1.72} & 9.02 &	80.93\%  \\
2018Q3 / 2018Q4 & \textcolor{black}{1.57}  & 9.42 &	83.33\% \\ 
2018Q4 / 2019Q1 & \textcolor{black}{1.80} & 8.56 &	78.97\% \\
2019Q1 / 2019Q2 & \textcolor{black}{1.48} & 10.27 & 85.59\% \\
\hline
Quarter Average & 1.64 & 9.32 &	82.40\% \\
 \hline
\end{tabular}
}
\vspace{-0.04in}
\caption{\label{c_test_result_testing} Out-of-Sample Response Time for Testing Quarters ($q = 10, p = 11$). \MLB{The confidence intervals for the responses time obtained with {\bf R-MILP} can be found in Table \ref{ci_test_milp} in Appendix \ref{APP-CI}.}}
\vspace{-0.1in}
\end{table}

% 2018Q2 / 2018Q3 & 1.51 \WM{(2.06)B(2.22, 3.26)} & 9.02 &	83.26\%  \\
% 2018Q3 / 2018Q4 & 1.38 \WM{(2.10)B(2.08, 2.91)} & 9.42 &	85.35\% \\ 
% 2018Q4 / 2019Q1 & 1.48 \WM{(1.76) B(2.01, 2.10)} & 8.56 &	82.71\% \\
% 2019Q1 / 2019Q2 & 1.38 \WM{(2.24)B(2.14,2.30)} & 10.27 &	86.56\% \\
% \hline
% Quarter Average & 1.43 \WM{(2.04)} & 9.32 &	84.57\% \\

%%%%%%%%%%%%%%%%%%%%%%%%%%%%%%%%%%%%%%%%%%%%%%%%%%%%%%%%%%%%
We now illustrate the spatio-temporal variability of the OTRs and its impact on the optimal configuration of the drone network.  
 Figure \ref{fig:dnn_network} 
 presents the distribution of the OTRs and 
 %the configurations  of the 
 drone networks for two consecutive quarters, i.e., 2018's second (2018Q2) and third (2018Q3) quarters. 
 In both networks, ten DBs are opened; one drone is located at nine of the DBs operating as $M/G/1$ systems while the last one operates as an $M/G/2$ system with two available drones. 
%We observe that there are 10 facilities selected as drone bases, with 9 of them having 1 drone deployed and the rest one housing 2 drones. 
Comparing the two networks, one can see that the location of most DBs remains unchanged over time. One noticeable change is that, in the 2018Q3 network a new DB is opened in the southwest to cover an increased  number of OTRs in this remote area. 
Another change is the new DB in the North. In Figure \eqref{fig:subim2}, the two new DBs are surrounded by a rectangle while the two closed ones are circled. The location changes reflect the impact of the spatial uncertainty on the optimal configuration of the network. In that regard, the flexibility, ease, and limited cost to open (close) DBs \cite{van2017drone} are important to adapt to geographical changes in the occurrence of OTRs. 

\begin{figure}[h]
\begin{subfigure}[b]{0.5\textwidth}
\includegraphics[width=\textwidth, height=6.7cm]{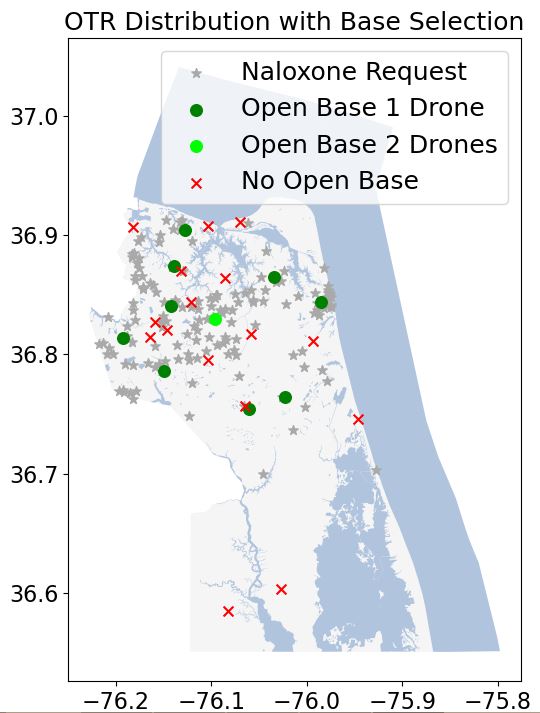} 
\caption{2018 Quarter 2}
\label{fig:subim1}
\end{subfigure}
\hfill
\begin{subfigure}[b]{0.5\textwidth}
\includegraphics[width=\textwidth, height=6.7cm]{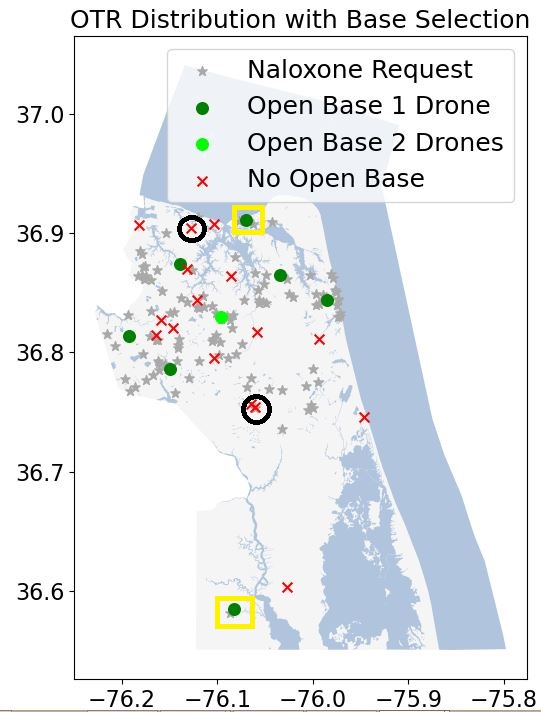}
\caption{2018 Quarter 3}
\label{fig:subim2}
\end{subfigure}
\vspace{-0.25in}
\caption{%(Color Online) 
OTR Demand and Drone Network Configuration}
\label{fig:dnn_network}
\vspace{-0.15in}
\end{figure}

%%%%%%%%%%%%%%%%%%%%%%%%%%%%%%%%%%%%
%%%%%%%%%%%%%%%%%%%%%%%%%%%%%%%%%%%%
%%%%%%%%%%%%%%%%%%%%%%%%%%%%%%%%%%%%
%%%%%%%%%%%%%%%%%%%%%%%%%%%%%%%%%%%%
% \newpage
% \subsubsection{Cross-Validated Benefits of Drone Network} \label{sub_insights}

%%%%%%%%%%%%%%%%%%%%%%%%%%%%%%%%%%%%%%%%%%%%%%%
% When designing a network of drones, one important question needs to be answered is how many bases to open and how many drones to deploy in the network so that the network is robust to serve all the future demands. With the data shown in table \ref{data_summary}, We find that the network needs to least contain 10 drone bases with 11 drones deployed in order to cover all the unknown demands in the subsequent quarter. Otherwise the optimal network configuration (i.e. opened drone bases) obtained from the training phase will lead to an infeasible problem in the evaluation phase. 

%%%%%%%%%%%%%%%%%%%%%%%%%%%%%%%%%%%%%%%%%%%%
%%%%%%%%%%%%%%%%%%%%%%%%%%%%%%%%%%%%%%%%%%%%
%%%%%%%%%%%%%%%%%%%%%%%%%%%%%%%%%%%%%%%%%%%%
%%%%%%%%%%%%%%%%%%%%%%%%%%%%%%%%%%%%%%%%%%%%
%%%%%%%%%%%%%%%%%%%%%%%%%%%%%%%%%%%%
%%%%%%%%%%%%%%%%%%%%%%%%%%%%%%%%%%%%
%%%%%%%%%%%%%%%%%%%%%%%%%%%%%%%%%%%%
\vspace{-0.115in}
\subsubsection{Impact on Probability of Survival} \label{sub_SURVIVAL}
\vspace{-0.085in}
The goal of this section is to quantify the benefits of the reduced response time afforded by the drone network (Section \ref{sub_response}) on the survival chance of overdose victims. 
Since we are not aware of any functional relationship between response time and  survival probability of an opioid overdose victim, we hereby employ an indirect two-step approach using known survival functions for cardiac arrests. The main reason for this is that many overdoses lead to cardiac arrests. Indeed, the increase in cardiac arrest caused by opioid medications is said to be “{\it the most dramatic manifestation of opioid use disorder}” \cite{dezfulian2021opioid}.

\noindent
\underline{Step 1:}
We calculate the number of out-of-hospital cardiac arrests (OHCA) due to opioid overdoses using the statistics reported by \citeauthor{dezfulian2021opioid} \cite{dezfulian2021opioid} \ML{who} indicate that 15\% of opioid overdoses lead to an OHCA. 
As shown in Table \ref{data_summary}, 560 overdose incidents are reported in Virginia Beach over one year (i.e., from 2018Q3 to 2019Q2), leading to an estimated (see \cite{dezfulian2021opioid}) number of 84 overdose-associated OHCAs. 

\noindent
\underline{Step 2:}
Using three known survival functions for OHCAs (see Table \ref{survival}) that define the survival probability $f(x)$ to an OHCA as either a semi-continuous or a logistic function of the response time $x$, we calculate the estimated probability of survival for overdose-associated OHCAs with the drone network and with the EMS ambulance network in use in Virginia Beach. This allows us in turn to derive the differences in the survival chance and in the expected number of saved lives with the drone and ambulance networks.  
%connect the which have been we compare the number of survival patients among those CA incidents in the proposed drone network and the current EMS system. We use four survival functions for CA displayed in Table \ref{survival}. 
%\vspace{-0.05in}
\begin{table}[h]
	\centering
	\vspace{-0.1in}
	%	\begin{adjustbox}{width=0.995\textwidth}
	\setlength\extrarowheight{4.75pt}
	%\setlength\extrarowheight{-5pt}
	% \renewcommand{\arraystretch}{0.75} % Default value: 1
	%		\renewcommand{\arraystretch}{2}
	%		\baselineskip=0.8\normalbaselineskip
	%\resizebox{\textwidth}{!}{%
	\begin{tabular}{c|c|c}
		Author & Function Type & $f(x)$ \\ 
		%\hline
  		%Weaver et al. \cite{weaver1986factors} 
  		%& 		1 
  		%& Semi-continuous &	$\max \left[0.466-0.0283x_{2},0\right]$	\\
		\hline
 		Bandara et al. \cite{Bandara:2014} & 
 		%Semi-continuous &        %2 & 
        Semi-continuous & 
		$\max \left[ 0.594-0.055x, 0\right]$	\\
		\hline			
 		De Maio et al. \cite{DEMAIO} %	&  3 
 		& Logistic & $\left(1+e^{0.679+0.262x}\right)^{-1}$\\
		\hline
		%			\cite{Knight:2012}	& Logistic &	$\left(1+e^{-0.26+0.139x_{1}}   \right)^{-1}$	\\
 		Chanta et al. \cite{Chanta2014}	%4 
 		& Logistic &	
 		$\left(1+e^{-0.015+0.245x}\right)^{-1}$	\\
		%			\cite{McCormack:2015} & Logistic &	$\left(1+e^{-0.26+0.139x_{1}}   \right)^{-1}$	\\
		%			\cite{matinrad2019optimal}	& Logistic &	$\left(1+e^{-1.3614+0.3429x_{1}+0.18633x_{2}}   \right)^{-1}$	\\
% 		\Xhline{4\arrayrulewidth}
	\end{tabular}%
		\caption{Survival Functions for OHCAs}
	\label{survival}
	%	\end{adjustbox}
\vspace{-0.175in}
\end{table}
%\ML{We refer the interested readers to \cite{LejeuneOHCA} for more details on OHCA survival functions.}

% As shown in the last row of \WM{Table \ref{c_test_result_testing}}, the \ML{out-of-sample} average response time (over one year) for the drone network is 2.04 minutes whereas the current EMS network takes 9.32 minutes.
%for the same period. 
Figure \ref{fig:survival} shows the \ML{average} survival probability -- for the three survival functions in Table \ref{survival} -- of the overdose-associated OHCAs obtained with the drone and with Virginia Beach's ambulance-based EMS network. 
\ML{The survival probability reported in Figure \ref{fig:survival} is the average of the survival probability of each OTR in the testing set from 2018Q3 to 2019Q2 in Table \ref{c_test_result_testing}.}
The difference between the two networks is striking. The drone network significantly increases the patients' survival probability.
%with each survival function.
%  \cite{weaver1986factors} 
Indeed, the survival chance with the drone network 
is \ML{3.54 (resp., 4.00 and 2.73)} times larger than the one with Virginia Beach's ambulance network as estimated by the survival function \cite{Bandara:2014}
 (resp., \cite{DEMAIO}, and \cite{Chanta2014}).
\vspace{-0.05in} 
\begin{figure}[H]
\centering
\includegraphics[width=\textwidth, height=5.5cm]{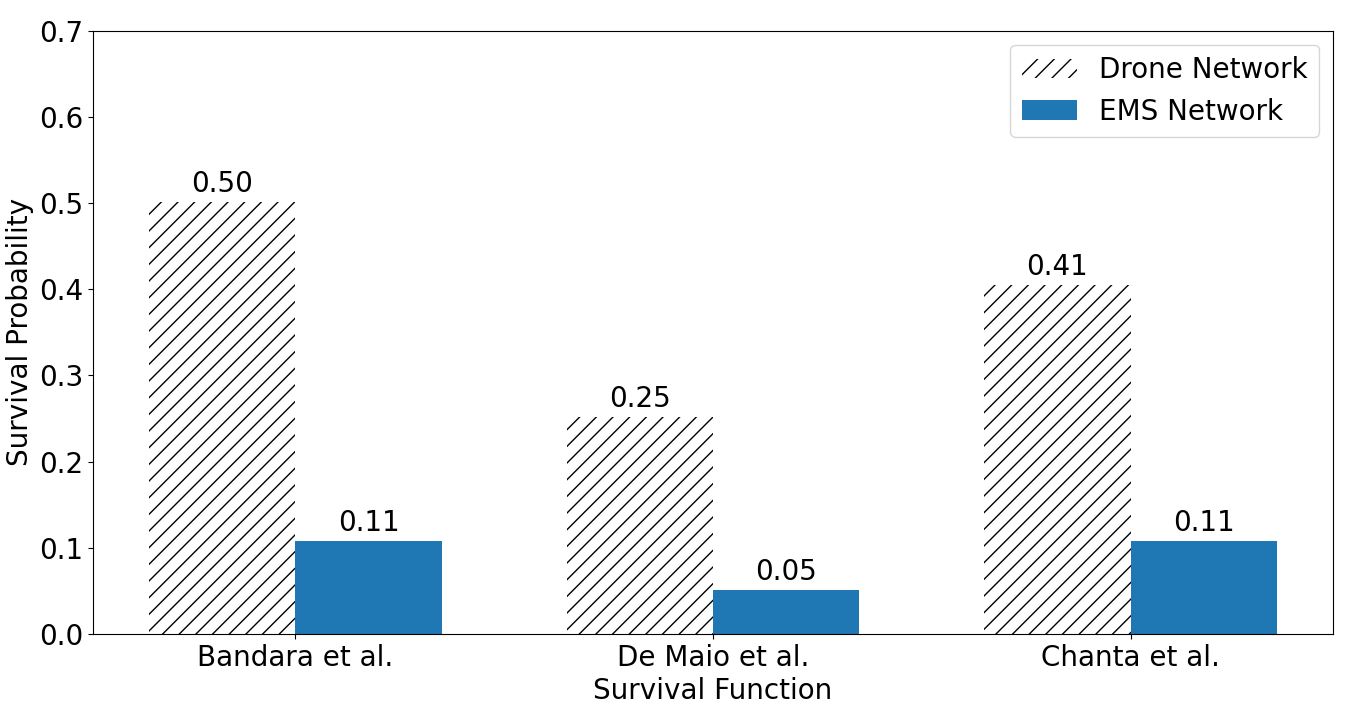}
\vspace{-0.3in}
\caption{%(Color Online) 
\ML{Average} Survival Probability for Drone Network and Virginia Beach Ambulance Network}
\label{fig:survival}
\vspace{-0.15in}
\end{figure}

Table \ref{casualty} displays the survival probability and the expected number of patients surviving an overdose-associated OHCA in Virginia Beach for the four quarters considered. These statistics are provided for the drone network (column 4) and for Virginia Beach's ambulance network (column 5).
One can see that the drone network is extremely beneficial and increases the number of survivors by \ML{366\%, (resp., 425\% and 278\%)} using the survival function \cite{Bandara:2014} (resp.\cite{DEMAIO} and \cite{Chanta2014}). 

\vspace{-0.05in}
\begin{table}[H]
\centering
%\resizebox{\columnwidth}{2cm}{%
\begin{tabular}{P{2.4cm} | P{2.8cm} | P{3cm} | P{2.65cm} | P{2.65cm} }
 \hline
Total Number & Overdose-related & 
Survival & \multicolumn{2}{c}{ Survival Prob. / No. of Survivors} \\ 
\cline{4-5} of Overdoses& OHCAs & Function & Drone Network & EMS Network \\
 \hline
\multirow{3}{*}{560} &
\multirow{3}{*}{84} 
%  & Weaver et al. \cite{weaver1986factors}  & 43\% / 36 & 20\%  / 16 \\
 & 	Bandara et al.
 \cite{Bandara:2014} &  50\% / 42 & 11\% / 9  \\
 & & De Maio et al. \cite{DEMAIO} & 25\% / 21 & 5\% / 4 \\
 & 	& Chanta et al. \cite{Chanta2014} & 41\% / 34 & 11\% / 9  \\
 \hline
\end{tabular}%
%}
\caption{\label{casualty} Expected Saved Lives with Drone Network}
\end{table}
\vspace{-0.16in}

Table \ref{casualty} shows that, depending on the survival function, the number of {\it additional lives} saved thanks to the \ML{drone} network would vary between \ML{17 and 33} in Virginia Beach over a one-year period of time.
This is quite a striking result and is most likely an underestimate of the actual number of lives that could be saved with the drone network. Indeed, the above results only take into account the benefits of the drone-based reduced response time for the estimated percentage (15\%) of opioid overdoses leading to a cardiac arrest. It is reasonable to assume that the reduced response time will also affect the outcomes of the overdoses not leading to a cardiac arrest, as underlined by \citeauthor{ornato2020feasibility} \cite{ornato2020feasibility} 
who state that: "{\it
%one minute saved increases the chance of survival by 10\%
Every minute that goes by before paramedics or others can attempt resuscitation, an opiate overdose victims’ chance of survival decreases by 10 percent}". 
%\cite{ornato2020feasibility}.
%\cite{2019Burroughs,ornato2020feasibility}.
%Using these statistics, the number of saved lives 
%\cite{ornato2020feasibility}: "When used medicinally opioids have a profound analgesia effect. In abuse they are used for their euphoric effect, and in overdose may result in respiratory depression and arrest with subsequent cardiac arrest without rapid intervention. Once this occurs, the chances of survival decline as much as 10\% per minute". 
%Similarly, \citeauthor{WILDE} \cite{WILDE} "Do Emergency Medical System Response Times Matter for Health Outcomes"P.5: 1 min response time increases mortality by 8-17\%.}

%%%%%%%%%%%%%%%%%%%%%%%%%%%%%
%%%%%%%%%%%%%%%%%%%%%%%%%%%%%%%%%%%%
%%%%%%%%%%%%%%%%%%%%%%%%%%%%%%%%%%%%
%%%%%%%%%%%%%%%%%%%%%%%%%%%%%%%%%%%%
\vspace{-0.075in}
\subsubsection{Impact on Quality-adjusted Life Year and Cost Analysis} \label{QALY}
\vspace{-0.1in}
The quality-adjusted life year (QALY) is a concept commonly used in healthcare economic analysis to evaluate how a healthcare issue impacts a survivor's quality and duration of future life (see, e.g., \cite{BogleQALY,SASSI}). 
% QALY represents a survivor's lifespan, adjusted for the quality of each year. 
A QALY equal to 1 is indicative of one year in perfect health. 
As shown in \eqref{qaly_formula}, the total QALY (T-QALY) for a patient surviving a healthcare problem is the sum of the discounted QALY of each year during one's mean life expectancy after the healthcare incident \cite{BogleQALY}:
\vspace{-0.07in}
\begin{equation}
%\label{1drone_cons_r1}
%\begin{align}
T-QALY  = \sum_{t = 1 }^{T} \frac{\alpha t}{(1 + c)^{t}} \ \ , \label{qaly_formula}
%\label{1drone_bilinear}
%\end{align}
\end{equation} 
where $T$ is the number of years of remaining life expectancy, $\alpha \in [0,1]$ is a coefficient accounting for the reduced quality of life consecutive to the healthcare incident, and $c$ is a discount rate reflecting that people prefer good health sooner than later (i.e., a QALY in earlier years is more valuable than in later~ones).

To conduct the QALY analysis, we assume, as in \cite{BogleQALY}, that a patient surviving an overdose-associated OHCA has a mean life expectancy of 11.4 years ($T$ = 11.4), that each year corresponds to $0.85$ QALY ($\alpha = 0.85$), and that the discount rate $c$ is 3\%. Using \eqref{qaly_formula}, the T-QALY is equal to 8.47 years  for each OHCA survivor (see Table \ref{qaly} in Appendix~\ref{DATA-COST}). 
%%%TABLE WAS HERE

Table \ref{add_qaly} presents the {\em additional QALY} gained by using drones instead of ambulances for the survival functions \cite{Bandara:2014}, \cite{DEMAIO}, and \cite{Chanta2014} (see Table \ref{survival}).  The additional QALY attributable to the drone network for the overdose-associated OHCAs in Virginia Beach over one year is very high, varying between \ML{143 and 279} years among the three survival functions. 

\vspace{-0.03in}
\begin{table}[H]
\centering
 \resizebox{\columnwidth}{1cm}{%
\begin{tabular}{P{1.5cm} | P{3cm} | P{2.4cm} | P{2.4cm} |P{3.5cm} | P{2.9cm}}
 \hline
\multirow{2}{*}{T-QALY} & 
Survival & \multicolumn{2}{c|}{Number of Survivors} & Additional Survivors & Additional QALY \\ 
%& & \multicolumn{2}{c|}{} & & \\
\cline{3-4} & Function  & Drone Network & EMS Network & with Drone Network& over one year\\
 \hline
\multirow{3}{*}{8.47}
%  & Weaver et al. \cite{weaver1986factors}  & 43\% / 36 & 20\%  / 16 \\
 & 	Bandara et al. \cite{Bandara:2014} & 42 & 9 & 33 & 279 \\
 & De Maio et al. \cite{DEMAIO} & 21 & 4 & 17 & 143 \\
 &  Chanta et al. \cite{Chanta2014} & 34 & 9 & 25 & 211  \\
 \hline
\end{tabular}%
}
 \vspace{-0.06in}
\caption{\label{add_qaly} Additional QALY with Drone Network over one Year (from 2018Q3 to 2019Q2)}
\end{table}
\vspace{-0.15in}
A drone and the accompanying DB are estimated to cost $\$15,000$, to have a four-year lifespan, and to have an annual maintenance cost of $\$3,000$ \cite{BogleQALY}. Using these statistics (see Table \ref{dronecost} in Appendix \ref{DATA-COST}), the total discounted costs of the eleven drones used in the base scenario amount to  $\$287,664$ (i.e., using a 3\% discount rate for the annual maintenance cost) and the drone-based network allows a reduction in the average response time of about 7 minutes and 16 seconds as compared to the ambulance network. 
That is quite a difference with the option of buying ambulances, each costing between \$150,000 to \$200,000, which, as stated by \citeauthor{ornato2020feasibility}. \cite{ornato2020feasibility} would cost “{\it millions and millions of dollars just to reduce the response time by a minute}” (i.e., from eight to seven).

%%%%%%%%%%%%%%%%%%%%%%%%%%%%%%%%%%%%%%%%%
%%%TABLE WAS HERE
Assuming that the number of overdoses in Virginia Beach remains stable, the total additional QALY in four years (i.e., lifespan of a drone) attributable to the drone network reaches \ML{1116, (resp., 572 and 844)} years with the survival function \cite{Bandara:2014} (resp., 
 \cite{DEMAIO}, and \cite{Chanta2014}), which in turn implies that the proposed drone network only costs \ML{\$257 (resp., \$503 and \$341)} per incremental QALY.
As a benchmark, \citeauthor{BogleQALY} \cite{BogleQALY} report a \$3143 cost per incremental QALY for a network of drones delivering defibrillators to OHCAs in Durham, NC. More generally, it is considered that a medical intervention with a \$50,000 per QALY ratio is cost-effective \cite{NEUMANN}.

%%%%%%%%%%%%%%%%%%%%%%%%%%%%%
\vspace{-0.03in}
\begin{table}[H]
\centering
%\resizebox{\columnwidth}{1.8cm}{%
\begin{tabular}{P{3.5cm} | P{3.5cm} | P{3.5cm} | P{3.5cm} }
 \hline
\multirow{2}{*}{Drone Network Cost} & \multirow{2}{*}{Survival Function} & Total Additional & Cost per\\
& & QALY in Four Years &Additional QALY \\
 \hline
\multirow{3}{*}{\$287,664} 
 & Bandara et al. \cite{Bandara:2014} & 1116 & \$257 \\
 & De Maio et al. \cite{DEMAIO} & 572 & \$503\\
 & Chanta et al. \cite{Chanta2014} & 844 & \$341 \\
 \hline
\end{tabular}%
%}
\vspace{-0.06in}
\caption{\label{QALY_4yr} Cost Analysis for Drone Network per Incremental QALY ($p$=11)}
\end{table}

%%%%%%%%%%%%%%%%%%%%%%%%%%%%%%%%%%%%%%%%%%%%%%
%%%%%%%%%%%%%%%%%%%%%%%%%%%%%%%%%%%%%%%%%%%%%%%
\vspace{-0.275in}
\subsection{Comparison with Alternative Approaches} \label{sec_baseline}
\vspace{-0.085in}
\ML{The results in the preceding section show that the drone network considerably reduces the average response time and increases the chance of survival of overdosed patients as compared to the ambulance-based EMS network in Virginia Beach. This is due, in part, to the fact that drones are not hindered by traffic and can reach the overdose incidents faster than ambulances. 
We investigate next to which extent the proposed drone network modelling and algorithmic technique contributes to the reduction in response  time. We compare the results of our approach to those obtained with two constructive greedy heuristics (Section \ref{Heuristics}) and with a drone-network optimization model from the literature \cite{BOCHAN} (Section \ref{sec_existing_model}).
}

\vspace{-0.075in}
\ML{
%%%%%%%%%%%%%%%%%%%%%%%%%%%%%%
\subsubsection{Comparison with Constructive Greedy Heuristics}
\label{Heuristics}
\vspace{-0.075in}
In this section, we implement two constructive greedy heuristics (see, e.g., \cite{Dziuba2023,2021Gwalani,2018Hoekstra} 
%2019Argenziano,
and compare the networks obtained with these heuristics with the one constructed with our approach.

\vspace{0.05in}
\noindent
{\bf Description}:
%\noindent
The first heuristic, referred to as \textbf{H-OTR}, focuses on the number of overdoses at each OTR location and ranks the locations in a decreasing order by overdose occurrence (arrival) rate. The location ranked one is the one which has the largest occurrence rate of overdoses. We create two lists including respectively the OTR locations and the DB candidate locations and we update the lists at each step.

The heuristic \textbf{H-OTR} proceeds as follows. At each step, a new DB is opened. At first, we consider the OTR location with the largest overdose occurrence rate (i.e., ranked one) and we open a DB at the candidate location which is closest to the highest-ranked OTR location. This OTR location and the selected DB location are then removed from their respective lists. At each subsequent step, we consider the highest-ranked OTR location remaining in the list and open a DB at the closest candidate location. The addition of DBs stops when the number of DBs that can be opened is attained. If some drones are left, i.e., instances in which the number of drones is larger than the number of DBs that can be opened, we place a second drone at the DB(s) opened first.
The pseudo-code of \textbf{\textbf{H-OTR}} is in Appendix~\ref{Constructive_baseline}. } 

\ML{
The second heuristic, referred to as \textbf{H-DB}, focuses on the proximity of the DB candidate locations from the overdose incidents. Instead of operating on a ranking of the overdose locations as in \textbf{H-OTR}, the \textbf{H-DB} heuristic ranks the DB candidate locations. The DB candidate location ranked first is the one with the largest potential demand, i.e., largest number of overdose incidents, within the catchment area of a drone positioned at this DB. As before, we create and update at each step of the heuristic two lists including respectively the OTR locations and the DB candidate locations.
The heuristic \textbf{H-DB} proceeds as follows. At each step, a new DB is opened. 
At first, we consider the DB candidate location with the largest demand and we open a DB there. After each iteration, the selected DB and the OTR locations closest to it are removed from the candidate sets considered at the subsequent iterations. A new ranking is then recalculated for the DB candidate locations based on the remaining (udpated) list of OTR locations. The greedy addition of DBs stops when the number of DBs that can be opened is attained. 
%Similar to heuristic 
As for \textbf{\textbf{H-OTR}}, if the number of available drones is larger than the number of DBs, we place a second drone at the DB(s) opened first, i.e., those with the highest arrival rate. 
The pseudo-code of \textbf{H-DB} is in Appendix \ref{Constructive_baseline}. 

\vspace{0.075in}
\noindent
{\bf Results and Comparison}:
We now compare the network built with our approach and model \textbf{R-MILP} with the two networks built using the heuristics \textbf{H-OTR} and \textbf{H-DB} respectively. We proceed as follows. 
%In the next paragraph, we describe the procedure used to obtain the results. 
First, we run the heuristics \textbf{H-OTR} and \textbf{H-DB} and solve \textbf{R-MILP} to obtain the location of the opened DBs and the number of drones positioned at each DB with each approach. Second, we run the simulation 100 times for each of the three networks built with \textbf{H-OTR}, \textbf{H-DB}, and \textbf{R-MILP}. For each, we record two types of response times from the simulation. First, we record the average network-wide response time $\bar{R}$ (i.e., average across all OTRs). Second, for each OTR $i$, we record the average response time $R_i$ across the 100 simulations. 
The results described next attest the superior performance of the network built with our model \textbf{R-MILP}. The advantages of the \textbf{R-MILP}-based network over those constructed with the heuristics are twofold. First, the \textbf{R-MILP}-based network reduces significantly the average response time $\bar{R}$.
Table \eqref{n_hdb_compare} 
shows that the average response time $\bar{R}$ of the \textbf{R-MILP} network is 36.02\% and 46.78\% lower than the average response time with the \textbf{H-DB}- and \textbf{H-OTR}-based networks, respectively.
}

\begin{table}[ht]
\centering
\setlength\extrarowheight{1pt}
\ML{
\begin{tabular}{P{2cm}|P{2cm}|P{2cm}|P{2cm}| P{2cm}|P{2cm}}
 \hline
\multirow{3}{*}{Training Set} & 
%\multirow{2}{*}{Training Set}
\multicolumn{3}{c|}{Average Network-wide}&
\multicolumn{2}{c}{\textbf{R-MILP} Time Reduction} \\ %\cline{2-4}  
%\multirow{2}{*}{} 
 & \multicolumn{3}{c|}{Response Time $\bar{R}$ (minutes)}&
\multicolumn{2}{c}{with respect to:} \\
\cline{2-6} 
& \textbf{H-DB} & \textbf{H-OTR} & \textbf{R-MILP} & \textbf{H-DB} &  \textbf{H-OTR} \\
 \hline
2018Q2 & 2.29 & 2.75& 1.49 & 34.93\% & 45.82\% \\
2018Q3 & 2.29 & 3.70 & 1.58 & 31.00\% & 57.30\%\\ 
2018Q4 & 2.49 & 2.31 & 1.54 & 38.15\% & 33.33\%\\
2019Q1 & 2.48 & 2.72 & 1.50  & 39.52\% & 44.85\% \\
\hline
Average & 2.39 & 2.87& 1.53 & 36.02\% & 46.78\%\\
 \hline
\end{tabular}
}
\caption{\label{n_hdb_compare} Average Response Time $\bar{R}$ for \textbf{H-DB}, \textbf{H-OTR}, and \textbf{R-MILP} Networks ($q = 10, p =~11$). }
\end{table}

% \begin{table}[H]
% \centering
% \setlength\extrarowheight{1pt}
% \ML{
% \begin{tabular}{P{3cm}|P{4.1cm}|P{4.1cm}| P{3cm}}
%  \hline
% \multirow{2}{*}{Training Set} & \multicolumn{2}{c|}{Average Network-wide Response Time (minutes)}&  \\ \cline{2-4}  
% \multirow{2}{*}{} & \textbf{\text{H-OTR}} 
% & R-MILP & Time Reduction\\
%  \hline
% 2018Q2 & 2.75 & 1.49 & 45.82\%  \\
% 2018Q3 & 3.70 & 1.58 & 57.30\% \\ 
% 2018Q4 & 2.31 & 1.54 & 33.33\% \\
% 2019Q1 & 2.72 & 1.50  & 44.85\% \\
% \hline
% Average & 2.87 & 1.53 & 46.78\% \\
%  \hline
% \end{tabular}
% \caption{\label{n_hotr_compare} Average Response Time  $\bar{R}$ for Networks Built with \text{H-OTR} and \text{R-MILP} ($q = 10, p = 11$).}
% }
% \end{table}

%\vspace{-0.15in}
\noindent

\ML{
Second, the {\bf R-MILP}-based network also reduces considerably the upper tail of the distribution of  the average OTR response time $R_i$. Figure \ref{hist_heuristic_2018Q2} (in Appendix \ref{Constructive_baseline2}) displays the distribution of the average response time $R_i$ for the {\bf R-MILP}-, \textbf{H-OTR}-, and \textbf{H-DB}-based networks.  As it can be clearly seen, the {\bf R-MILP}-based network does not only lessen the average response time (red solid line), but also significantly reduces the gap, i.e., time difference, between the average response time and the 90th quantile of the response time (red dashed line). The narrower gap is indicative of a more equitable EMS system since the response time exhibits less variability among all OTRs.
}

As few seconds can be the difference between life and death in the context of opioid overdoses \cite{ornato2020feasibility}, the value of our model
%a more advanced model like ours 
becomes evident and explains why Buckland et al. \cite{BucklandDesignConsider} urge, in regards to using drones for medical emergencies, to develop advanced resource allocation and optimization algorithms.

\vspace{-0.08in}
\subsubsection{Comparison with Alternative Drone Delivery Model} \label{sec_existing_model}
\vspace{-0.08in}
%%%%%%%%%%%%%%%%%%%%%%%%%%%%%%%%%%%%%%%%%%%%%%
To show the added benefit of our modeling approach, referred to as \textbf{R-MILP}, we build a benchmark network using the optimization model proposed in \cite{BOCHAN} in which drones are utilized to deliver automated external defibrillators to people having out-of-hospital cardiac arrests. We refer thereafter to this model with the acronym \textbf{BM} short for benchmark model. At a high level, the model \textbf{BM} seeks to minimize the total number of drones to be deployed while ensuring that the average response time of the drone network is at least $\gamma$ minutes shorter (i.e., $\gamma$ is a parameter fixed a priori) than the current EMS system. Using the data from the second quarter of 2018, we set $\gamma=3$ minutes -- which is the time improvement targeted in \cite{BOCHAN} -- in model \textbf{BM}.
We carry out the simulation process 100 times using the optimal DB locations and number of drones determined by the optimal solutions of the two models \textbf{BM} and \textbf{R-MILP}. For each simulation run in each model, we record the mean network-wide response time $\bar{R}$ across all OTRs.
Table \ref{bm_compare} reports the statistics for the network-wide response time $\bar{R}$ based on 100 simulations. 

\begin{table}[H]
\centering
\setlength\extrarowheight{1pt}
\ML{
\begin{tabular}{P{3cm}|P{4.1cm}|P{4.1cm}| P{3cm}}
 \hline
\multirow{2}{*}{Summary Statistics} & \multicolumn{3}{c}{Response Time (minutes)}  \\ \cline{2-4}  
\multirow{2}{*}{} & BM 
& R-MILP & Time Reduction\\
\hline 
5th Percentile & 5.48 & 0.80 & 85\%  \\
\hline
Average & 6.67 & 1.49 & 78\% \\ 
\hline
95th Percentile & 8.03 & 2.19 & 73\%
\\ 
\hline
\end{tabular}
\vspace{-0.04in}
\caption{\label{bm_compare} Response Times $\bar{R}$ for network BM and R-MILP ($q = 10, p = 11$) using 2018 Q2 data.}
}
\vspace{-0.172in}
\end{table}

Clearly, the network based on the \textbf{R-MILP} model is much superior to the one based on the BM model as the average response time with the \textbf{R-MILP} model is 78\% lower than with the BM model. The average response time with our model \textbf{R-MILP} is one minute and 29 seconds while the average response-time with the BM model is six minutes and 40 seconds. Similarly, the 5th and 95th response time percentiles with the \textbf{R-MILP} model are respectively 85\% and 73\% lower than those with the BM model.
The results in Table \ref{bm_compare} are unequivocal about the superiority and the advantages offered by the model constructed with our model \textbf{R-MILP}. The response time gain is crucial in view of how a few seconds can be the difference between life and death in this context.

\vspace{-0.125in}
\subsection{Computational Efficiency} \label{sub_sec_compute_e}
\vspace{-0.1in}
In this section, we conduct a battery of tests to assess the computational efficiency and tractability of the proposed reformulation and algorithms. 
Section \ref{PER_EVA} %highlights the benefits of the reformulation and algorithmic framework and 
compares the two proposed algorithms and the direct solution of the reformulated problem with {\sc Gurobi}. Performance profile plots \cite{PROFILE} highlight the increased benefits gained with our approaches as the size of the problem and the volume of OTRs increases.
Section \ref{SENS} assess the sensitivity of the proposed method with respect to available resources.

All formulations are coded in Python 3.7 and solved with the {\sc Gurobi} 
%9.1.2 
solver on a Linux machine, with Intel Core i7-6700 CPU 3.40GHz processors and 64 GB installed physical memory.
For each instance, the optimality tolerance is %set to
0.01\%, % for each solver,
the maximum solution time is one hour, and we use one thread only.

%A grid search conducted on the data described in Section revealed \ref{sub_data} that the minimum number of drone bases and the minimum number of drones needed to cover all the opioid overdose requests are respectively equal to 10 and 11. 
%We consider four congestion levels defined in terms of the service time that depends on the traveling times from the MTF and the loading times at the CCP, which are themselves affected by the presence of hostile troops and weather conditions. 

%%%%%%%%%%%%%%%%%%%%%%%%%%%%%%%%%%%%%%%%%%%%%%%
%%%%%%%%%%%%%%%%%%%%%%%%%%%%%%%%%%%%%%%%%%%%%%%%%%%%%%%%%%%%%%%%%%%%%%%%%%%%%%%%%%%%%%%%%%%%%%%
%%%%%%%%%%%%%%%%%%%%%%%%%%%%%%%%%%%%%%%%%%%%%%%
\vspace{-0.1in}
\subsubsection{Computational Efficiency of Scalability} \label{PER_EVA}
\vspace{-0.06in}

We evaluate here the computational efficiency and scalability of the proposed reformulations  \WMRMin{\textbf{L-MILP}, \textbf{R-MILP}} and algorithms with respect to the size of the problem, namely the number $|I|$ of OTRs to which the network must respond.
Using the data described in Section \ref{sub_data} and considering the base network scenario with up to 10 DBs and 11 drones, 
%($(q,p)=(10,11)$),  
we have created ten problem types that differ in the number of OTRs ranging from 50 to 500, by increment of 50: 
$|I| \in \{50,100,150,200,250,300,350,400,450,500 \}$.  
Each instance type is identified by the tuple 
$(q,p,|I|)$ 
%(q,p,|I|)=(10,11,|I|)$ 
and we have generated five problem instances for each instance type. This gives a total of 50 instances which we solve with the following~approaches:
\vspace{-0.1in}

\begin{itemize}
\item {\tt REFO}: 
Direct solution of the \MLB{larger-size MILP problem \textbf{L-MILP} \eqref{LARGE}} with the default settings of the {\sc Gurobi} solver.
\vspace{-0.074in}
\item {\tt OA}: Outer approximation algorithm (\WMRMin{Remark \ref{rem_oa}}) with model \WMRMin{\textbf{L-MILP}}.  
\vspace{-0.074in}
\item {\tt OA-B\&C}:  
Outer Approximation branch-and-cut algorithm 
%(Section \ref{sub_sec_B&C}) 
with model \WMRMin{\textbf{L-MILP}}.
\vspace{-0.074in}
\item {\tt R-OA-B\&C}:  
Outer Approximation branch-and-cut algorithm 
with model \WMRMin{\textbf{R-MILP}}.
\end{itemize}

%with cuts (or valid inequalities) defined presented in Section \ref{sub_sec_valid_e} utilized as {\it user cuts}.}

\vspace{-0.055in}
%Solving the problem instances with the three above approaches permits to assess the contribution of the several algorithmic ingredients proposed. 
The \ML{four} approaches are compared in terms of solution times and the number of instances for which optimality can be proven in one hour.
%Before presenting the results of obtained with our approach, 
We note that none of the 50 problem instances formulated with the base formulation $\textbf{B-IFP}$ can be solved in one hour with the state-of-the-art solver {\sc Baron} specialized for nonconvex MINLP problems.
We use performance profile plots displayed in Figure  \eqref{fig:p_base_pp} to compare the efficiency of the solution methods. The horizontal axis indicates the running time in seconds while the vertical axis represents the number of instances solved to optimality within the corresponding running time. 
The solution time for each instance is given in Table \ref{table_compute} in Appendix \ref{APP-CEFF}. 
\begin{figure}[ht]
\begin{subfigure}[b]{0.5\textwidth}
\includegraphics[width = \textwidth, height = 7cm]{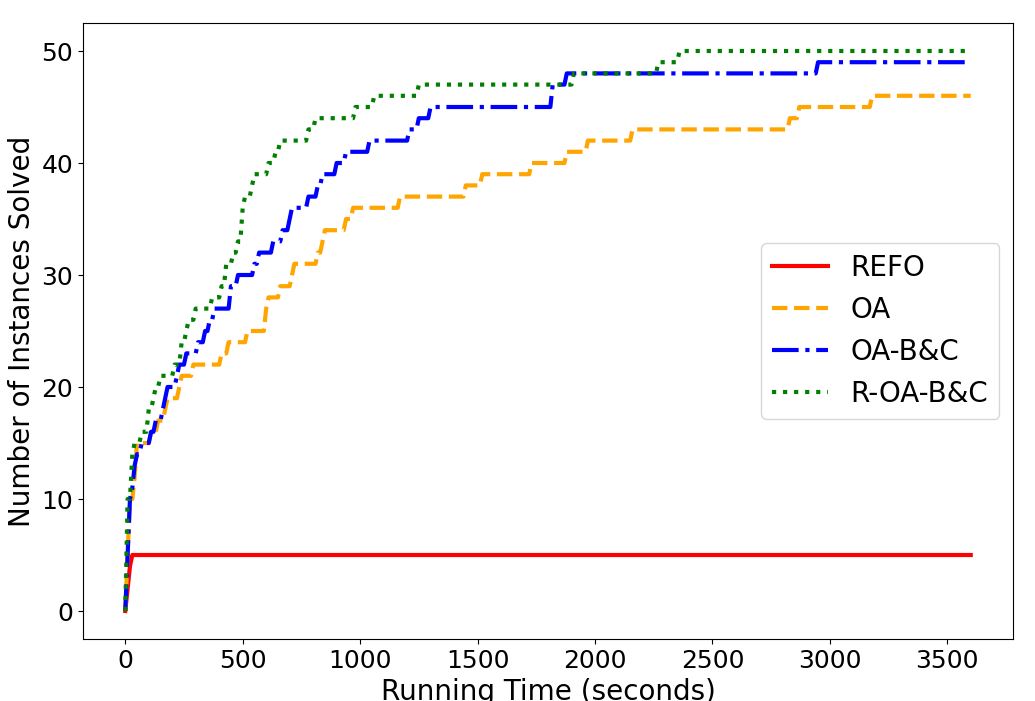} 
\caption{Performance Profile}
\label{fig:p_base_pp}
\end{subfigure}
\hfill
\begin{subfigure}[b]{0.5\textwidth}
\includegraphics[width=\textwidth, height=7cm]{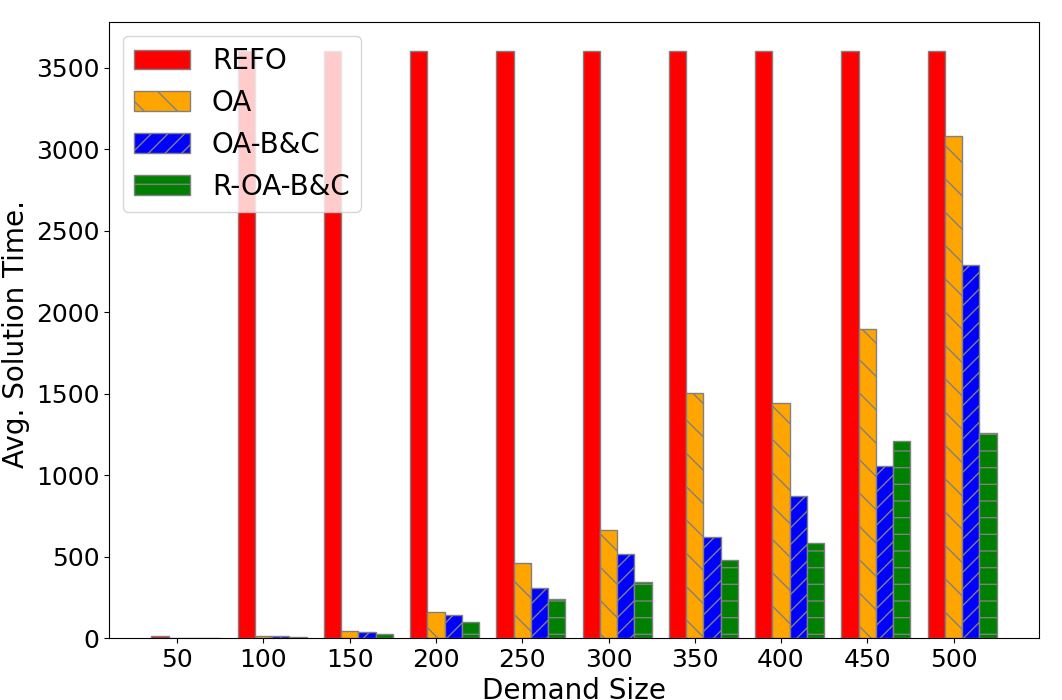}
\caption{Average Solution time}
\label{fig:p_base_avg}
\end{subfigure}
% \captionsetup{justification=centering}
\vspace{-0.25in}
\caption{%(Color Online) 
Computational Efficiency for $(10,11, |I|)$: Performance Profile in Figure \eqref{fig:p_base_pp} and Average Solution Time by Demand Size $|I|$ in Figure \eqref{fig:p_base_avg}}
\label{fig:p_base}
\vspace{-0.1in}
\end{figure}
% \begin{figure}[h] %[!tbp]
% \includegraphics[width = 0.5\textwidth, height = 7cm]{performance_profile.jpg}
% \caption{Performance Profile ($q$ = 10 and $p$ = 11)}
% \label{fig:perf_profile}
% \end{figure}

The performance profile shows clearly  that the proposed methods \ML{{\tt OA}, {\tt OA-B\&C} and {\tt R-OA-B\&C}} are much faster, scale much better, and allow the solution of many more instances than the direct solution method {\tt REFO}. 
%As compared to the direct solution of $\mathbb{R-MILP}$, the outer approximation method {\tt OA} significantly speeds up the solution process. 
While the direct solution approach {\tt REFO} only solves five (i.e. the five smallest instances with $|I|$ = 50 OTRs) of the 50 instances in one hour, \ML{{\tt OA}, {\tt OA-B\&C}, and {\tt R-OA-B\&C} solve and prove the optimality of the solution for respectively 46, 49, and 50 instances.}
Comparing now the algorithms {\tt OA} and {\tt OA-B\&C}, the dynamic incorporation of valid inequalities and optimality cuts leads to further efficiency and scalability gains. This can be seen from the {\tt OA-B\&C} line in the performance plot being above the {\tt OA} line at any time interval, thereby indicating that the {\tt OA-B\&C} method solves more instances to optimality in any amount of time.
% Second, the proposed valid inequalities accelerate the solution process as given almost any running time on the x-axis, {\tt OA-B\&C} consistently solves more instances than {\tt OA}. 
As an illustration, Table \ref{Num_solved} in Appendix shows that {\tt OA-B\&C} solves 96\% of the instances in less than 40 minutes while {\tt OA} needs one hour to solve 92\% of the instances. 

The solution method {\tt R-OA-B\&C} based on the B\&C outer approximation method and the compact MILP reformulation \WMRMin{\textbf{R-MILP}} improves further the solution process over  
{\tt OA-B\&C}. The performance profile of {\tt R-OA-B\&C} is systematically above that of {\tt OA-B\&C}. which shows that, in any considered amount of time, {\tt R-OA-B\&C} is able to solve more instances than {\tt OA-B\&C}, and demonstrates the \MLB{computational benefits provided by the compact MILP reformulation \WMRMin{\textbf{R-MILP}}}. For instance, while {\tt OA-B\&C} solves 64\% of the instances to optimality in 600 seconds, {\tt R-OA-B\&C} does so for 78\% of them in the same time.

Figure \eqref{fig:p_base_avg} shows the average (across five instances) solution times for each instance type and associated size $|I|$. When an instance cannot be solved to optimality within one hour, the solution time is estimated to be 3600 seconds, which explains that for all instances of size $|I| \in [100,500]$ the solution time for {\tt REFO} is 3600 seconds.
Figure \eqref{fig:p_base_avg} highlights the added benefits of {\tt R-OA-B\&C} over {\tt OA-B\&C} and  {\tt OA} for larger instances. 
While {\tt R-OA-B\&C} (green bar), {\tt OA-B\&C} (blue bar), and {\tt OA} (orange bar)
perform similarly for $|I| \in [50, 300]$, {\tt R-OA-B\&C} and {\tt OA-B\&C} solve the instances of larger size $|I| \in [350, 500]$ much faster than  {\tt OA} can do. 
It can be seen that {\tt OA} and {\tt OA-B\&C} take respectively  140\% and 80\% more time than {\tt R-OA-B\&C} to solve the 500-sized instances.
%%%%%%%%%%%%%%%%%%%%%%%%%%%%%%%%%%%
%%%%%%%%%%%%%%%%%%%%%%%%%%%%%%%%%%%
%%%%%%%%%%%%%%%%%%%%%%%%%%%%%%%%%%%
%%%%%%%%%%%%%%%%%%%%%%%%%%%%%%%%%%%
\vspace{-0.08in}
\subsubsection{Sensitivity Analysis with Respect to Network Resources} 
\label{SENS} 
\vspace{-0.1in}
In this section, we evaluate the sensitivity of the solution times with respect to the resources of the network and, in particular the number of DBs ($q$) and drones ($p$). We use the algorithm {\tt \ML{R-OA-B\&C}} since it enjoys the quickest solution time as shown in Section \ref{PER_EVA}.

\vspace{-0.1in}
\begin{figure}[H]
\begin{subfigure}[b]{0.5\textwidth}
\includegraphics[width=\textwidth, height=6.7cm]{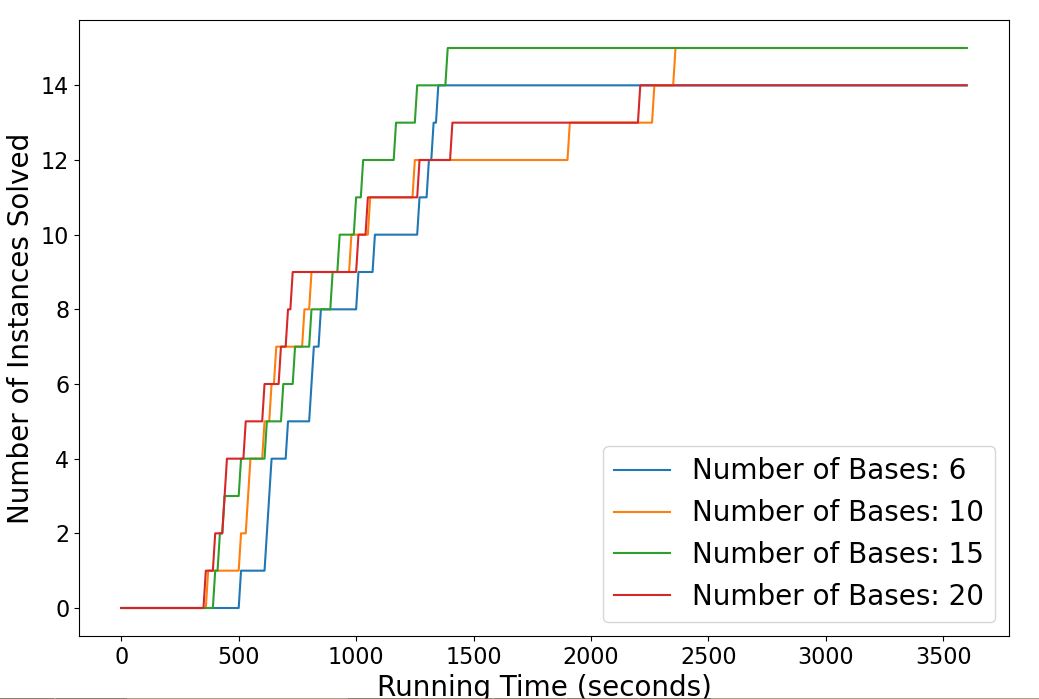}
\caption{Sensitivity to Number $q$ of DBs ($p = 11$)}
\label{fig:p_sense_base}
\end{subfigure}
\hfill
\begin{subfigure}[b]{0.5\textwidth}
\includegraphics[width=\textwidth, height=6.7cm]{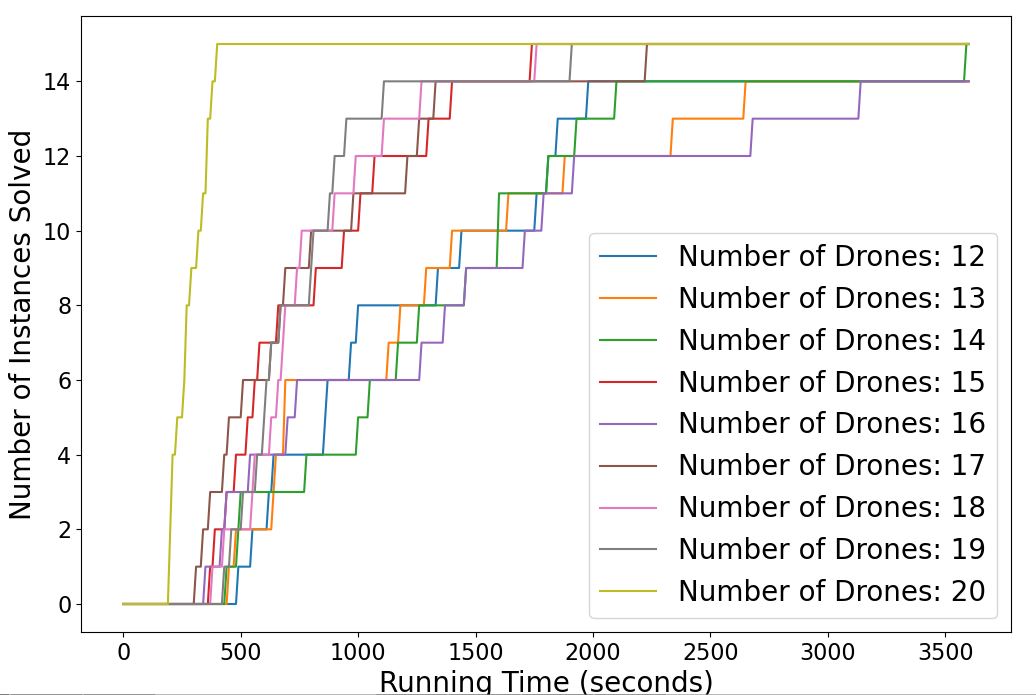}
\caption{Sensitivity to Number $p$ of Drones  ($q = 10$)}
\label{fig:p_sense_drone}
\end{subfigure}
% \captionsetup{justification=centering}
\vspace{-0.25105in}
\caption{%(Color Online) 
Sensitivity with Respect to Number of DBs (Figure \eqref{fig:p_sense_base}) and Number of Drones (Figure \eqref{fig:p_sense_drone}}
\label{fig:p_sense}
\vspace{-0.16in}
\end{figure}

We consider four possible values $q \in \{6, 10, 15, 20 \}$ for the number of DBs and three values $|I| \in \{400, 450, 500\}$ for the demand size (number of OTRs), and generate for each combination of $q$ and $|I|$ five problem instances. We assume that $p=11$ drones are available. This gives 12 instance types $(q,11,|I|)$ for a total of 60 problem instances and we solve each of these with the \ML{{\tt R-OA-B\&C}} algorithm.
%which was shown to be the most efficient in the previous subsection.
Figure \eqref{fig:p_sense_base} shows the performance profile associated to each considered number $q$ of DBs. It appears that the  computational time of the \ML{{\tt R-OA-B\&C}} method is not particularly sensitive to the number of open DBs since the performance profiles for the considered values of $q$ intersect, thereby indicating the solution time is not systematically higher (or lower) for any considered number of DBs.

We proceed similarly to evaluate the sensitivity of the computational time with respect to the number of drones. We assume that the maximum number of DBs that can be open is $q=10$. We consider nine possible number of available drones $p \in \{12, 13,14,15,16,17,18,19,20\}$, three possible demand sizes $|I| \in \{400, 450, 500\}$, and generate five instances for each of the 27 instance types of each type $(10, p, |I|)$ for a total of 135 instances \ML{that} we solve with the \ML{{\tt R-OA-B\&C}} method.
The corresponding performance profile %corresponding to each considered $p$ is displayed 
in Figure \ref{fig:p_sense_drone} shows that the solution process tends to be quicker (i.e. more instances solved in a given amount of time) when the number of available drones is larger. We notice in particular that the average solution time is much lower for $p = 20$ than for any other (smaller) value assigned to $p$. 
%A likely explanation for this is that each DB will automatically house two drones when $q = 10$.

%%%%%%%%%%%%%%%%%%%%%%%%%%%%%%%%%%%%
%%%%%%%%%%%%%%%%%%%%%%%%%%%%%%%%%%%%
%%%%%%%%%%%%%%%%%%%%%%%%%%%%%%%%%%%%
%%%%%%%%%%%%%%%%%%%%%%%%%%%%%%%%%%%%

\vspace{-0.125in}
\section{Conclusions} \label{sec_conclusion}
\vspace{-0.11in}
%Opioid overdose is the primary cause of death for US adults between 25 and 64 years of age,
Opioid use affects about 2 million Americans and cost \$78.5 billion in annual health care expenses \cite{dezfulian2021opioid}. %In this context, 
It has been argued that the drone-based delivery of naloxone has the “{\it potential to be a transformative innovation due to its easily deployable and flexible nature}” \cite{gao2020dynamic}.
The efficacy of naloxone 
%is an antidote to respiratory failure caused by an overdose and its efficacy 
depends on how quickly it is administered, which is a priority of the US Food and Drug Administration~\cite{ornato2020feasibility}. 

This study responds to the above pressing needs and proposes a new queueing-optimization model to design an EMS network in which drones deliver naloxone to overdose victims in a timely fashion. % Naloxone is a medication commonly used for respiratory failure from opioid overdose. It is included as part of an emergency overdose response kit distributed to opioid drug users and emergency responders \cite{gao2020dynamic}: %Arguing for the design of operational drone-delivered naloxone systems, 
%We propose a novel stochastic bilocation-allocation model that can be used at the strategic level to design an EMS network.
The objective is to increase the chance of survival of overdose victims by reducing the time needed to deliver and administer naloxone and is accomplished by determining the location of drones and DBs, the capacity of DBs, and the dispatching of drones to overdose incidents. Some new distinctive features in the model are its survival objective function, the explicit modeling of congestion and decision-dependent (parameter) uncertainty in the service time that affects the performance of the network, and its flexibility as the capacity of DBs is not fixed ex-ante but is instead determined by the model.

Besides modeling, the contributions of this study are threefold.
On the reformulation side, we propose a tractable MILP reformulation of the stochastic network design problem, which, in its original form, is a \ML{fractional integer} optimization problem.
%including polynomial, fractional, exponential, and factorial nonlinearities and which is NP-hard. 
We further demonstrate the generalizability of the reformulation approach and prove that the problem of minimizing the \ML{average network-wide  response time of a collection of interdependent $M/G/K$ queueing systems} in which the capacity $K$ of each system is variable is MILP-representable.  
On the algorithmic side, we \MLB{design an outer approximation branch-and cut algorithmic framework which is a significant upgrade} over the direct solution of the MILP reformulations. \MLB{The algorithm, in particular when used with the reduced-size MILP reformulation, significantly reduces the solution time, increases the number of solved instances,} 
% to optimality, 
and enables the solution of problems of larger size than those encountered in our case study. %based on real-life data.
On the EMS practice side, the tests based on real-life data from Virginia Beach reveal that the use of drones to deliver naloxone in response to overdoses could possibly be a game-changer. 
Besides showing the out-of-sample robustness and applicability of the proposed approach, the tests attest that using the drone network: 
1) the response time is on average \ML{(1 min and 39} sec) close to seven times smaller than the one (9 min and 19 sec) with the ambulance network used in Virginia Beach; 
2) the estimated chance of survival to an overdose is more than \ML{2.73} times larger than the one with the Virginia Beach ambulance network; 
3) many more lives (of overdose victims) could be saved;
%as compared to the Virginia Beach ambulance network;
4) the total QALY per patient is 8.47 years amounting to up to \ML{279} additional per year across all overdose victims in Virginia Beach; 
5) the cost per additional QALY is very low, varying between \ML{\$257 and \$502}; and 
6) one can markedly decrease the response time and increase the survival chance as compared to using heuristics and an alternative model from the literature.

%%%%%%%%%%%%%%%%%%%%%%%%%%%%%%%%
This study showcases the potential of using drones to alleviate the devastating consequences of the opioid overdose crisis and could pave the way for the practical implementation of drone-based EMS systems.
While very promising, we must however be aware of a number of obstacles for the  widespread use of drones to respond to overdoses or other time-critical medical emergencies. These barriers include, for example, regulation, flying condition and zones, safety, data privacy, operations related to the design and maintenance of a medical drone network \cite{Johnson2021impact}. 
While not the focus of this study, future research will be needed to investigate these hindrances as well as user perceptions and acceptance. 
%The removal of some of these barriers could be hastened by the ongoing coronavirus pandemic which has seen health policy makers urged to adopt new and effective healthcare measures (including the use of drones in some countries) to reduce death rate \cite{CIRRIN}.

%\cite{Johnson2021impact}: 
%Current challenges to expanding their use in emergency medicine and emergency medical system (EMS) include regulation, safety, flying conditions, concerns about privacy, consent, and confidentiality, and details surrounding the development, operation, and maintenance of a medical drone network. Future research is needed to better understand end user perceptions and acceptance. 

% \ML{To summarize, our tests demonstrate that method D, by taking advantages of the proposed lazy constraint, OBBT and the valid inequalities all together, solves most of problem instances and provides the fastest solution time when optimality can be reached within 1 hour. }
%%%%%%%%%%%%%%%%%%%%%%%%%%%%%%%%
%%%%%%%%%%%%%%%%%%%%%%%%%%%%%%%%
%%%%%%%%%%%%%%%%%%%%%%%%%%%%%%%%
%%%%%%%%%%%%%%%%%%%%%%%%%%%%%%%%

%%%%%%%%%%%%%%%%%%%%%%%%%%%%%%%%%%%%%%%%%%%%%
%%%%%%%%%%%%%%%%%%%%%%%%%%%%%%%%%%%%%%%%%%%%%%%%%%%%%%%%%%%%%%%%
%%%%%%%%%%%%%%%%%%%%%%%%%%%%%%%%%%%%%%%%%%%%%%%%%%%%%%%%%%%%%%%%
\vspace{-0.07in}
\printbibliography
%%%%%%%%%%%%%%%%%%%%%%%%%%%%%%%%%%%%%%%%%%%%%%%%%%%%%%%%%%%%%%%%
%%%%%%%%%%%%%%%%%%%%%%%%%%%%%%%%%%%%%%%%%%%%%%%%%%%%%%%%%%%%%%%%
%\bibliographystyle{informs2014}
%\bibliography{main}
%%%%%%%%%%%%%%%%%%%%%%%%%%%%%%%%%%%%%%%%%%%%%%%%%
\newpage
\setcounter{page}{1}
\begin{appendices}

\ML{\section{Abbreviations}
\label{sec:abbre}
\begin{itemize}
  %  \item [AED:] Automated external defibrillator.
    \item [DB:] Drone base.
    \item [EMS:] Emergency medical services.
    \item [OHCA:] Out-of-hospital cardiac arrest.
    \item [OTR:] Overdose-triggered request.
    \item [QALY:] Quality-adjusted life year.
\end{itemize}

%%%%%%%%%%%%%%%%%%%%%%%%%%%%%%%%%
\section{Notations}
\label{sec:notations}
\addcontentsline{toc}{section}{Appendix A. Notations}

\subsection{Sets and Indices}
\begin{itemize}
	%IEEEdescription}[\IEEEusemathlabelsep\IEEEsetlabelwidth{$V_1,V_2,V_3$}]
	\item[$i \in I$:] Index and set of OTR locations.
	\item[$j \in J$:] Index and set of candidate DBs.
	\item[$i \in I_j$:] Index and set of OTRs that are within the catchment area of a drone at DB $j$. 
	\item[$j \in J_i$:] 
	Index and set of DBs that can cover OTR $i$.

    	\item[$m \in \{1,\ldots, M\}$:] 
	Index and set of drones.

	%\item[$m \in \mathbf{M}$:] Index and set of drones. %at a drone base.
\end{itemize}
%%%%%%%%%%%%%%%%%%%%%%%%%%%%%%%%%%%%%%%%%%%
%\textit{B. Parameters and Constants}
\vspace{-0.1in}
\subsection{Parameters and Constants}
\begin{itemize}
	%\item[$a_{ij}$:] Boolean parameter indicating if demand $i$ is within the flight radius of drone base $j$ ($a_{ij} = 1$) or not ($a_{ij} = 0$). 
	\item[ $A_j$:] Vector of parameters $(a_j, b_j, c_j)$ representing the coordinates of DB $j$ in the earth-centered, earth-fixed coordinate system.
	\item [$A_i$:] Vector of parameters $(a_i, b_i, c_i)$ representing the location of OTR $i$ in the earth-centered, earth-fixed coordinate system.
	\item[$M$:] Maximum number of drones that can be stationed at any DB.
	
	\item[$d_{ij}$:] Distance between OTR $i$ and candidate DB $j$.
	\item[$v$:] Flight speed of drones.
	%It is a weighted average of $s_{ij}$, depending on the allocation decision variable $y_{ij}$.
	\item[$\lambda_i$:] OTR arrival (occurrence) rate from  at  $i$.

	\item[$p$:] Maximal number of available drones.
	\item[$q$:] Maximal number of drone bases that can can be established. % within the network.
	\item[$\beta$:] Coefficient modulating the travel speed to and from an OTR location. 
	\item[$r$:] Flight radius of a drone.
	
\end{itemize}
%%%%%%%%%%%%%%%%%%%%%%%%%%%%%%%%%%%%%%%%%%%%%%%%%%%%%%%%
%\textit{C. Decision Variables}
\subsection{Decision Variables}
\begin{itemize}
	\item[$x_j$:] Binary variable determining if a DB is open at candidate location $j$ ($x_j = 1$) or not ($x_j = 0$).
	\item[$y_{ij}$:] Binary variable defining if  OTR at location $i$ is assigned to open DB $j$ ($y_{ij} = 1$) or not ($y_{ij} = 1$).
    \item[$\gamma_j^{\WMR{(m)}}$:] Binary variable indicating the number $m$ of drones deployed at  DB $j$: $\gamma_j^{\WMR{(m)}} = 1$ if and only if $m$ drones are positioned at DB $j$.
    \item[$K_j$:] General integer variable defining the number of drones stationed at DB $j$.
    \item[$\eta_j$:] Arrival rate of OTRs serviced by DB $j$.
    \item[$U_j^{\WMR{(m)}}$:] Auxiliary variable introduced to reformulate the fractional terms of the objective function. 
    %\item[$V_j$:] Auxiliary variable introduced to linearize fractional terms.	
    \item[$\mu_{ij}^{\WMR{(m)}}$:] Auxiliary variable introduced to linearize the bilinear term $U_j^{\WMR{(m)}} y_{ij}$.
    \item[$z_{jit}^{}$:] Auxiliary variable introduced to linearize the bilinear term $y_{ij}y_{lj}$.
    \item[$\tau_{jil}^{}$:] Auxiliary variable introduced to linearize the bilinear term $U_{j}^{\WMR{(2)}}z_{jil}^{}$.
    \item[$\omega_{ij}^{\WMR{(m)}}$:] Auxiliary variable introduced to linearize the bilinear term $\gamma_j^{\WMR{(m)}}\mu_{ij}^{\WMR{(m)}}$.
\end{itemize}

%%%%%%%%%%%%%%%%%%%%%%%%%%%%%%%%%%%%%%
\subsection{Random Variables} \label{nota_rv}
\begin{itemize}
	\item[$S_{ij}$:] Service time for OTR at $i$ if serviced with a drone stationed at DB $j$.
	%	\item[$s_{ij}^2$:] \ML{Drone service time squared at base $j$ to service demand at location $i$.}
	\item[$S_{j}$:] Service time of a drone at DB $j$.
	\item[$Q_{j}$:] Queueing delay time for DB $j$.
	\item[$R_{i}$:] Response time between reception of the delivery request and arrival of a drone for an OTR at $i$.
%	\item[$\bar{R}$:] Average response time for an OTR among all demands in $I$.
	\item[$\alpha_i$:] On-scene service time (e.g. naloxone toolkit unloading time) for location $i$.
	\item[$\epsilon_i$:] Drone reset time needed to recharge and load new naloxone toolkit for location $i$.
	\item[$\xi_i$:] Non-travel drone service time equal to $\alpha_i + \epsilon_i$ for OTR at location $i$.
 
\end{itemize}
}

%%%%%%%%%%%%%%%%%%%%%%%%%%%%%%%%%%
%%%%%%%%%%%%%%%%%%%%%%%%%%%%%%%%%%
\vspace{0.1in}
\section{Proofs}\label{PRO}
%%%%%%%%%%%%%%%%%%%%%%%%%%%%%%%%%%%%%%%%%%%%
\subsection{Proof of Proposition \ref{thm_avg_resp}} \label{AVE-PR1}

{\it \underline{\bf{Proposition \ref{thm_avg_resp}}}:
The functional form of the average response time is a fractional expression with nonlinear numerator and denominator:
\begin{equation}
\bar{R} = \sum_{i \in I} \sum_{j \in J_i} \left[ \frac{\eta_j^{K_j} \mathbb{E}[S_j^2] \mathbb{E}[S_j]^{K_j-1}}{2(K_j-1)! (K_j-\eta_j \mathbb{E}[S_j])^2 \big[ \sum_{n = 0}^{K_j-1} \frac{(\eta_j \mathbb{E}[S_j])^n}{n!} + \frac{(\eta_j \mathbb{E}[S_j])^{K_j}}{(K_j - 1)!(K_j - \eta_j \mathbb{E}[S_j]}\big]} + \frac{d_{ij}}{v}\right]\frac{\lambda_i y_{ij}}{\sum_{l \in I} \lambda_l} \ . \notag
\end{equation}
}

\noindent
\begin{proof} %{Proof of Theorem \ref{thm_avg_resp}}
The average response time is the weighted average of the expected response times for each OTR $i$. Therefore, we have: 
\begin{equation}
\label{R1}
\bar{R} = \sum_{i \in I} \frac{\lambda_i}{\sum_{l \in I} \lambda_l} \mathbb{E}[R_i]    
\end{equation}
Using the definition \eqref{e_r_i} of $\mathbb{E}[R_i]$,  \eqref{R1} becomes:
\begin{align}
\bar{R} \; = \; \sum_{i \in I} \frac{\lambda_i}{\sum_{l \in I} \lambda_l} \sum_{j \in J_i} \left(\mathbb{E}[Q_j] + \frac{d_{ij}}{v}\right)y_{ij} \; 
= \; \sum_{i \in I} \sum_{j \in J_i} \label{r_bar_q_j}   \left(\mathbb{E}[Q_j] + \frac{d_{ij}}{v}\right) \frac{\lambda_i y_{ij}}{\sum_{l \in I} \lambda_l} 
\end{align}
Now expanding $\mathbb{E}[Q_j]$ using  \eqref{e_q_j}, we obtain the expression given in Proposition \ref{thm_avg_resp}.
\hfill$\Box$ 
\end{proof}

%%%%%%%%%%%%%%%%%%%%%%%%%%%%%%%%%%%%%%%
\subsection{Proof of Remark \ref{PROP1}} \label{REM1}
{\it \underline{\bf{Remark \ref{PROP1}}}: %\newline
Problem $\mathbf{B-IFP}$ is a \ML{nonconvex} %an NP-hard 
optimization problem in which:
\newline 
(i) The objective function is neither convex, nor concave: the denominator and the numerator in each ratio term of the objective function are nonconvex functions; 
(ii) The ratio terms can involve division by 0 and be indeterminate;
(iii) Any ratio term related to a 
location where no DB is set up in undefined;
(iv) There is a mix of binary and bounded general integer variables;
% \newline
% \MLR{(v) The nonlinear strict inequality constraints \eqref{steady-state} include fractional and polynomial terms, and define a nonconvex feasible area.} \WM{This seems trivial since by just canceling the term we can get a linear constraint}
(v) The continuous relaxation of $\mathbf{B-IFP}$ is a nonconvex~problem. 
}

\noindent
%{\bf Proof:}
\begin{proof} %{Proof of Proposition \ref{PROP1}}
%It has been shown \cite{28} that the unconstrained 0-1 linear- fractional problem is NP-hard when the objective function sums two or more ratio terms. The NP-hardness of $\mathbf{B-IFP}$ follows immediately since it adds other complexity sources (constraints, nonlinear denominator and numerator in ratios, etc.). 
\newline
(i) The  numerator and denominator include polynomial and exponential terms. 
The denominator has also factorial and fractional terms. Both are nonconvex functions. 
It can be easily shown that the hessian matrix of the fractional objective function is not positive semidefinite, nor negative semidefinite.
\newline
(ii) Any term $(K_j- \sum_{l \in I_j} \lambda_l y_{lj}\mathbb{E}[S_{lj}])$ can be equal to 0, which leads to a division by 0, if 
%1) the service rate is equal to the arrival rate at a DB $j$ or 
%2) 
no DB is set up at location $j$, thereby implying $K_j=y_{ij}=0, \forall i \in J_i$.
\newline
(iii) If no DB is set up at the potential location $j$, no drone can be stationed at $j$, which implies that $K_j=0$. In that case, the denominator includes a term involving taking the factorial of a negative number:  
$(K_j-1)! = (-1)!$ % which is not defined. 
\newline
(iv) Obvious: see constraints \eqref{binary}-\eqref{k_j_integer}.
% \newline
% (v) The right-hand side is fractional with a bilinear numerator and linear denominator.
\newline
(v) See part (i).
\end{proof}
\newpage
\subsection{Proof of Proposition \ref{T1}} \label{ATH40}

{\it \underline{\bf{Proposition \ref{T1}}}:
Let $\gamma_{j}^{\WMR{(m)}} \in \{0, 1\},  j \in J, m=1,\ldots,M$.  
%\newline
The fractional nonlinear binary~problem 
\vspace{-0.075in}
\begin{subequations}
%\label{F-R-BFP}
\begin{align}
\mathbf{R-BFP}: & \; \min	    
\sum_{i \in I} \sum_{j \in J_i} \frac{y_{ij}d_{ij} \lambda_i}{v \sum_{l \in I} \lambda_l} \; + \; \sum_{i \in I} \sum_{j \in J_i} \sum_{m=1}^{M} \ \frac{y_{ij}\lambda_i}{\sum_{l \in I} \lambda_l} 
\notag \\ % \label{OBJ2} \\ 
& \hspace{-0.8in} 
\left[ \frac{\gamma_j^{\WMR{(m)}} \sum_{l \in I_j} \lambda_l y_{lj} \mathbb{E}[S_{lj}^2] (\sum_{l \in I_j} \lambda_l y_{lj} \mathbb{E}[S_{lj}])^{m-1}}{2(m-1)! (m-\sum_{l \in I_j} \lambda_l y_{lj} \mathbb{E}[S_{lj}])^2 \big[ \sum_{n = 0}^{m-1} \frac{(\sum_{l \in I_j} \lambda_l y_{lj} \mathbb{E}[S_{lj}])^n}{n!} + \frac{(\sum_{l \in I_j} \lambda_l y_{lj} \mathbb{E}[S_{lj}])^{m}}{(m - 1)!(m - \sum_{l \in I_j} \lambda_l y_{lj} \mathbb{E}[S_{lj}]}\big]} \right] \notag\\ 
\text{s.to} \; & \eqref{assignment}- \eqref{eq_q} ; \eqref{binary} \notag \\ %-\eqref{y_binary}
& \sum_{i \in I_j} \lambda_i y_{ij}\mathbb{E}[S_{ij}]   \le \sum_{m=1}^{M} m  \gamma_j^{\WMR{(m)}} - \epsilon , \ \ j \in J \notag \\ %  \label{steady-state-2}\\
& x_j \le \sum\limits_{m=1}^{M}  m  \gamma_j^{\WMR{(m)}} \le M x_j, \ j \in J \notag \\ % \label{open_drone-2}\\
& \sum\limits_{j \in J}\sum\limits_{m=1}^{M} m\gamma_j^{\WMR{(m)}} = p \notag \\ %  \label{eq_p-2}\\
&\sum\limits_{m=1}^{M} \gamma_j^{\WMR{(m)}} \leq 1 \ , \ j \in J  \notag \\ %  \label{NEW1} \\
& \gamma_j^{\WMR{(m)}} \in \{0,1\}, \ \ j \in J, m =1, \ldots, M \notag  %  \label{binary2} 
\end{align}
\end{subequations}
is equivalent to the fractional nonlinear integer problem  $\mathbf{B-IFP}$. }

\noindent
\begin{proof} 
Part (i): 
We first replace each bounded general integer variable $K_j$ by a weighted sum of $M$ binary variables $\gamma_j^{\WMR{(m)}}, m=1,\ldots,M$ in the constraints.  
For this %variable 
substitution to work, \eqref{SUB1} 
%the following  relationship 
must hold: 
\vspace{-0.1in}
\begin{equation}
\label{SUB1}
K_j := \sum\limits_{m=1}^{M} m  \gamma_j^{\WMR{(m)}}  \ . \end{equation}
This is accomplished in two steps. First, we introduce the linear constraints 
\vspace{-0.1in}
\begin{subequations}\label{SUBSTI}
\begin{align}
\label{SUB2}
\sum\limits_{m=1}^{M} \gamma_j^{\WMR{(m)}} \leq 1 \ &   \\
\label{SUB3}
\gamma_j^{\WMR{(m)}} \in \{0,1\}  \ &  \quad m=1,\ldots,M  
\vspace{-0.2in}
\end{align}
\end{subequations}
for each $j \in J$ (see \eqref{NEW1} and \eqref{binary2}) to ensure that
$\sum_{m=1}^{M} m  \gamma_j^{\WMR{(m)}}$ can take any integer value in $[0, M]$ which is the restriction imposed on each $K_j$ via \eqref{open_drone} and \eqref{k_j_integer}. Accordingly, we substitute \eqref{NEW1}-\eqref{binary2} for \eqref{k_j_integer}
%Second, to ensure the strict equality between the two sides of \eqref{SUB1}, 
and then replace $K_j$ by the right-side term of \eqref{SUB1} in \eqref{steady-state}, \eqref{open_drone}, and \eqref{eq_p}. 
We divide both sides of \eqref{steady-state} by $\sum_{i \in I_j} \lambda_i y_{ij}$ and substract the infinitesimal positive constant $\epsilon$ from the right side.
% to obtain its linear equivalent \eqref{steady-state-2}.
The constraints \eqref{steady-state}, \eqref{open_drone}, and \eqref{eq_p} can then be replaced by \eqref{steady-state-2}, \eqref{open_drone-2}, and \eqref{eq_p-2} in the constraint set of $\mathbf{R-BFP}$. 
%Constraint \eqref{k_j_integer} can also be removed and replaced by \eqref{binary2}. 
%each constraint. Since $K_j$ does not appear anymore in the objective function and in the constraint set, we can drop \eqref{}
This gives the mixed-integer feasible set given in the statement of Proposition \ref{T1}.

\noindent
Part (ii): We now remove the general integer variables $K_j$ from the objective function \eqref{D1_obj}. 
In each %ratio 
term %($m=1,\ldots,M$)
%\begin{equation}
%\label{INTERM0}
%\frac{\sum_{l \in I_j} \lambda_l y_{lj} \mathbb{E}[S_{lj}^2] (\sum_{l \in I_j} \lambda_l y_{lj} \mathbb{E}[S_{lj}])^{K_j-1}}{2(K_j-1)! (K_j-\sum_{l \in I_j} \lambda_l y_{lj} \mathbb{E}[S_{lj}])^2 \big[ \sum_{n = 0}^{K_j-1} \frac{(\sum_{l \in I_j} \lambda_l y_{lj} \mathbb{E}[S_{lj}])^n}{n!} + \frac{(\sum_{l \in I_j} \lambda_l y_{lj} \mathbb{E}[S_{lj}])^{K_j}}{(K_j - 1)!(K_j - \sum_{l \in I_j} \lambda_l y_{lj} \mathbb{E}[S_{lj}]}\big]}
%\end{equation}
of the objective, we first replace the general integer variable $K_j$ by the index parameter $m$, which~gives:

\begin{equation}
\label{INTERM1}
%\sum_{i \in I} \sum_{j \in J} \sum_{m=1}^{M}
\frac{\sum_{l \in I_j} \lambda_l y_{lj} \mathbb{E}[S_{lj}^2] (\sum_{l \in I_j} \lambda_l y_{lj} \mathbb{E}[S_{lj}])^{m-1}}{2(m-1)! (m-\sum_{l \in I_j} \lambda_l y_{lj} \mathbb{E}[S_{lj}])^2 \big[ \sum_{n = 0}^{m-1} \frac{(\sum_{l \in I_j} \lambda_l y_{lj} \mathbb{E}[S_{lj}])^n}{n!} + \frac{(\sum_{l \in I_j} \lambda_l y_{lj} \mathbb{E}[S_{lj}])^{m}}{(m - 1)!(m - \sum_{l \in I_j} \lambda_l y_{lj} \mathbb{E}[S_{lj}]}\big]}
\vspace{-0.05in}
\end{equation}
\begin{comment}
\begin{equation}
\label{INTERM10}
\frac{\gamma_j^{\WMR{(m)}} \ \sum_{l \in I_j} \lambda_l y_{lj} \mathbb{E}[S_{lj}^2] (\sum_{l \in I_j} \lambda_l y_{lj} \mathbb{E}[S_{lj}])^{m-1}}{2(m-1)! (m-\sum_{l \in I_j} \lambda_l y_{lj} \mathbb{E}[S_{lj}])^2 \big[ \sum_{n = 0}^{m-1} \frac{(\sum_{l \in I_j} \lambda_l y_{lj} \mathbb{E}[S_{lj}])^n}{n!} + \frac{(\sum_{l \in I_j} \lambda_l y_{lj} \mathbb{E}[S_{lj}])^{m}}{(m - 1)!(m - \sum_{l \in I_j} \lambda_l y_{lj} \mathbb{E}[S_{lj}]}\big]}
\end{equation}
\end{comment}
%and sum the resulting ratio terms \eqref{INTERM10} 
% giving
\noindent
Next, we multiply each expression \eqref{INTERM1} by the corresponding binary variable $\gamma_j^{\WMR{(m)}}$ and we sum  the $m$ resulting ratio terms (i.e., $\gamma_j^{\WMR{(m)}} \times \eqref{INTERM1}$), which gives
\begin{equation}
\label{INTERM2}
\hspace{-0.2cm}
\sum_{m=1}^{M}
\frac{\gamma_j^{\WMR{(m)}} \sum_{l \in I_j} \lambda_l y_{lj} \mathbb{E}[S_{lj}^2] (\sum_{l \in I_j} \lambda_l y_{lj} \mathbb{E}[S_{lj}])^{m-1}}{2(m-1)! (m-\sum_{l \in I_j} \lambda_l y_{lj} \mathbb{E}[S_{lj}])^2 \big[ \sum_{n = 0}^{m-1} \frac{(\sum_{l \in I_j} \lambda_l y_{lj} \mathbb{E}[S_{lj}])^n}{n!} + \frac{(\sum_{l \in I_j} \lambda_l y_{lj} \mathbb{E}[S_{lj}])^{m}}{(m-1)!(m-\sum_{l \in I_j}\lambda_l y_{lj} \mathbb{E}[S_{lj}]}\big]}\ .
\vspace{-0.025in}
\end{equation}
with $m$ ($m=1,\ldots,M$) denoting the possible number of drones positioned at any open DB. 

\noindent
Due to \eqref{NEW1}, at most one binary variable $\gamma_j^{\WMR{(m)}}$ for each $j\in J$ takes a non-zero value (i.e., 1) and, therefore, at most one term in the summation is \eqref{INTERM2} is non-zero.
This, combined with the fact that the feasible set of $\mathbf{R-BFP}$  ensures that \eqref{SUB1} holds true for each $j$ (see Part i)), implies that \eqref{INTERM2} is equivalent to the expression in the second line of the objective function \eqref{D1_obj} of {\bf B-IFP}, as shown next. 

\noindent
Two cases are now to be considered.
First, if for any arbitrary $j \in J$, we have
$K_j=0$, each $\gamma_j^{\WMR{(m)}}, m=1,\ldots,M$ is equal to 0 since the constraints in $\mathbf{R-BFP}$ imply \eqref{SUB1}, and 
%the expression in 
\eqref{INTERM2} is equal to 0.
Second, if $K_j\neq 0$ with $K_j \leq M$, then exactly one of the $\gamma_j^{\WMR{(m)}}, m=1,\ldots,M$ is equal to 1. More precisely, due to \eqref{SUB1}, we have 
$\gamma_j^{\WMR{(K_j)}}=1$ and 
$\gamma_j^{\WMR{(m)}}=0, m=1,\ldots,M, m \neq K_j$.
This means that the only term in \eqref{INTERM2} taking a non-zero value is the one for which $m=K_j$, which in turn implies that each term $m$ in \eqref{INTERM2} is equal to the expression in the second line of \eqref{D1_obj}, which provides the result that we set out to prove.
\hfill$\Box$
\end{proof}
%%%%%%%%%%%%%%%%%%%%%%%%%%%%%%%%%%
%%%%%%%%%%%%%%%%%%%%%%%%%%%%%%%%%%
%%%%%%%%%%%%%%%%%%%%%%%%%%%%%%%%%%
%%%%%%%%%%%%%%%%%%%%%%%%%%%%%%%%%%
\subsection{Proof of Theorem \ref{T2}} \label{ATH41}
\vspace{-0.1in}
{\it \underline{\bf{Theorem \ref{T2}}}:
Define the index sets 
$D_j = \{(l,t): l,t \in I_j, l <t \}, j\in J$. 
%$D_j = \{(l,t):l=1,\ldots,|I_j|-1, t=l+1,\ldots,|I_j|\}, j \in J$.
Let 
$z_{jlt} \in [0,1]$,
$\mu^{(m)}_{ij} \in [0,\bar{U}_j^{(m)}]$, 
$\tau_{jlt} \in [0,\bar{U}^{(2)}_j]$, and
$\omega^{(m)}_{ij} \in [0,\bar{\mu}_{ij}^{(m)}]$
%and $\omega^r_{ij} \in [0,\bar{U}_j^m]$,
$i,l,t \in I_j, (l,t) \in D_j, j\in J, m=1,\ldots,M$
be continuous auxiliary variables used for the linearizization of bilinear terms.
The MILP problem $\mathbf{R-MILP}$ 
\begin{subequations}
%\label{RDN3}
\begin{align}
\min & \ \sum_{i \in I} \sum_{j \in J_i} \frac{y_{ij}d_{ij} \lambda_i}{v \sum_{l \in I} \lambda_l} 
+  \sum_{i \in I} \sum_{j \in J_i} 
\Bigg[ \frac{\omega_{ij}^{(1)}}{2} + \frac{\omega_{ij}^{\WMR{(2)}}}{2} \Bigg] 
\frac{\lambda_i}{\sum_{l \in I} \lambda_l} & \notag \\ 
\text{s.to} \ & (x,y,\gamma) \in \mathcal{B} & \notag \\ 
& U_j^{(1)} = \sum_{l \in I_j}\lambda_l \mu^{(1)}_{lj} \tilde{S}_{lj} + \sum_{l \in I_j}\lambda_l y_{lj} \tilde{S}_{lj}^{2} & j \in J  \notag \\ 
&4U^{(2)}_{j} = 
\sum_{l \in I_j} \lambda_l^2 \tilde{S}_{lj} (\mu^{(2)}_{lj}\tilde{S}_{lj} + y_{lj}  \tilde{S}_{lj}^2) 
+ 2 \sum_{(l,
t) \in D_j} \lambda_l \lambda_t \tilde{S}_{tj} 
(z_{jlt} \tilde{S}_{lj}^{2} +  \tau_{jlt}^{}\tilde{S}_{lj} )
& j \in J \label{V} \notag \\ 
&(y_{lj},y_{tj},z_{jlt}) \in \mathcal{F}_y^z & (l,t) \in D_j, j \in J   \notag \\  
&(y_{lj},U^{(m)}_j,\mu^{(m)}_{lj}) \in  \mathcal{M}_{yU}^{\mu} & \hspace{-2cm} l \in I_j, j \in J, {m}=1,\ldots,M \notag \\ 
&(z_{jlt},U^{(2)}_\ML{j},\tau_{jlt}) \in \mathcal{M}_{zU}^{\tau}& (l,t) \in D_j, j \in J   \notag \\ 
&\MLB{(\gamma^{(1)}_j,\mu^{(1)}_{ij}, \omega^{(1)}_{ij}) \in  \mathcal{M}_{\ \gamma\mu}^{'\omega}}& \MLB{i \in I, j \in J_i} \notag 
\end{align}
\end{subequations}
is equivalent to the nonlinear integer problems $\mathbf{B-IFP}$ and 
$\mathbf{R-BFP}$ for $M$=2.
}

\noindent
\begin{proof}
\MLB{{\bf (i) Linearization}:} 
For $M = 2$, the fractional terms 
\begin{equation}
\label{FT1}
\sum_{m=1}^{M}
\left[ \frac{\gamma_j^{\WMR{(m)}} \sum_{l \in I_j} \lambda_l y_{lj} \tilde{S}_{lj}^2 (\sum_{l \in I_j} \lambda_l y_{lj} \tilde{S}_{lj})^{m-1}}{2(m-1)! (m-\sum_{l \in I_j} \lambda_l y_{lj} \tilde{S}_{lj})^2 \big[ \sum_{n = 0}^{m-1} \frac{(\sum_{l \in I_j} \lambda_l y_{lj} \tilde{S}_{lj})^n}{n!} + \frac{(\sum_{l \in I_j} \lambda_l y_{lj} \tilde{S}_{lj})^{m}}{(m - 1)!(m - \sum_{l \in I_j} \lambda_l y_{lj} \tilde{S}_{lj}}\big]} \right]
\end{equation}
in \eqref{OBJ2} can be rewritten as 
\begin{align*}
&\frac{1}{2}
\Bigg[
\frac{\gamma_j^{\WMR{(1)}}\sum_{l \in I_j}\lambda_l y_{lj}\tilde{S}_{lj}^{2}}{(1 - \sum_{l \in I_j}\lambda_l y_{lj}\tilde{S}_{lj})} +
\\
& \frac{ \gamma_j^{\WMR{(2)}} \sum_{l \in I_j} \lambda_{l}y_{lj}\tilde{S}_{lj}^2   \sum_{l \in I_j} \lambda_{l}y_{lj}\tilde{S}_{lj}}
{(2- \sum\limits_{l \in I_j} \lambda_{l}y_{lj}\tilde{S}_{lj})^2 + (2- \sum\limits_{l \in I_j} \lambda_{l}y_{lj}\tilde{S}_{lj})^2 
\sum\limits_{l \in I_j} \lambda_{l}y_{lj}\tilde{S}_{lj}
 + (2- \sum\limits_{l \in I_j} \lambda_{l}y_{lj}\tilde{S}_{lj})( \sum\limits_{l \in I_j} \lambda_{l}y_{lj}\tilde{S}_{lj})^{2}} 
 \Bigg]
\end{align*}
In order to remove the fractional terms from the objective function, we introduce the nonnegative auxiliary variables $U_j^{(m)}, m=1,2$ defined as:

{\small
\begin{align}
U_j^{(1)}  &= \frac{\sum_{l \in I_j}\lambda_l y_{lj}\tilde{S}_{lj}^{2}}{1 - \sum_{l \in I_j}\lambda_l y_{lj}\tilde{S}_{lj}} \label{1drone_cons} \\
U_j^{(2)} &= \frac{ \sum_{l \in I_j} \lambda_{l}y_{lj}\tilde{S}_{lj}^2   \sum\limits_{l \in I_j} \lambda_{l}y_{lj}\tilde{S}_{lj}}{(2- \sum\limits_{l \in I_j} \lambda_{l}y_{lj}\tilde{S}_{lj})^2  + (2- \sum\limits_{l \in I_j} \lambda_{l}y_{lj}\tilde{S}_{lj})^2 
\sum\limits_{l \in I_j} \lambda_{l}y_{lj}\tilde{S}_{lj}
 + (2- \sum\limits_{l \in I_j} \lambda_{l}y_{lj}\tilde{S}_{lj}) (\sum\limits_{l\in I_j} \lambda_{l}y_{lj}\tilde{S}_{lj}
)^{2}} \label{2drone_cons}
\end{align} 
}
Problem {\bf R-BFP} can now be equivalently rewritten as: 
\begin{align}
%\mathbf{R1}: \; 
\min & \sum_{i \in I} \sum_{j \in J_i} \frac{y_{ij}d_{ij} \lambda_i}{v \sum_{l \in I} \lambda_l} 
+  \sum_{i \in I} \sum_{j \in J_i} \Bigg[ \frac{\gamma_j^{(1)} y_{ij}U_j^{(1)}}{2} 
+ \frac{\gamma_j^{(2)} y_{ij}U_j^{(2)}}{2} \Bigg] 
\frac{\lambda_i}{\sum_{l \in I} \lambda_l} \label{OBJ1}\\ 
\; \text{s.to} \ &  % \eqref{steady-state-2} ; 
\eqref{1drone_cons} - \eqref{2drone_cons} \notag \\ 
& \ (x,y,\gamma) \in \mathcal{B} \notag
\end{align}
The second step is the linearization of the nonconvex equality constraints \eqref{1drone_cons} and \eqref{2drone_cons} and the~polynomial terms of the objective function \eqref{OBJ1}. 
Multiplying both sides of \eqref{1drone_cons} by 
$(1 - \sum_{l \in I_j}\lambda_l y_{lj}\tilde{S}_{lj})$ gives  
\begin{equation}
U_j^{(1)}  = \sum_{l \in I_j}\lambda_l y_{lj} U^{(1)}_{j} \tilde{S}_{lj} + \sum_{l \in I_j} \lambda_l y_{lj} \tilde{S}_{lj}^{2} \; , \; j \in J \label{multi1}
\vspace{-0.12in}
\end{equation} 
which can be linearized as \eqref{U} by introducing the linearization auxiliary variables $\mu^{(1)}_{lj}$ and the McCormick inequalities 
%\eqref{MAC1}-\eqref{MAC4} 
in the set $\mathcal{M}_{yU}^{\mu}$ to ensure that: 
$\mu^{\WMR{(1)}}_{lj} = y_{lj} U^{\WMR{(1)}}_{j}, l\in I_j, j\in J$.

\noindent
Similarly, multiplying both sides of \eqref{2drone_cons} by 
$$(2- \sum_{l \in I_j} \lambda_{l}y_{lj}\tilde{S}_{lj})^2  + (2- \sum_{l \in I_j} \lambda_{l}y_{lj}\tilde{S}_{lj})^2 
\sum_{l \in I_j} \lambda_{l}y_{lj}\tilde{S}_{lj}
 + (2- \sum_{l \in I_j} \lambda_{l}y_{lj}\tilde{S}_{lj})( \sum_{l \in I_j} \lambda_{l}y_{lj}\tilde{S}_{lj}
)^{2}$$ 
\vspace{-0.455in}
gives
\begin{align} 
\vspace{-0.1in}
& \; U_j^{\WMR{(2)}}   \left((2-\sum_{l \in I_j} \lambda_{l}y_{lj}\tilde{S}_{lj})^2+(2-\sum_{l \in I_j} \lambda_{l}y_{lj}\tilde{S}_{lj})^2 \sum_{l \in I_j} \lambda_{l}y_{lj}\tilde{S}_{lj} + (2-\sum_{l \in I_j} \lambda_{l}y_{lj}\tilde{S}_{lj})( \sum_{l \in I_j} \lambda_{l}y_{lj}\tilde{S}_{lj} )^{2} \right) \notag \\
 = & \; \sum_{l \in I_j} \lambda_{l}y_{lj}\tilde{S}_{lj}^2   \sum_{l \in I_j} \lambda_{l}y_{lj}\tilde{S}_{lj} \ .
 \label{V_mul_denomi}
\end{align}
The left-hand side of \eqref{V_mul_denomi} can be rewritten as 
\begin{subequations}
\label{SIMPL}
\begin{align} & \; 
U_j^{\WMR{(2)}} \Bigg((2- \sum_{l \in I_j} \lambda_{l}y_{lj}\tilde{S}_{lj})
\times
\left[
2- \sum_{l \in I_j} \lambda_{l}y_{lj}\tilde{S}_{lj}
+ (2- \sum_{l \in I_j} \lambda_{l}y_{lj}\tilde{S}_{lj}) 
\sum_{l \in I_j} \lambda_{l}y_{lj}\tilde{S}_{lj}
 + ( \sum_{l \in I_j} \lambda_{l}y_{lj}\tilde{S}_{lj})^{2}
\right] \Bigg)  \\
  = & \; U_j^{(2)} \Bigg(
(2- \sum_{l \in I_j} \lambda_{l}y_{lj}\tilde{S}_{lj})
(2 + \sum_{l \in I_j} \lambda_{l}y_{lj}\tilde{S}_{lj}) \Bigg)   
%\\ = & \; 
\; = \; 
U_j^{(2)} \Bigg(
4- \sum_{l \in I_j}  \sum_{t \in I_j} \lambda_{l} \lambda_{t} y_{lj} y_{tj} \tilde{S}_{lj} \tilde{S}_{tj} \Bigg) \ .
%\\
%= & \; U_j^2 \Bigg(
%4- \sum_{(l,t) \in D_j}  2 \lambda_{l} \lambda_{t} y_{lj} y_{tj} \tilde{S}_{lj} \tilde{S}_{tj} - \sum_{l \in I_j}  (\lambda_{l})^2 y_{lj} (\tilde{S}_{lj})^2 \Bigg)   \ .
\end{align} 
\label{V_LHS} 
\end{subequations}
The double summation expression in \eqref{V_LHS} can be simplified. 
First, using the idempotent identity of binary variables, we have:  
$(y_{lj})^2 = y_{lj}$.
Second, we split the summands in \eqref{V_LHS} into two parts including respectively the separable bilinear terms $(y_{lj})^2$ and the nonseparable ones $y_{lj}y_{tj}, l \neq t$: 
\vspace{-0.05in}
\begin{equation}
\sum_{l \in I_j} \sum_{t \in I_j} \lambda_{l} \lambda_{t} y_{lj} y_{tj} \tilde{S}_{lj} \tilde{S}_{tj} = 
\sum_{l \in I_j} \lambda_l^2 y_{lj} (\tilde{S}_{lj})^{2}  + 2\sum_{(l,t) \in D_j} \lambda_l \lambda_t y_{lj} y_{tj}\tilde{S}_{lj} \tilde{S}_{tj} \ .
\label{SIMPLI}
\end{equation}
Combining \eqref{SIMPL} and \eqref{SIMPLI}, we reformulate the left side of \eqref{V_mul_denomi} as:
\begin{equation}
U_j^{\WMR{(2)}} \Big(
4- \sum_{l \in I_j} \lambda_l^2 y_{lj} (\tilde{S}_{lj})^{2}  - 2\sum_{(l,
t) \in D_j} \lambda_l \lambda_t y_{lj} y_{tj}\tilde{S}_{lj} \tilde{S}_{tj}
\Big)
\label{SIMPLI2}
\end{equation}
% By expanding the binomial terms and cancelling out common terms, the left-hand side of \eqref{V_mul_denomi} can be simplified as
% \begin{align*}
% &V_{j}\Big(4 - 4\sum_{l \in I_j} \lambda_{l}y_{lj}\tilde{S}_{lj} + \sum_{l \in I}\lambda_{l}^2 y_{lj} (\tilde{S}_{lj})^{2} + \sum_{l,t \in I, l \neq t}\lambda_{l} \lambda_{t} y_{lj}y_{tj} \tilde{S}_{lj}\tilde{S}_{tj} \\
% &+ 4\sum_{l \in I}\lambda_l y_{lj} \tilde{S}_{ij} 
% - 4\sum_{l \in J} \lambda_{l} y_{lj} \tilde{S}_{lj} \sum_{l \in J}\lambda_{l} y_{lj}
% \tilde{S}_{lj}
% +\sum_{l \in I} \lambda_{l}^{2} y_{lj} (\tilde{S}_{lj})^2 \sum_{l \in I}\lambda_l y_{lj}\tilde{S}_{lj}
% +\sum_{l,t \in I, l \neq t}\lambda_{l} \lambda_{t} y_{lj}y_{tj} \tilde{S}_{lj}\tilde{S}_{tj}\sum_{l \in I}\lambda_l y_{lj}\tilde{S}_{lj} \\
% &+(2-\sum_{l \in I} \lambda_{l}y_{lj}\mathbb{E}[S_{lj}])( \sum_{l \in I} \lambda_{l}^{2}y_{lj}(\bar{S}_{lj})^2 + \sum_{l,t \in I, l \neq t} \lambda_{l}\lambda_{t}y_{tj}y_{lj}\bar{S}_{lj}\bar{S}_{tj}) \\
% &+2 \sum_{l \in I} \lambda_{l}^{2}y_{lj}(\tilde{S}_{lj})^2 
% + 2\sum_{l,t \in I, l \neq t} \lambda_{l}\lambda_{t}y_{tj}y_{lj}\tilde{S}_{lj}\tilde{S}_{tj}
% -\sum_{l \in I} \lambda_{l}y_{lj}\tilde{S}_{lj}
% \sum_{l \in I} \lambda_{l}^{2}y_{lj}(\tilde{S}_{lj})^2
% -\sum_{l \in I} \lambda_{l}y_{lj}\tilde{S}_{lj}
% \sum_{l,t \in I, l \neq t} \lambda_{l}\lambda_{t}y_{tj}y_{lj}\tilde{S}_{lj}\tilde{S}_{tj}\Big) \\
% &=V_{j}\Big(4 - \sum_{l} \lambda_l^2 y_{lj} {\tilde{S}}_{lj}^{2} - \sum_{l,t \in I, l \neq t}\lambda_{l} \lambda_{t} y_{lj} y_{tj}\tilde{S}_{lj} \tilde{S}_{tj}\Big). 
% \end{align*}
Using the same reasoning, the right-hand side of \eqref{V_mul_denomi} is equal to
$$
\sum_{l \in I_j} \lambda_l^2 y_{lj} \tilde{S}_{lj}^{2} \tilde{S}_{lj} +2\sum_{(l,
t) \in D_j} \lambda_l \lambda_t y_{lj} y_{tj}\tilde{S}_{lj}^{2} \tilde{S}_{tj} \ .
$$
Thus, \eqref{2drone_cons} and \eqref{V_mul_denomi} are equivalent to
%\begin{subequations} 
{\small
\begin{align}
U_j^{\WMR{(2)}}
\Big(4 - \sum_{l \in I_j} \lambda_l^2 y_{lj} (\tilde{S}_{lj})^{2} - 2\sum_{(l,t) \in D_j} \lambda_{l} \lambda_{t} y_{lj} y_{tj}\tilde{S}_{lj} \tilde{S}_{tj}\Big) % \notag \\
= \sum_{l \in I_j} \lambda_l^2 y_{lj} \tilde{S}_{lj}^{2} \tilde{S}_{lj} + 2\sum_{(l,
t) \in D_j} \lambda_l \lambda_t y_{lj} y_{tj}\tilde{S}_{lj}^{2} \tilde{S}_{tj} %\hspace{-0.21in}
 \label{2cons_r1}
% \sum_{l} \lambda_{l}^{2}y_{lj}\tilde{S}_{lj} + \sum_{l, t \in I, l \neq t} \lambda_l \lambda_t y_{lj} y_{tj} S_{tj} &\leq \sum_{m=1}^{M} m\gamma_j^m\sum_l \lambda_l  y_{lj}, \ j \in J
\vspace{-0.31in}
\end{align}
}
% \end{subequations}
which defines a nonlinear equality constraint including bilinear terms $U_j^{\WMR{(2)}} y_{lj}$, $y_{lj} y_{tj}$ and trilinear terms $U_j^{\WMR{(2)}} y_{lj} y_{tj}$ and has therefore a nonconvex feasible area. 

Next, we linearize the binary bilinear term $y_{lj}y_{tj}$ by introducing the variables  $z_{jlt}$ and the linear inequalities %\eqref{MAC_z1}-\eqref{MAC_z4} 
in the set $\mathcal{F}_{y}^z$ which ensures that $z_{jlt}:= y_{lj}y_{tj}$ and gives the equality:

{\small 
 \begin{align} \label{2cons_R2}
    4U_j^{\WMR{(2)}} - \sum_{l \in I_j} \lambda_l^2 U_j^{\WMR{(2)}} y_{lj} (\tilde{S}_{lj})^{2} - 2\sum_{(l,t) \in D_j} \lambda_{l} \lambda_{t} U^{\WMR{(2)}}_{j}z_{lt}^{j}\tilde{S}_{lj} \tilde{S}_{tj} &= \sum_{l \in I_j} \lambda_l^2 y_{lj} \tilde{S}_{lj}^{2} \tilde{S}_{lj} + 2\sum_{(l,t) \in D_j} \lambda_l \lambda_t z_{jlt} \tilde{S}_{lj}^{2} \tilde{S}_{tj}\ ,\ j\in J 
\end{align}
}
To linearize the remaining bilinear terms $U_j^{\WMR{(2)}} y_{lj}$ and $U_j^{\WMR{(2)}} z_{jlt}$ in \eqref{2cons_R2}, we respectively use the inequalities 
%\eqref{MAC1}-\eqref{MAC4} and \eqref{MAC_psi1}-\eqref{MAC_psi4} 
in $\mathcal{M}^{\mu}_{yU}$ and $\mathcal{M}^{\tau}_{Uz}$ 
to ensure $\mu^{\WMR{(2)}}_{lj}:= U_j^{\WMR{(2)}} y_{lj}$ and $\WMR{\tau_{jlt}}:= U_j^{\WMR{(2)}} z_{jlt}$, 
which gives us, in fine, a mixed-integer linear feasible set equivalent to \eqref{2drone_cons}. 

Having linearized the bilinear terms in \eqref{1drone_cons} and \eqref{2drone_cons}, we do the same for the trilinear terms $\gamma_j^{(1)} y_{ij}U_j^{(1)}$ and  
$\gamma_j^{(2)} y_{ij}U_j^{(2)}$ in the objective function \eqref{OBJ1}. 
First, since $\mu^{(m)}_{ij} = U_{j}^{(m)} y_{ij}$, the trilinear terms can be reduced to the bilinear terms $\gamma_j^{(1)} \mu^{(1)}_{ij}$ and $\gamma_j^{(2)} \mu^{(2)}_{ij}$. 
\MLB{Introducing the variables $\omega_{ij}^{(m)}$ and the set of inequalities 
\eqref{OLD-R-MILP}
$(\gamma^{(m)}_j,\mu^{(m)}_{ij},\omega^{(m)}_{ij}) \in  \mathcal{M}_{\gamma\mu}^{\omega}, i \in I, j \in J_i, m=1,\ldots,M$ \eqref{OLD-R-MILP} 
which enforce  $\omega_{ij}^{(m)}:= \gamma_{j}^{(m)}\mu^{(m)}_{ij}$, we have  $\omega_{ij}^{(m)} = \gamma_{j}^{(m)}\mu^{(m)}_{ij} =
 \gamma_{j}^{(m)} y_{ij} U_{j}^{(m)}$ 
 and substituting $\omega_{ij}^{(m)}$ for $\gamma_{j}^{(m)} y_{ij} U^{(m)}_j$ gives a linear objective function. 

\noindent
\MLB{{\bf (ii) Compaction}:}
The last part of the proof shows that we can replace the sets of inequalities $\mathcal{M}_{\gamma\mu}^{\omega}, i \in I, j \in J_i, m=1,\ldots,M$ in \eqref{OLD-R-MILP} by the much more \MLB{compact sets
$\mathcal{M}_{\ \gamma\mu}^{' \omega}, i \in I, j \in J_i$ (with $m=1$; see last equation in {\bf R-MILP})}. 
 To do that, we first show that
\begin{equation}
\label{PR1}
\omega^{(2)}_{ij} = \mu^{(2)}_{ij}, i \in I, j \in J_i \ .
\end{equation}	
Consider any arbitrary pair of terms $\mu^{(2)}_{ij}$  and $\omega^{(2)}_j$.
The set of linearization constraints 
$(\gamma^{(2)}_j,\mu^{(2)}_{ij},\omega^{(2)}_{ij}) \in  \mathcal{M}_{\gamma\mu}^{\omega}$ %in {\bf R-MILP}
implies that $\omega^{(2)}_j$ can only take values 0 and $\mu^{(2)}_{ij}$ while the linearization constraints 
$(y_{ij},U_j^{(2)},\mu_{ij}^{(2)}) \in  \mathcal{M}_{yU}^{\mu}$ 
implies that $\mu^{(2)}_{ij}$  can only take values 0 and $U^{(2)}_{j}$.
Consider all possible values for $\mu^{(2)}_{ij}$  and $\omega^{(2)}_j$.
\newline
The case $\omega^{(2)}_{ij}=\mu^{(2)}_{ij}$ is obvious.
If  $\omega^{(2)}_{ij}=0$, then:
either  $\mu^{(2)}_{ij}=0$ and \eqref{PR1} holds, 
or $\gamma^{(2)}_{j}=0$, which in turn implies that $U^{(2)}_j=0$. Therefore, $\mu^{(2)}_{ij}=0$ and \eqref{PR1} holds. 
\newline
If $\mu^{(2)}_{ij}=0$, then $\omega^{(2)}_{ij} =0$ due to 
$(\gamma^{(2)}_j,\mu^{(2)}_{ij},\omega^{(2)}_{ij}) \in  \mathcal{M}_{\gamma\mu}^{\omega}$ 
%in {\bf R-MILP}  
and \eqref{PR1} holds.
If $\mu^{(2)}_{ij}=U^{(2)}_j >0$, then $\gamma^{(2)}_{j}=1$ and  $\omega_{ij}^{(2)} = \mu^{(2)}_{ij}=U^{(2)}_j$. \\
This shows that \eqref{PR1} always holds true, for any value that  $\omega^{(2)}_{ij}$ and $\mu^{(2)}_{ij}$ can possibly take. Hence, the variables $\omega^{(2)}_{ij}, i\in I, j \in J_i$ are superfluous, and each variable $\omega^{(2)}_{ij}$ can be replaced by its counterpart $\mu^{(2)}_{ij}$. 
Additionally, since, for each $i \in I, j \in J_i$, we have $\omega^{(2)}_{ij}=\mu^{(2)}_{ij}$ and
$\omega^{(2)}_{ij} = \mu^{(2)}_{ij} \cdot \gamma^{(2)}_j$ due to $(\gamma^{(2)}_j,\mu^{(2)}_{ij},\omega^{(2)}_{ij}) \in  \mathcal{M}_{\gamma\mu}^{\omega}$, it follows that: 
\begin{equation}
\label{PR2}
\omega^{(2)}_{ij} = 
\mu^{(2)}_{ij} =  \mu^{(2)}_{ij} \cdot \gamma^{(2)}_j , i \in I, j \in J_i \ .
\end{equation}	
Therefore, each bilinear term $\mu^{(2)}_{ij} \cdot \gamma^{(2)}_j$ can be replaced by $\mu^{(2)}_{ij}$ and the set of linearization constraints $(\gamma^m_j,\mu^m_{ij},\omega^m_{ij}) \in  \mathcal{M}_{\gamma\mu}^{\omega}, i \in I, j \in J_i,m=1,\ldots,M$ \eqref{OLD-R-MILP}
%in {\bf R-MILP} 
can be replaced by its proper subset $(\gamma^{(1)}_j,\mu^{(1)}_{ij},\omega^{(1)}_{ij}) \in  
\mathcal{M}_{\ \gamma\mu}^{' \omega}, i \in I, j \in J_i \ (m=1)$. This gives the result that we set out to prove. 
}
\hfill$\Box$
\end{proof}

%%%%%%%%%%%%%%%%%%%%%%%%%%%%%%%%%%%%%%%%%%%%
%%%%%%%%%%%%%%%%%%%%%%%%%%%%%%%%%%%%%%%%%%%%%
\subsection{Proof of Theorem \ref{TH-GEN}} \label{ATH4}

{\it \underline{\bf{Theorem \ref{TH-GEN}}}:
A stochastic network design model of form $\mathbf{R-BFP}$ that minimizes the average response time of a network of $M/G/K_j$ queueing systems with variable and finitely bounded number of servers $K_j$ is always MILP-representable.
%\ML{as long as the capacity is finitely bounded.}
}

Before presenting the proof of Theorem 
\ref{TH-GEN}, we recall two linearization methods  \cite{adams2007} used in the proof:

%\begin{prop}
%[Linear representation of polynomial terms with products of $n$ binary variables] 
%\label{P1}
\begin{itemize}
	\item Let $y \in \{0,1\}^n$.
	The polynomial set 
	$\{(y,z) \in \{0,1\}^n \times [0,1]: z = \prod_{i=1}^n y_i \}$ 
	can be equivalently represented with the MILP set defined by:
	\begin{equation}
	\label{P1}
	\Big\{ (y,z): 
	z \geq 0, \;
	z \geq \sum_{i=1}^n y_i - n + 1, \; 
	z \leq  y_i,\; i=1,\ldots,n \Big\}\ .
	\end{equation}
	%\end{prop}
\item Let $x \in [0,\bar{x}]$, $y \in \{0,1\}^n$.
The polynomial set %of degree $n+1$
$\{(x,y,z) \in [0,\bar{x}] \times \{0,1\}^n \times [0,\bar{x}]: z = x \prod_{i=1}^n y_i  \}$ 
can be equivalently represented with the MILP set defined by:
\begin{equation}
\label{P2}    
\Big\{ (x,y,z): z \geq 0, \;
z \geq x - \bar{x} \Big(n-\sum_{i=1}^n y_i\Big), \; 
z \leq \bar{x} y_i,  \; 
z \leq x,  \; i=1,\ldots,n \Big\} \ .
\end{equation}
%\end{prop}
\end{itemize}

\noindent
\begin{proof} 
%{Proof of Theorem \ref{thm_avg_resp}}
	% \proof%{Proof of Theorem \ref{TH-GEN}}
	Each fractional term \eqref{FT1} ($m \in \{1,M\}$) in the objective function 
 of
 $\mathbf{R-BFP}$ can be equivalently rewritten~as:
	\begin{equation}
		\label{FT2}
		\frac{\gamma_j^{\WMR{(m)}} \sum_{l \in I_j} \lambda_l y_{lj} \tilde{S}_{lj}^2 (\sum_{l \in I_j} \lambda_l y_{lj} \tilde{S}_{lj} )^{m-1}}
		{2\big(
			(m-1)! (m-\sum_{l \in I_j} \lambda_l y_{lj} \tilde{S}_{lj})^2
			\sum_{n = 0}^{m-1} \frac{(\sum_{l \in I_j} \lambda_l y_{lj} \tilde{S}_{lj})^n}{n!} 
			+ (m - \sum_{l \in I_j} \lambda_l y_{lj} \tilde{S}_{lj}) 
			(\sum_{l \in I_j} \lambda_l y_{lj} \tilde{S}_{lj})^m
			\big)} .
		%\ , \ 
	\end{equation}
	%As in Theorem \ref{T2}, 
	Introducing an auxiliary continuous variable $V_j^{\WMR{(m)}} \in [0, \bar{V}_{j}^{\WMR{(m)}}]$ for each term ($m$) in \eqref{FT2} gives:
	\begin{equation}
		\label{FT3}
		V_j^{\WMR{(m)}} = \frac{\gamma_j^{\WMR{(m)}} \sum_{l \in I_j} \lambda_l y_{lj} \tilde{S}_{lj}^2 (\sum_{l \in I_j} \lambda_l y_{lj} \tilde{S}_{lj} )^{m-1}}
		{2\big(
			(m-1)! (m-\sum_{l \in I_j} \lambda_l y_{lj} \tilde{S}_{lj})^2
			\sum_{n = 0}^{m-1} \frac{(\sum_{l \in I_j} \lambda_l y_{lj} \tilde{S}_{lj})^n}{n!} 
			+ (m - \sum_{l \in I_j} \lambda_l y_{lj} \tilde{S}_{lj}) 
			(\sum_{l \in I_j} \lambda_l y_{lj} \tilde{S}_{lj})^m
			\big)} 
		%\ , \ 
	\end{equation}
	Substituting $V_j^{\WMR{(m)}}$ for \eqref{FT2} in the objective function \eqref{OBJ2}, problem $\mathbf{R-BFP}$ becomes:
	
%	\begin{subequations}
%		\label{RDN4}
		\begin{align}
			\min & \ \sum_{i \in I} \sum_{j \in J_i} \frac{y_{ij}d_{ij} \lambda_i}{v \sum_{l \in I} \lambda_l} 
			+  \sum_{i \in I} \sum_{j \in J_i} \sum_{m=1}^{M} \ \frac{y_{ij}\lambda_i V_j^{\WMR{(m)}}}{\sum_{l \in I} \lambda_l} \notag \\ 
			\text{s.to} \ 
			& (x,y,\gamma) \in \mathcal{B} & \notag   \\ 
			& \underbrace{V_j^{\WMR{(m)}} 
				(m-1)! (m-\sum_{l \in I_j} \lambda_l y_{lj} \tilde{S}_{lj})^2
				\sum_{n = 0}^{m-1} \frac{(\sum_{l \in I_j} \lambda_l y_{lj} \tilde{S}_{lj})^n}{n!}}_{T1} 
			+ \underbrace{V_j^{\WMR{(m)}}(m - \sum_{l \in I_j} \lambda_l y_{lj} \tilde{S}_{lj}) 
				(\sum_{l \in I_j} \lambda_l y_{lj} \tilde{S}_{lj})^m}_{T2}
			\notag \\
			= \ & \underbrace{1/2 \ \gamma_j^{\WMR{(m)}} \sum_{l \in I_j} \lambda_l y_{lj} \tilde{S}_{lj}^2 (\sum_{l \in I_j} \lambda_l y_{lj} \tilde{S}_{lj} )^{m-1}}_{T3} \; ,  \; j\in J, \ML{m=1,\ldots,M} \label{SUBST1BIS}
		\end{align}
%	\end{subequations}
It can be seen from the above that:
\vspace{-0.06in}
\begin{itemize} 
		\item The objective function includes bilinear terms involving the product of a continuous variable $V_j^{\WMR{(m)}}$ and a binary variable $y_{ij}$.
		These bilinear terms can be linearized using  \eqref{P2}.
		\vspace{-0.06in}
		\item The expressions T1 and T2 in \eqref{SUBST1BIS}  include both polynomial terms of degree $(m+2)$ %various degrees, ranging from 2 to $m+2$}, 
		with monomials involving the product of a continuous variable by up to $(m+1)$ binary variables. 
		These polynomial terms can be linearized using  \eqref{P2}.
		%\item The expression T2 in \eqref{SUBST1}  includes polynomial terms of degree $(m+2)$ with monomials involving products of $m$ or $m+1$ binary variables by one continuous one. These polynomial terms can be linearized with the approach described in Proposition~\ref{P2}.
		\vspace{-0.06in}
		\item The expression T3 in the right-side of  \eqref{SUBST1BIS}  includes polynomial terms of degree $(m+1)$ involving products of $m+1$ binary variables.
		These polynomial terms can be linearized using  \eqref{P1}.
	\end{itemize}
\vspace{-0.06in}
This shows that both the objective function and each constraint \eqref{FT2} can be linearized regardless of the value of $M$, which is the result that we set out to prove.
\hfill$\Box$
\end{proof}
\vspace{-0.04in}
\subsection{Proof of Proposition \ref{VALID1}} \label{A3}
\vspace{-0.03in}
{\it \underline{\bf{Proposition \ref{VALID1}}}:
	Let $U_{j}^{\WMR{(m)}} \in [0,\bar{U}_j^{\WMR{(m)}}], m=1,2$.
	The linear constraints 
\vspace{-0.04in}
	\begin{equation} 
		\notag %\label{VI3}
		\gamma_j^{\WMR{(1)}} + \gamma_j^{\WMR{(2)}} \geq U_j^{\WMR{(m)}} / \bar{U}_{j}^{\WMR{(m)}} , \quad j \in J, m=1,2 
	\end{equation}
	%with $\bar{U}_{jm}$ denoting the $U_j^m$, 
	are valid inequalities for problem \textbf{R-MILP}.
}

\noindent
\begin{proof} %{Proof of Theorem \ref{thm_avg_resp}}
If for any $m \in \{1,2\}$, we have $U_j^{\WMR{(m)}} / \bar{U}_j^{\WMR{(m)}} > 0$, which means $U_j^{\WMR{(m)}} >0$ and 
	$U_j^{\WMR{(m)}} / \bar{U}_j^{\WMR{(m)}} \leq 1$ since $U_j^{\WMR{(m)}} \in [0, \bar{U}_j^{\WMR{(m)}}]$, this implies in turn $\gamma^{\WMR{(1)}}_{j} + \gamma^{\WMR{(2)}}_{j} = 1$. This forces either $\gamma^{\WMR{(1)}}_j$ or $\gamma^{\WMR{(2)}}_j$ to take value 1, which does not cut off any integer solution.
	If there is no DB open at $j$, i.e., $x_j=0 =\gamma_j^{\WMR{(1)}} + \gamma_j^{\WMR{(2)}}$, then $y_{ij}=0, i \in J_i$ due to \eqref{eq_lim} and $U^{\WMR{(1)}}_j=U^{\WMR{(2)}}_j=0$ due to \eqref{1drone_cons} and \eqref{2drone_cons}, which is valid for \eqref{VI3}.
\hfill$\Box$ 
\end{proof}

%%%%%%%%%%%%%%%%%%%%%%%%%%%%%%%%%%
%%%%%%%%%%%%%%%%%%%%%%%%%%%%%%%%%%

\subsection{Proof of Proposition \ref{valid_c_6}} \label{A4}
\vspace{-0.03in}
{\it \underline{\bf{Proposition \ref{valid_c_6}}}:
The linear constraints
\vspace{-0.05in}
\begin{equation}  
\notag %\label{VI4}
U_j^{\WMR{(1)}} \geq U_j^{\WMR{(2)}}  , \quad j \in J
\vspace{-0.04in}
\end{equation}
are valid inequalities for problem \textbf{R-MILP}.
}

\noindent
\begin{proof} %{Proof of Theorem \ref{thm_avg_resp}}
If no DB is set up at location, we have  $U_j^{\WMR{(1)}} = U_j^{\WMR{(2)}} = 0$ due to \eqref{1drone_cons} and \eqref{2drone_cons}, and \eqref{VI4} holds.
\newline
If a DB is set up at location $j$, we have from \eqref{2drone_cons} the first equality below: 
\begin{align}
U_j^{\WMR{(2)}} &= \frac{ \sum_{l \in I_j} \lambda_{l}y_{lj}\tilde{S}_{lj}^2   \sum\limits_{l \in I_j} \lambda_{l}y_{lj}\tilde{S}_{lj}}{(2- \sum\limits_{l \in I_j} \lambda_{l}y_{lj}\tilde{S}_{lj})^2  + (2- \sum\limits_{l \in I_j} \lambda_{l}y_{lj}\tilde{S}_{lj})^2 
\sum\limits_{l \in I_j} \lambda_{l}y_{lj}\tilde{S}_{lj}
+ (2- \sum\limits_{l \in I_j} \lambda_{l}y_{lj}\tilde{S}_{lj}) (\sum\limits_{l\in I_j} \lambda_{l}y_{lj}\tilde{S}_{lj})^{2}} \label{INEQ1} \\
&\leq \frac{ \sum_{l \in I_j} \lambda_{l}y_{lj}\tilde{S}_{lj}^2   \sum\limits_{l \in I_j} \lambda_{l}y_{lj}\tilde{S}_{lj}}{(2- \sum\limits_{l \in I_j} \lambda_{l}y_{lj}\tilde{S}_{lj})^2 
\sum\limits_{l \in I_j} \lambda_{l}y_{lj}\tilde{S}_{lj}} 
= \frac{ \sum_{l \in I_j} \lambda_{l}y_{lj}\tilde{S}_{lj}^2}{(2- \sum\limits_{l \in I_j} \lambda_{l}y_{lj}\tilde{S}_{lj})^2} \label{INEQ2} \\ 
& \leq \frac{ \sum_{l \in I_j} \lambda_{l}y_{lj}\tilde{S}_{lj}^2}{2- \sum\limits_{l \in I_j} \lambda_{l}y_{lj}\tilde{S}_{lj}} \label{INEQ3} \\
&\leq \frac{ \sum_{l \in I_j} \lambda_{l}y_{lj}\tilde{S}_{lj}^2}{1- \sum\limits_{l \in I_j} \lambda_{l}y_{lj}\tilde{S}_{lj}} = U^{\WMR{(1)}}_j. \label{INEQ4}
\end{align}
The validity of the first inequality is implied by the steady-state requirement \eqref{steady-state} according to which we have either $1 - \sum_{l \in I_j} \lambda_{l}y_{lj}\tilde{S}_{lj} \geq 0$ when $\gamma_j^{\WMR{(1)}}=1$ or $2 - \sum_{l \in I_j} \lambda_{l}y_{lj}\tilde{S}_{lj} \geq 0$ when $\gamma_j^{\WMR{(2)}}=1$.
	It follows immediately that the  
	denominator of \eqref{INEQ1} is larger than the one in \eqref{INEQ2}.
	%due to \eqref{steady-state} 
	Since the numerators are the same in \eqref{INEQ1} and \eqref{INEQ2}, we have  \eqref{INEQ2} $\geq$ \eqref{INEQ1}
	Next, observe that \eqref{1drone_cons} implies that we have always $1 \geq \sum_{l \in I_j} \lambda_{l}y_{lj}\tilde{S}_{lj}$ since otherwise the nonnegative auxiliary variable  $U_j^{\WMR{(1)}}$ would be negative.
	%take a strictly negative value. 
	This in turn implies 
	$2- \sum_{l \in I_j} \lambda_{l}y_{lj}\tilde{S}_{lj} \geq 1$ and 
	$(2- \sum_{l \in I_j} \lambda_{l}y_{lj}\tilde{S}_{lj})^2 \geq 
	2- \sum_{l \in I_j} \lambda_{l}y_{lj}\tilde{S}_{lj} \geq 1$. Therefore, the denominator of \eqref{INEQ2} is larger than the one of \eqref{INEQ3} and \eqref{INEQ3} $\geq$ \eqref{INEQ2}
	Similarly, the third inequality is valid since
	$2- \sum_{l \in I_j} \lambda_{l}y_{lj}\tilde{S}_{lj} > 
	1- \sum_{l \in I_j} \lambda_{l}y_{lj}\tilde{S}_{lj}$ which implies \eqref{INEQ4} $>$ \eqref{INEQ3} and allows concluding that $U^{\WMR{(2)}}_j \leq U^{\WMR{(1)}}_j$.
\hfill$\Box$ 
\end{proof}

%%%%%%%%%%%%%%%%%%%%%%%%%%%%%%%%%%
%%%%%%%%%%%%%%%%%%%%%%%%%%%%%%%%%%
%%%%%%%%%%%%%%%%%%%%%%%%%%%%%%%%%%
%%%%%%%%%%%%%%%%%%%%%%%%%%%%%%%%%%
\subsection{Proof of Proposition \ref{valid_c_5}} \label{A5}

{\it \underline{\bf{Proposition \ref{valid_c_5}}}:
The linear constraints
\vspace{-0.05in}
\begin{equation} 
\notag %	\label{VI5}
\gamma_j^{\WMR{(1)}} + \gamma_j^{\WMR{(2)}} = x_{j} , \quad j \in J
\vspace{-0.05in}
\end{equation}
are optimality cuts for problem \textbf{R-MILP}.
}

\noindent
\begin{proof} %{Proof of Theorem \ref{thm_avg_resp}}
The sum $\gamma^{\WMR{(1)}}_j+\gamma^{\WMR{(2)}}_j$ can never exceed 1 due to \eqref{NEW1}.
If $\gamma^{\WMR{(1)}}_j+\gamma^{\WMR{(2)}}_j =1$, it follows from \eqref{open_drone-2} that $x_j=1$, 
If $\gamma^{\WMR{(1)}}_j+\gamma^{\WMR{(2)}}_j =0$, \eqref{open_drone-2} allows 
$x_j$ to be equal to 0 or 1. However, \eqref{VI5} cuts off the integer solutions $(x_j,\gamma_j^{\WMR{(1)}},\gamma_j^{\WMR{(2)}}) = (1,0,0), j \in J$, which are not optimal and do not give a better objective value than $(x_j, \gamma_j^{\WMR{(1)}},\gamma_j^{\WMR{(2)}}) = (0,0,0), j \in J$. 
\hfill$\Box$
\end{proof}
%%%%%%%%%%%%%%%%%%%%%%%%%%%%%%%%%
%%%%%%%%%%%%%%%%%%%%%%%%%%%%%%%%%

%%%%%%%%%%%%%%%%%%%%%%%%%%%%%%%%%%
%%%%%%%%%%%%%%%%%%%%%%%%%%%%%%%%%%
\vspace{-0.1in}
\section{Drone-delivery Network for Opioid Overdoses- Illustration}\label{APP-ILLU}
\vspace{-0.09in}
We use a small network to illustrate the formulations presented in this study.
%To further illustrate each term of the objective function in \textbf{B-IFP}, we develop and present a numerical example in this section. 
%\begin{illustration}: \label{illus1} 
We consider two candidate locations for drone bases (DB $j$=1 and $j$=2) which can each house at most two ($M$) drones. 
There are three OTRs ($i=1,2,3$). Requests $i=1$ and $i=2$  are within the catchment area of a drone positioned at DB $j=1$ while
requests $i=2$ and $i=3$  are within the catchment area of a drone at DB $j=2$. Figure \ref{fig-illustration} provides a visualization of the network. 
The solid circle represents DBs while the dotted circles represent catchment areas of a DB and its drones. The diamonds are the locations of the overdose incidents.
\vspace{-0.1in}
\begin{figure}[H]
\centering
\includegraphics[width=11cm, height=6.25cm]{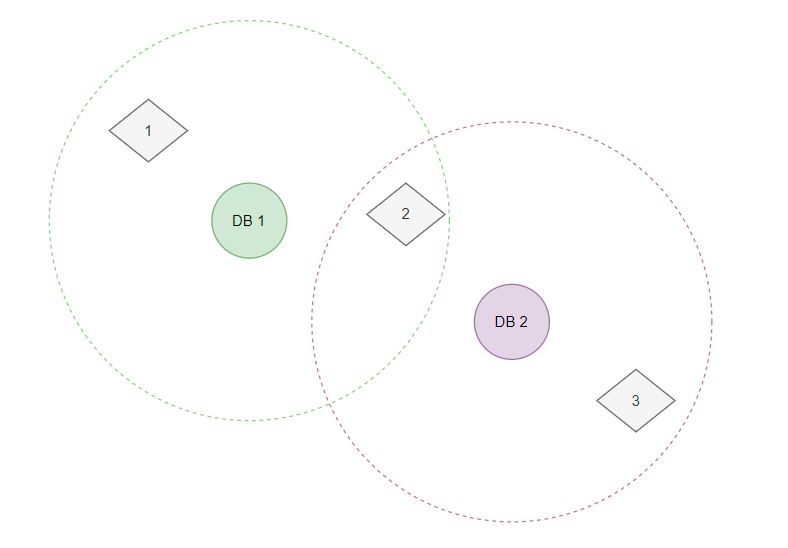}
% \begin{minipage}{.5\textwidth}
\vspace{-0.07in}
\caption{Small Drone Network and OTRs}
\label{fig-illustration}
% \end{minipage}
\end{figure}
%\end{illustration}

%\newpage
The objective function \eqref{D1_obj} of problem  \textbf{B-IFP} can be written as:
\begin{subequations}
\begin{align}
\min  & \; \frac{1}{\lambda_1 + \lambda_2 + \lambda_3} 
\Bigg[
 \frac{\lambda_1 d_{11}y_{11}}{v} + \frac{\lambda_2 d_{21}y_{21}}{v} + \frac{\lambda_2 d_{22}y_{22}}{v} + \frac{\lambda_3 d_{32}y_{32}}{v} +  
\notag \\
&\hspace{-0.6cm}    \frac{\lambda_1 y_{11} (\lambda_1 y_{11}\mathbb{E}[{S}_{11}^2]+ \lambda_2 y_{21}\mathbb{E}[{S}_{21}^2])  (\lambda_1 y_{11}
\mathbb{E}[S_{11}] + \lambda_2 y_{21}
\mathbb{E}[S_{21}])} {2(2-  (\lambda_1 y_{11}\mathbb{E}[S_{11}] + \lambda_2 y_{21}\mathbb{E}[S_{21}]  ))^2 \big[  1 +  
 (\lambda_1 y_{11}\mathbb{E}[S_{11}] + \lambda_2 y_{21}\mathbb{E}[S_{21}]
) + \frac{(\lambda_1 y_{11}
\mathbb{E}[S_{11}] + \lambda_2 y_{21}
\mathbb{E}[S_{21}])^2}
{2-  (\lambda_1 y_{11}\mathbb{E}[S_{11}] + \lambda_2 y_{21}\mathbb{E}[S_{21}]  )}\big]} +
\notag \\
&\hspace{-0.6cm}    \frac{\lambda_2 y_{21} (\lambda_1 y_{11}\mathbb{E}[{S}_{11}^2]+ \lambda_2 y_{21}\mathbb{E}[{S}_{21}^2])  (\lambda_1 y_{11}
\mathbb{E}[S_{11}] + \lambda_2 y_{21}
\mathbb{E}[S_{21}])} {2(2-  (\lambda_1 y_{11}\mathbb{E}[S_{11}] + \lambda_2 y_{21}\mathbb{E}[S_{21}]  ))^2 \big[  1 +  
 (\lambda_1 y_{11}\mathbb{E}[S_{11}] + \lambda_2 y_{21}\mathbb{E}[S_{21}]
) + \frac{(\lambda_1 y_{11}
\mathbb{E}[S_{11}] + \lambda_2 y_{21}
\mathbb{E}[S_{21}])^2}
{2-  (\lambda_1 y_{11}\mathbb{E}[S_{11}] + \lambda_2 y_{21}\mathbb{E}[S_{21}]  )}\big]} +
\notag  \\
&\hspace{-0.6cm}    \frac{\lambda_2 y_{22} (\lambda_2 y_{22}\mathbb{E}[{S}_{22}^2]+ \lambda_3 y_{32}\mathbb{E}[{S}_{32}^2])  (\lambda_2 y_{22}
\mathbb{E}[S_{22}] + \lambda_3 y_{32}
\mathbb{E}[S_{32}])} {2(2-  (\lambda_2 y_{22}\mathbb{E}[S_{22}] + \lambda_3 y_{32}\mathbb{E}[S_{32}]  ))^2 \big[  1 +  
 (\lambda_2 y_{22}\mathbb{E}[S_{22}] + \lambda_3 y_{32}\mathbb{E}[S_{32}]
) + \frac{(\lambda_2 y_{22}
\mathbb{E}[S_{22}] + \lambda_3 y_{32}
\mathbb{E}[S_{32}])^2}
{2-  (\lambda_2 y_{22}\mathbb{E}[S_{22}] + \lambda_3 y_{32}\mathbb{E}[S_{32}]  )}\big]} +
\notag  \\
&\hspace{-0.6cm}    \frac{\lambda_3 y_{32} (\lambda_2 y_{22}\mathbb{E}[{S}_{22}^2]+ \lambda_3 y_{32}\mathbb{E}[{S}_{32}^2])  (\lambda_2 y_{22}
\mathbb{E}[S_{22}] + \lambda_3 y_{32}
\mathbb{E}[S_{32}])} {2(2-  (\lambda_2 y_{22}\mathbb{E}[S_{22}] + \lambda_3 y_{32}\mathbb{E}[S_{32}]  ))^2 \big[  1 +  
 (\lambda_2 y_{22}\mathbb{E}[S_{22}] + \lambda_3 y_{32}\mathbb{E}[S_{32}]
) + \frac{(\lambda_2 y_{22}
\mathbb{E}[S_{22}] + \lambda_3 y_{32}
\mathbb{E}[S_{32}])^2}
{2-  (\lambda_2 y_{22}\mathbb{E}[S_{22}] + \lambda_3 y_{32}\mathbb{E}[S_{32}]  )}\big]}
\Bigg]
\notag
\end{align}
\end{subequations}

%%%%%%%%%%%%%%%%%%%%%%%%%%%%%%%%%%%%%%%%%%%%%%%
\vspace{0.1in}
The objective function \eqref{OBJ2} of problem  \textbf{R-BFP} can be written as:
\begin{subequations}
\begin{align}
\min  & \; \frac{1}{\lambda_1 + \lambda_2 + \lambda_3} 
\Bigg[
 \frac{\lambda_1 d_{11}y_{11}}{v} + \frac{\lambda_2 d_{21}y_{21}}{v} + \frac{\lambda_2 d_{22}y_{22}}{v} + \frac{\lambda_3 d_{32}y_{32}}{v} +  
\notag \\
&\hspace{-0.6cm}    \frac{\lambda_1 y_{11} \gamma_{1}^{\WMR{(1)}} (\lambda_1 y_{11}\mathbb{E}[{S}_{11}^2]+ \lambda_2 y_{21}\mathbb{E}[{S}_{21}^2])} {2(1-  (\lambda_1 y_{11}\mathbb{E}[S_{11}] + \lambda_2 y_{21}\mathbb{E}[S_{21}]  ))^2 \big[  1 +  
  \frac{\lambda_1 y_{11}
\mathbb{E}[S_{11}] + \lambda_2 y_{21}
\mathbb{E}[S_{21}]}
{1-  (\lambda_1 y_{11}\mathbb{E}[S_{11}] + \lambda_2 y_{21}\mathbb{E}[S_{21}]  )}\big]} +
\notag \\
&\hspace{-0.6cm}    \frac{\lambda_1 y_{11} \gamma_{1}^{\WMR{(2)}} (\lambda_1 y_{11}\mathbb{E}[{S}_{11}^2]+ \lambda_2 y_{21}\mathbb{E}[{S}_{21}^2])  (\lambda_1 y_{11}
\mathbb{E}[S_{11}] + \lambda_2 y_{21}
\mathbb{E}[S_{21}])} {2(2-  (\lambda_1 y_{11}\mathbb{E}[S_{11}] + \lambda_2 y_{21}\mathbb{E}[S_{21}]  ))^2 \big[  1 +  
 (\lambda_1 y_{11}\mathbb{E}[S_{11}] + \lambda_2 y_{21}\mathbb{E}[S_{21}]
) + \frac{(\lambda_1 y_{11}
\mathbb{E}[S_{11}] + \lambda_2 y_{21}
\mathbb{E}[S_{21}])^2}
{2-  (\lambda_1 y_{11}\mathbb{E}[S_{11}] + \lambda_2 y_{21}\mathbb{E}[S_{21}]  )}\big]} +
\notag \\
&\hspace{-0.6cm}    \frac{\lambda_2 y_{21} \gamma_{1}^{\WMR{(1)}} (\lambda_1 y_{11}\mathbb{E}[{S}_{11}^2]+ \lambda_2 y_{21}\mathbb{E}[{S}_{21}^2])} {2(1-  (\lambda_1 y_{11}\mathbb{E}[S_{11}] + \lambda_2 y_{21}\mathbb{E}[S_{21}]  ))^2 \big[  1 +  
  \frac{(\lambda_1 y_{11}
\mathbb{E}[S_{11}] + \lambda_2 y_{21}
\mathbb{E}[S_{21}])}
{1-  (\lambda_1 y_{11}\mathbb{E}[S_{11}] + \lambda_2 y_{21}\mathbb{E}[S_{21}]  )}\big]} +
\notag  \\
&\hspace{-0.6cm}    \frac{\lambda_2 y_{21} \gamma_{1}^{\WMR{(2)}} (\lambda_1 y_{11}\mathbb{E}[{S}_{11}^2]+ \lambda_2 y_{21}\mathbb{E}[{S}_{21}^2])  (\lambda_1 y_{11}
\mathbb{E}[S_{11}] + \lambda_2 y_{21}
\mathbb{E}[S_{21}])} {2(2-  (\lambda_1 y_{11}\mathbb{E}[S_{11}] + \lambda_2 y_{21}\mathbb{E}[S_{21}]  ))^2 \big[  1 +  
 (\lambda_1 y_{11}\mathbb{E}[S_{11}] + \lambda_2 y_{21}\mathbb{E}[S_{21}]
) + \frac{(\lambda_1 y_{11}
\mathbb{E}[S_{11}] + \lambda_2 y_{21}
\mathbb{E}[S_{21}])^2}
{2-  (\lambda_1 y_{11}\mathbb{E}[S_{11}] + \lambda_2 y_{21}\mathbb{E}[S_{21}]  )}\big]} +
\notag \\
&\hspace{-0.6cm}    \frac{\lambda_2 y_{22} \gamma_{2}^{\WMR{(1)}} (\lambda_2 y_{22}\mathbb{E}[{S}_{22}^2]+ \lambda_3 y_{32}\mathbb{E}[{S}_{32}^2])  } {2(1-  (\lambda_2 y_{22}\mathbb{E}[S_{22}] + \lambda_3 y_{32}\mathbb{E}[S_{32}]  ))^2 \big[  1 +  
   \frac{\lambda_2 y_{22}
\mathbb{E}[S_{22}] + \lambda_3 y_{32}
\mathbb{E}[S_{32}]}
{1-  (\lambda_2 y_{22}\mathbb{E}[S_{22}] + \lambda_3 y_{32}\mathbb{E}[S_{32}]  )}\big]} + 
\notag \\
&\hspace{-0.6cm}    \frac{\lambda_2 y_{22} \gamma_{2}^{\WMR{(2)}} (\lambda_2 y_{22}\mathbb{E}[{S}_{22}^2]+ \lambda_3 y_{32}\mathbb{E}[{S}_{32}^2])  (\lambda_2 y_{22}
\mathbb{E}[S_{22}] + \lambda_3 y_{32}
\mathbb{E}[S_{32}])} {2(2-  (\lambda_2 y_{22}\mathbb{E}[S_{22}] + \lambda_3 y_{32}\mathbb{E}[S_{32}]  ))^2 \big[  1 +  
 (\lambda_2 y_{22}\mathbb{E}[S_{22}] + \lambda_3 y_{32}\mathbb{E}[S_{32}]
) + \frac{(\lambda_2 y_{22}
\mathbb{E}[S_{22}] + \lambda_3 y_{32}
\mathbb{E}[S_{32}])^2}
{2-  (\lambda_2 y_{22}\mathbb{E}[S_{22}] + \lambda_3 y_{32}\mathbb{E}[S_{32}]  )}\big]} +
\notag  \\
&\hspace{-0.6cm}    \frac{\lambda_3 y_{32} \gamma_{2}^{\WMR{(1)}} (\lambda_2 y_{22}\mathbb{E}[{S}_{22}^2]+ \lambda_3 y_{32}\mathbb{E}[{S}_{32}^2])  } {2(1-  (\lambda_2 y_{22}\mathbb{E}[S_{22}] + \lambda_3 y_{32}\mathbb{E}[S_{32}]  ))^2 \big[  1 +  
   \frac{\lambda_2 y_{22}
\mathbb{E}[S_{22}] + \lambda_3 y_{32}
\mathbb{E}[S_{32}]}
{1-  (\lambda_2 y_{22}\mathbb{E}[S_{22}] + \lambda_3 y_{32}\mathbb{E}[S_{32}]  )}\big]} +
\notag \\
&\hspace{-0.6cm}    \frac{\lambda_3 y_{32} \gamma_{2}^{\WMR{(2)}} (\lambda_2 y_{22}\mathbb{E}[{S}_{22}^2]+ \lambda_3 y_{32}\mathbb{E}[{S}_{32}^2])  (\lambda_2 y_{22}
\mathbb{E}[S_{22}] + \lambda_3 y_{32}
\mathbb{E}[S_{32}])} {2(2-  (\lambda_2 y_{22}\mathbb{E}[S_{22}] + \lambda_3 y_{32}\mathbb{E}[S_{32}]  ))^2 \big[  1 +  
 (\lambda_2 y_{22}\mathbb{E}[S_{22}] + \lambda_3 y_{32}\mathbb{E}[S_{32}]
) + \frac{(\lambda_2 y_{22}
\mathbb{E}[S_{22}] + \lambda_3 y_{32}
\mathbb{E}[S_{32}])^2}
{2-  (\lambda_2 y_{22}\mathbb{E}[S_{22}] + \lambda_3 y_{32}\mathbb{E}[S_{32}]  )}\big]}
\Bigg]
\notag
\end{align}
\end{subequations}

The objective function \eqref{obj_lin} of problem  \textbf{R-MILP} can be written as:
\begin{subequations}
\begin{align}
\min  & \; \frac{1}{\lambda_1 + \lambda_2 + \lambda_3} 
\Bigg[
 \frac{\lambda_1 d_{11}y_{11}}{v} + \frac{\lambda_2 d_{21}y_{21}}{v} + \frac{\lambda_2 d_{22}y_{22}}{v} + \frac{\lambda_3 d_{32}y_{32}}{v} +  
\notag \\
&    \lambda_1(\frac{\omega_{11}^{\WMR{(1)}}}{2} + \frac{\omega_{11}^{\WMR{(2)}}}{2}) + \lambda_2(
\frac{\omega_{21}^{\WMR{(1)}}}{2} + \frac{\omega_{21}^{\WMR{(2)}}}{2} + \frac{\omega_{22}^{\WMR{(1)}}}{2} + \frac{\omega_{22}^{\WMR{(2)}}}{2}) +
\lambda_{3}(
\frac{\omega_{32}^{\WMR{(1)}}}{2} + \frac{\omega_{32}^{\WMR{(2)}}}{2}) \notag
\Bigg]
\notag
\end{align}
\end{subequations}

%%%%%%%%%%%%%%%%%%%%%%%%%%%%%%%%%%%%%%%%%
%%%%%%%%%%%%%%%%%%%%%%%%%%%%%%%%%%%%%%%%%
%%%%%%%%%%%%%%%%%%%%%%%%%%%%%%%%%%%%%%%%%
%%%%%%%%%%%%%%%%%%%%%%%%%%%%%%%%%%%%%%%%%
\vspace{-0.2in}
%\newpage
\section{Size of Formulations} \label{SIZE}
\vspace{-0.2in}
\begin{table}[h!t]
	\centering
	%\vspace{-0.1in}
	\begin{adjustbox}{width=0.85\textwidth}
		%	\resizebox{\textwidth}{!}{%
		\begin{tabular}{c|c|c}
			\hline
			& $\mathbf{B-IFP}$ & $\mathbf{R-BFP}$ \\ 
\hline
%			\Xhline{2\arrayrulewidth}
Number of  constraints	    
& $\sum_{i \in I}|J_i| + |I|+2|J|+2$  
& $ \sum_{i \in I}|J_i| + |I|+3|J|+2 $\\
\hline
Number of  binary variables	& $\sum_{i \in I}|J_i| + |J|$ &  
$ \sum_{i \in I}|J_i| + (M + 1)|J| $ \\
\hline
Number of  general integer variables & $|J|$ & 0 \\
\hline
		\end{tabular}%
	\end{adjustbox}
\caption{Dimensions of Problems $\mathbf{B-IFP}$ and  $\mathbf{R-BFP}$}
	\label{T01}	
\end{table}

\vspace{-0.2in}
\section{Pseudo-code of Outer Approximation Branch-and-Cut Algorithm {\tt OA-B\&C}} \label{A-PS}
\vspace{-0.15in}
\begin{algorithm}[] 
	\caption{Outer Approximation Branch-and-Cut Algorithm (OA-B\&C)} \label{algo_oa-bc}
	{\small
		\begin{algorithmic} 
			\State \textbf{Part 1 (Initialization)}:
			$\mathcal{L}_0:= 
			\{\mathcal{F}_y^z \cup \mathcal{M}^{\tau}_{zU}\}$;
			%\mathcal{L}_0:= \{\eqref{MAC_z1}-\eqref{MAC_z4} ; \eqref{MAC_psi1} - \eqref{MAC_psi4}\}$; 
			\ $\mathcal{B} = \{\eqref{VI5}\};$ \ $\mathcal{L'}_0:= \mathcal{L}_0 \cup \mathcal{B}$;
			\ $\mathcal{U}_0:= \{\eqref{VI3};\eqref{VI4} \}$;
			\ $\mathcal{A}_0:= C \setminus \mathcal{L}_0$; $\mathcal{V}^L_0 = \emptyset$;
			\ $\mathcal{V}^U_0 = \emptyset$.
			\State \textbf{Part 2 (Iterative Procedure): At node $k$} 
			\State \quad \textbf{Step 1: Solution of nodal relaxation problem:} 
		$	\mathbf{OA-MILP}_k: \; \min \eqref{obj_lin} \quad \text{s.to} \quad (x,y,\gamma,z,U,\mu,\tau,\omega) \in \mathcal{A}_k  \ . $

			\State \quad \textbf{Step 2: Set Update:} \\
			
			\begin{itemize}
				\item \textbf{If the objective value corresponding to $X^*_k$ is not better than that of the incumbent}, the node is pruned.
				\item \textbf{If the objective value corresponding to $X^*_k$ is better than that of the incumbent}, then:
				\begin{itemize}

					%%%%%%%%%%%%%%%%%%%%%%%%%%%%%%
					\item \textbf{If  $X^*_k$ is fractional}, apply user callback:
					
					\begin{itemize}
						\item If $X_k^{*}$ violates any constraint in $\mathcal{U}_o$: 
						\begin{itemize}
							\item Move violated constraints to $\mathcal{V}_k^{U}$ and discard $X_k^{*}$.
							
							\item Update sets of user cuts and active constraints for each open node $o \in \mathcal {O}$
							\[
							\mathcal{U}_o \leftarrow \mathcal{U}_o  \setminus  \mathcal{V}^U_k \quad \text{and} \quad
							\mathcal{A}_o \leftarrow \mathcal{A}_o  \cup \mathcal{V}^U_k.
							\]   
						\end{itemize}
						% \end{itemize}
					\item If no valid inequality in $\mathcal{U}_o$ is violated by $X^*_k$, then branching constraints are entered to cut off $X^*_k$ and the next open node is processed. 
				\end{itemize}
				% \end{itemize}
			
		 \item 	\textbf{If $X^*_k$ is integer-valued},
    %and the objective value is better than the incumbent one}, 
    check for possible violation of lazy constraints: 
    \begin{itemize}
    \item If $X^*_k$ violates any constraint in $\mathcal{L}^{'}_o$:
    
    \begin{itemize}
    \item Move violated lazy constraints to $\mathcal{V}^L_k$ and discard $X^*_k$. 

    \item  Update sets of lazy and active constraints for each open node $o  \in \mathcal {O}$: 
\[
    \mathcal{L}^{'}_o \leftarrow \mathcal{L}^{'}_o \setminus \mathcal{V}^L_k 
    \qquad \text{and} \qquad 
    \mathcal{A}_o \leftarrow \mathcal{A}_o  \cup \mathcal{V}^L_k.
    \]
\end{itemize}
    
    \item If $X^*_k$ does not violate any constraint in $\mathcal{L}^{'}_o$, $X^*_k$ becomes the incumbent and node $k$ is pruned.
\end{itemize}
\end{itemize}
			%\item \textbf{If $X^*_k$ is integer-valued}, update the sets of lazy and active constraints according to Step 2 (Part 2) of Algorithm \ref{algo_oa}. 
		%\end{itemize}
		
	\end{itemize}
	\State \textbf{Part 3 (Termination):} The algorithm stops when $\mathcal{O} = \emptyset$.
\end{algorithmic}
}
\end{algorithm}

%%%%%%%%%%%%%%%%%%%%%%%%%%%%%%%%%%%%%%%
%%%%%%%%%%%%%%%%%%%%%%%%%%%%%%%%%%%%%%%
%%%%%%%%%%%%%%%%%%%%%%%%%%%%%%%%%
%%%%%%%%%%%%%%%%%%%%%%%%%%%%%%%%%
\vspace{-0.2in}
\section{Outer Approximation Algorithm with Lazy Constraints}
\label{APP-sub_sec_lazy_c}
\vspace{-0.05in}

The outer approximation algorithm {\tt OA} is designed as follows. 
At the root node ($k=0$), we have:
\vspace{-0.05in}
\begin{align}
& \mathcal{L}_0:= 
\{\mathcal{F}_y^z \cup \mathcal{M}^{\tau}_{zU}\}.
\nonumber \\ %\label{S1} \\
& \mathcal{A}_0:= \{
\mathcal{B} ; \eqref{U}-\eqref{V} ;  \mathcal{M}^{\mu}_{yU} ; \mathcal{M}^{\omega}_{\gamma \mu}  \}. 
\nonumber \\ %\label{S2} \\
&\mathcal{V}_0^L:= \emptyset. 
\nonumber  % \label{S3}
\end{align}
\vspace{-0.15in}
At any node $k$, % of the branch-and-bound tree, 
the reduced-size relaxation (outer approximation) problem $\mathbf{OA-MILP}_k$ is solved:
\[
\mathbf{OA-MILP}_k: \; \min \eqref{obj_lin} \quad \text{s.to} \quad (x,y,\gamma,z,U,\mu,\tau,\omega) \in \mathcal{A}_k \}. 
\vspace{-0.05in}
\]
Two possibilities arise depending on the optimal solution  $X^*_k$ %=(x^*,y^*,\gamma^*,z^*,U^*,\mu^*,\tau^*,\omega^*)$
of the continuous relaxation of $\mathbf{OA-MILP}_k$:
\begin{enumerate}
    \item If  $X^*_k$ is fractional, we introduce branching linear inequalities to cut off the fractional nodal optimal solution and we continue the branch-and-bound process. 
    \item If $X^*_k$ is an integer-valued solution with better objective value than the one of the incumbent, we check for possible violation of the current lazy constraints: 
    \begin{itemize}
    \item If some constraints in $\mathcal{L}_k$ are violated by    $X^*_k$, they are inserted in $\mathcal{V}^L_k \subseteq \mathcal{L}_k$ and $X^*_k$ is discarded.   
    %pull the violated lazy constraints in the model 
    The lazy and active constraint sets of each open node $o \in \mathcal {O}$ are updated as follows:  
    \[
    \mathcal{L}_o \leftarrow \mathcal{L}_o  \setminus \mathcal{V}^L_k  \quad \text{and} \quad     
    \mathcal{A}_o \leftarrow \mathcal{A}_o  \cup \mathcal{V}^L_k.
    \]
    \item If no lazy constraint 
    %in $\mathcal{L}_k$ 
    is violated, % by $X^*_k$, 
    $X^*_k$ becomes the incumbent solution %for the true problem 
    and the node~is~pruned.
    \end{itemize}
\end{enumerate}

The above process terminates when all nodes are pruned. The verification of the possible violation of the lazy constraints is carried out within a callback function, which is not performed at each node of the tree, but only when a better integer-valued feasible solution is found.

%%%%%%%%%%%%%%%%%%%%%%%%%%%%%%%%%%%%%%%
%%%%%%%%%%%%%%%%%%%%%%%%%%%%%%%%%%%%%%%
\newpage
%\vspace{-0.12in}
\section{Pseudo-Code of Simulator \label{A-Sim}}
\vspace{-0.12in}
\ML{
%Throughout this paper, 
We leverage the simulator %(see Section \ref{sec_sim}) 
to generate and compare the average response times obtained with our model {\bf R-MILP}, an alternative formulation {\bf BM}, and two constructive greedy heuristics {\bf H-DB and \bf H-OTR}. 

To ensure that the steady state condition holds, we follow the procedure adopted in two recent EMS studies \cite{BoutilierOR2020,CORMACK} and process first a buffer of overdose requests prior to the period of interest. Experiments show that a buffer of at least five hours of OTR requests produces results that are independent of buffer duration. 
For each test, we run the simulation 100 times and we record the mean as well as the 5th and 95th percentiles of the response times.
% which allows us to derive the 90\% confidence interval. 
%Preliminary experiments have shown that the sampling average response time converges quickly and that 100 replications are more than enough.
As in \cite{EVERS} for the stochastic UAV mission planning with uncertain service times, we sample the service time from a gamma distribution with shape parameter equal to 4 and mean equal to the expected value of $\mathbb E[S_{ij}]$. %which depends on the assignment policy. 
%To test the robustness of the results against the distribution type, we have also run the simulation with a normally distributed service time as suggested in \cite{EVERS}. 
%The normal distribution has the same mean and variance as the gamma distribution. As shown in Table 1 below, R-MILP performs very similarly with the two distributions, which shows that the chosen M/G/K approximation formula is robust and can be used for different distributions.

}

%Algorithm \eqref{algo_sim} presents the pseudo-code for the simulation. 
Algorithm \ref{algo_sim} present the pseudo-code of the simulator. We use $t_c$ to denote the current time, $t_f$ to denote the drone flight time from its base to OTR or from OTR to its base, and $t_s$ to denote the combined on-scene service time and time to clean and recharge drones.

%\vspace{-0.05in}
\begin{algorithm}[H] 
\caption{Drone Network Simulator} \label{algo_sim}
% 	{\small
\begin{algorithmic}
\State {$O \gets$ Load all OTRs with time and location}
\State {$D \gets$ Load all deployed drones}
\State {function Simulate$(O, D):$}
    \Indent
    \State {$E \gets O$} \Comment{Initialize event queue with OTRs} 
    \State {$A \gets D$} \Comment{Initialize list of available drones}  
    \State {$Q \gets \emptyset$} \Comment{Initialize empty call queue}
    \While {$|E| > 0:$}
        \State {Remove next event $e$ from $E$} 
        \State{update the current time $t_c$}
        \If {$e$ is an OTR:}
            \If {$A$ has no drone available within the flight radius}
                \State {$Q \gets Q + e$} \Comment{Queue incoming OTR}
            \Else {\ Dispatch the closet drone d}
                \State {$A \gets A - d$} 
                \Comment{Remove drone from $A$}
                \State {$e_{new}$}
                \Comment{Create new event when drone is available after task completion}
                \State{
                Insert $e_{new}$ to $E$ at time $t = t_c + t_f + t_s + t_f$}
            \EndIf
        \Else { $e$ is a drone available event:}
            \State {$A \gets A + d$} \Comment{Add drone to available list}
            \If {$|Q| > 0$ and is within radius} 
            \Comment{If a drone can be assigned to OTR in Q}
            \State {Dispatch closest drone d}
            \State {$A \gets A - d'$}
            \State {$e_{new}$}
            \Comment{Create new event when drone is available after task completion}
            \State{
            Insert $e_{new}$ to $E$ at time $t = t_c + t_f + t_s + t_f$}
            \EndIf
        \EndIf
    \EndWhile
\EndIndent
\end{algorithmic}
    % }
\end{algorithm}
%%%%%%%%%%%%%%%%%%%%%%%%%%%%%%%%%
%%%%%%%%%%%%%%%%%%%%%%%%%%%%%%%%%
%%%%%%%%%%%%%%%%%%%%%%%%%%%%%%%%%%%%%%%%%%%%%%%%%%%%
\vspace{-0.25in}
\section{Data Summary}\label{app-sec-data-summary}
\vspace{-0.15in}
\begin{table}[H]
\centering
\setlength\extrarowheight{1pt}
\begin{tabular}{P{3.5cm} |P{2.5cm}|P{2.5cm}|P{2.5cm}|P{2.5cm}}
 \hline
Pairs of Sets & Training Sets & $|I|$ - Training & Testing Sets & $|I|$ - Testing\\
 \hline
(2018Q2, 2018Q3) & 2018Q2 &173  & 2018Q3  &  131\\
(2018Q3, 2018Q4) & 2018Q3 &131  & 2018Q4  &  120\\
(2018Q4, 2019Q1) & 2018Q4 &120  & 2019Q1  &  138\\
(2019Q1, 2019Q2) & 2019Q1 &138  & 2019Q2  &  171\\
 \hline
\end{tabular}
\vspace{-0.05in}
\caption{\label{data_summary} Data Summary}
\end{table}

%%%%%%%%%%%%%%%%%%%%%%%%%%%%%%%%%%%%%%%%%%%%%%%%%%%%%%%%
%%%%%%%%%%%%%%%%%%%%%%%%%%%%%%%%%%%%%%%%%%%%%%%%%%%%
\section{\textcolor{blue}{Confidence Interval for Response Times of R-MILP}}
\label{APP-CI}

\begin{table}[H]
\centering
\setlength\extrarowheight{1pt}
\ML{
\begin{tabular}{P{3cm}|P{3cm}|P{4.5cm}| P{3cm}}
 \hline
\multirow{2}{*}{Training Set} & \multicolumn{3}{c}{Response Time (min.) for \bf{R-MILP}}  \\ \cline{2-4}  
\multirow{2}{*}{} & 5th Percentile
& Mean & 95th Percentile \\
 \hline
2018Q2 & \textcolor{black}{0.80} & 1.49 & 2.19    \\
2018Q3 & \textcolor{black}{0.67} & 1.58 & 2.60  \\ 
2018Q4 & \textcolor{black}{0.50} & 1.54 & 2.49  \\
2019Q1 & \textcolor{black}{0.45} & 1.50  & 2.50 \\
\hline
Quarter Average & 0.61 & 1.53 & 2.45 \\
 \hline
\end{tabular}
\vspace{-0.04in}
\caption{\label{ci_train_milp} Response Times for Quarterly Training Sets ($q = 10, p = 11$)}
}
\end{table}

\begin{table}[H]
\centering
\setlength\extrarowheight{1pt}
\ML{
\begin{tabular}{P{3.2cm}|P{3cm}|P{4.5cm}| P{3cm}}
 \hline
\multirow{2}{*}{Training/Testing Data} & \multicolumn{3}{c}{Testing  Quarter Response Time (min.) \bf{R-MILP}}  \\ \cline{2-4}  
\multirow{2}{*}{} & 5th Percentile
& Mean & 95th Percentile \\
 \hline
2018Q2 / 2018Q3 & \textcolor{black}{0.82} & 1.72 & 2.79    \\
2018Q3 / 2018Q4 & \textcolor{black}{0.54} & 1.57 & 2.55  \\ 
2018Q4 / 2019Q1 & \textcolor{black}{0.74} & 1.80 & 2.83  \\
2019Q1 / 2019Q2 & \textcolor{black}{0.57} & 1.48  & 2.38 \\
\hline
Quarter Average & 0.67 & 1.64 & 2.64 \\
 \hline
\end{tabular}
\vspace{-0.04in}
\caption{\label{ci_test_milp} Response Times for Quarterly Testing Sets ($q = 10, p = 11$)}
}
\end{table}

%%%%%%%%%%%%%%%%%%%%%%%%%%%%%%%%%%%%%%%%%%%%%%%%%%%%%%%%
\newpage

\section{Pseudo-codes  of Constructive Greedy Heuristics Approaches \label{Constructive_baseline}}
\vspace{-0.15in}
%To examine the added benefit from our modeling framework, we develop two benchmark networks through constructive heuristics . The first heuristic focuses on the significance of OTRs (\textbf{\textbf{H-OTR}}). Given OTR locations have different arrival rates, we use the arrival rate $\lambda_{i}$ to represent the significance of each $i$. That is location $i$ with a higher $\lambda_{i}$ is more significant than the one with a lower $\lambda_{i}$ because more overdoses will occur at $i$. Algorithm \eqref{algo_baseline_hotr} presents the heuristic \textbf{\textbf{H-OTR}} to build a baseline network with $10$ bases $11$ drones.
\begin{algorithm}[] 
\caption{OTR-Focused Heuristic \textbf{\textbf{H-OTR}}} \label{algo_baseline_hotr}
% 	{\small
\begin{algorithmic}
\State {$I \gets$ Load all OTR locations }
\State {$J \gets$ Load candidate DBs}
% \State {}
\State {function BuildBaselineNetwork$(I, J):$}
    \Indent
    \State {$J_{o} \gets \emptyset$} \Comment{Initialize empty set of opened DBs} 
    \State {$I_{c} \gets \emptyset$}
    \Comment{Initialize empty set of selected OTR locations}
    % \State {$D_{d} \gets \emptyset$} \Comment{Initialize an empty set of deployed drones}  
    % \State {$Q \gets \emptyset$} \Comment{Initialize an empty call queue}
    \While {$|J_{o}| < q:$}
        \State {$i^{*} =  \arg \max_{i}\  \{\lambda_{i} | i \in I \}$}
        \Comment{Select highest-ranked OTR location $i^*$}
        % \State {For $j \in J \setminus J_{o}$ : 
        \State {$j^{*} =  \arg \min_{j}\  \{d_{i^{*}j} | j \in J \}$}
        \Comment{Open DB  $j^*$  closest to OTR $i^*$}
% \State{update the current time $t_c$}
        \State{$J_{o} \gets J_{o} + j^{*}$}
        \Comment{Add $j^*$ to the set opened bases}
        \State{$I \gets I \setminus \{i^*\}, J \gets J \setminus \{ j^{*}\}  $}
        \Comment{Remove opened DB $j^*$ and closest OTR  $i^*$ from candidate sets}
        \If{$|J_{o}| = 1$}
            \State{$\gamma_{j*}^{2} \gets 1$}
            \Comment{Assign 2 drones to DB which is closest to the highest-ranked OTR}
        \Else
            \State{$\gamma_{j*}^{1} \gets 1$}            \Comment{Assign 1 drone to other DBs}
        \EndIf
        % \State{$B_{o} \gets B_{o} + j^{*}$}
        % \State{$I_{c} \gets I_{c} + \tilde{I}_{j^*}$}
        % \If {$e$ is an OTR:}
        %     \If {$A$ has no drone available within the flight radius}
        %         \State {$Q \gets Q + e$} \Comment{Queue the OTR}
        %     \Else {\ Dispatch the closet drone d}
        %         \State {$A \gets A - d$} 
        %         \Comment{Remove the drone from $A$}
        %         \State {$e_{new}$}
        %         \Comment{Create a new event when drone is available after completing the task}
        %         \State{
        %         Insert $e_{new}$ to $E$ at time $t = t_c + t_f + t_s + t_f$}
        %     \EndIf
        % \Else { $e$ is a drone available event:}
        %     \State {$A \gets A + d$} \Comment{Add drone to available list}
        %     \If {$|Q| > 0$ and is within radius} 
        %     \Comment{If a drone can be assigned to an OTR in Q}
        %     \State {Dispatch the closest drone d}
        %     \State {$A \gets A - d'$}
        %     \State {$e_{new}$}
        %     \Comment{Create a new event when drone is available after completing the task}
        %     \State{
        %     Insert $e_{new}$ to $E$ at time $t = t_c + t_f + t_s + t_f$}
        %     \EndIf
        % \EndIf
    \EndWhile
\EndIndent
\end{algorithmic}
\end{algorithm}

%The second heuristic focuses on the significance of DBs (\textbf{H-DB}). Given DBs have varying number of covered OTRs, we iteratively open DBs with the highest OTRs within its radius until $10$ bases and $11$ drones are deployed. Algorithm \eqref{algo_baseline_hdb} describes the pseudocode for \textbf{H-DB}.
%}

%\vspace{-0.1in}
\begin{algorithm}[] 
\caption{DB-Focused Heuristic \textbf{H-DB} } \label{algo_baseline_hdb}
% 	{\small
\begin{algorithmic}
\State {$I \gets$ Load all OTR locations }
\State {$J \gets$ Load candidate DBs}
% \State {}
\State {function BuildBaselineNetwork$(I, J):$}
    \Indent
    \State {$J_{o} \gets \emptyset$} \Comment{Initialize empty set of opened DB} 
    \State {$I_{c} \gets \emptyset$}
    \Comment{Initialize empty set of selected OTR location}
    % \State {$D_{d} \gets \emptyset$} \Comment{Initialize an empty set of deployed drones}  
    % \State {$Q \gets \emptyset$} \Comment{Initialize an empty call queue}
    \While {$|J_{o}| < q:$}
        \State {$j^{*} =  \arg \max_{j}\  \{c_{j} | j \in J \}$}
        \Comment{Select DB $j^*$ with largest coverage} % within ts radius}
        \State {$i^{*} =  \arg \min_{i}\  \{d_{ij^{*}} | i \in I \}$}
        \Comment{Select the OTR location closest to $j^*$}
        % \State {For $j \in J \setminus J_{o}$ : 
        
% \State{update the current time $t_c$}
        \State{$J_{o} \gets J_{o} + j^{*}$}
        \Comment{Add $j^*$ to the set opened bases}
        \State{$I \gets I \setminus \{i^*\}, J \gets J \setminus \{j^*\}$}
        \Comment{Remove opened DB $j^*$ and closest OTR $i^*$ from candidate sets}
        \If{$|J_{o}| = 1$}
            \State{$\gamma_{j*}^{2} \gets 1$}
            \Comment{Assign 2 drones to DB closest to OTR with highest $\lambda_i$}
        \Else
            \State{$\gamma_{j*}^{1} \gets 1$}            \Comment{Assign 1 drone to other DBs.}
        \EndIf
        % \State{$B_{o} \gets B_{o} + j^{*}$}
        % \State{$I_{c} \gets I_{c} + \tilde{I}_{j^*}$}
        % \If {$e$ is an OTR:}
        %     \If {$A$ has no drone available within the flight radius}
        %         \State {$Q \gets Q + e$} \Comment{Queue the OTR}
        %     \Else {\ Dispatch the closet drone d}
        %         \State {$A \gets A - d$} 
        %         \Comment{Remove the drone from $A$}
        %         \State {$e_{new}$}
        %         \Comment{Create a new event when drone is available after completing the task}
        %         \State{
        %         Insert $e_{new}$ to $E$ at time $t = t_c + t_f + t_s + t_f$}
        %     \EndIf
        % \Else { $e$ is a drone available event:}
        %     \State {$A \gets A + d$} \Comment{Add drone to available list}
        %     \If {$|Q| > 0$ and is within radius} 
        %     \Comment{If a drone can be assigned to an OTR in Q}
        %     \State {Dispatch the closest drone d}
        %     \State {$A \gets A - d'$}
        %     \State {$e_{new}$}
        %     \Comment{Create a new event when drone is available after completing the task}
        %     \State{
        %     Insert $e_{new}$ to $E$ at time $t = t_c + t_f + t_s + t_f$}
        %     \EndIf
        % \EndIf
    \EndWhile
\EndIndent
\end{algorithmic}
\end{algorithm}

%%%%%%%%%%%%%%%%%%%%%%%%%%%%%%%%%%%%%%%%%%%%%%%%%%%%
\section{Comparison with Constructive Greedy Heuristics \label{Constructive_baseline2}}

\vspace{-0.25in}
\begin{figure}[H]
\includegraphics[width = \textwidth]{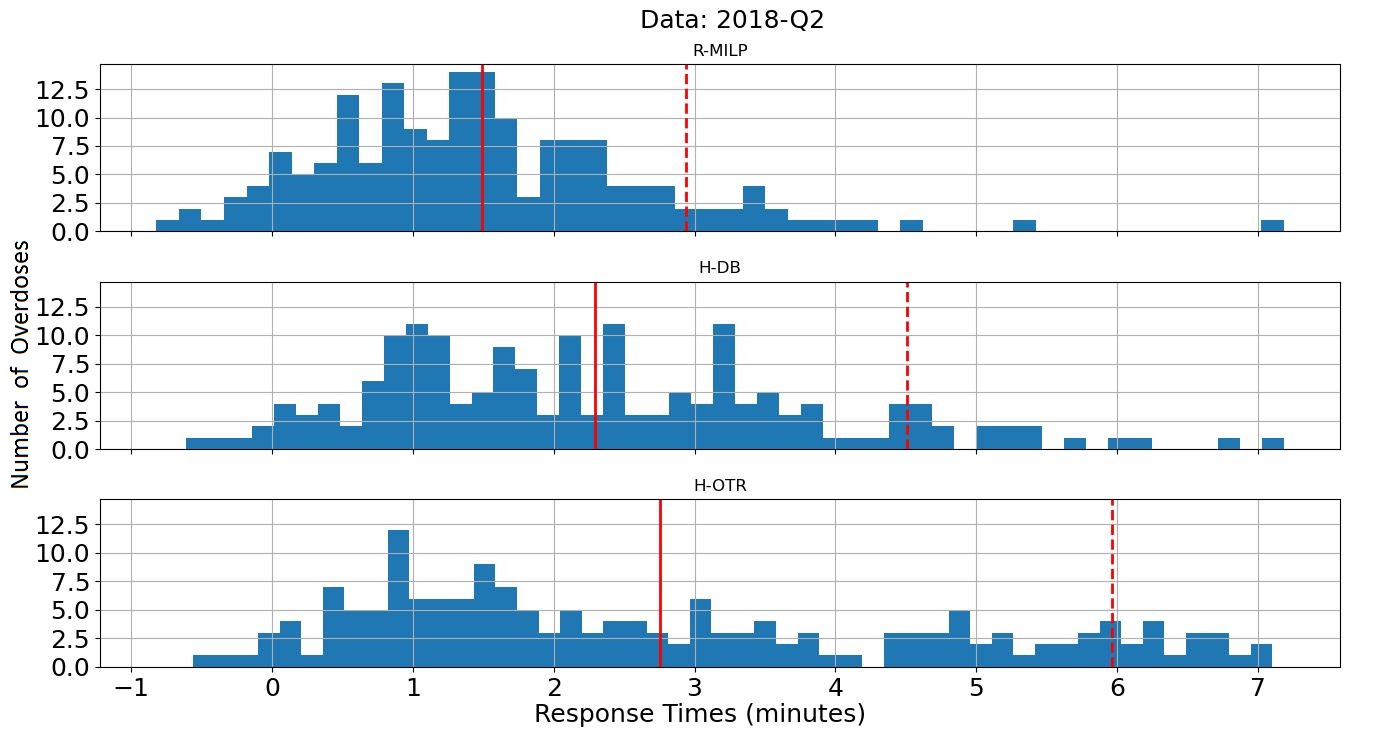}
\vspace{-0.25in}
\caption{\label{hist_heuristic_2018Q2} %(Color Online) 
Histogram of average simulated individual response times $R_{i}$ for each OTR from 100 runs. The solid (red) line is the average response times, and the dashed (red) line is the 90th percentile.}
\end{figure}

\vspace{-0.05in}
\section{Data for QALY and Cost Analysis} \label{DATA-COST}
\vspace{-0.1in}
\begin{table}[h]
\centering
\setlength\extrarowheight{-0.1pt}
%\resizebox{\columnwidth}{0.9cm}{
\begin{tabular}{P{4cm} |P{4cm}|P{3cm}|P{3cm}}
 \hline
 $T$ &  $\alpha$  & $c$ & T-QALY \\
 \hline
 11.4 years & 0.85 & 3\% & 8.47 years \\
 \hline
\end{tabular}
%}
\vspace{-0.05in}
\caption{\label{qaly} QALY Analysis for each OHCA Survivor}
\end{table}
% being , 305 (resp. 152 and 237) 

\begin{table}[H]
%\vspace{-0.1in}
\centering
\setlength\extrarowheight{-0.01pt}
%\resizebox{\columnwidth}{1cm}{
\begin{tabular}{P{2cm} |P{3.5cm}|P{2.5cm}|P{2cm}|P{3.5cm}}
 \hline
 Price per &  Annual & Number of & Discount & Total Discounted \\
 Drone &  Maintenance Cost & Drones & Rate & Cost in 4 Years. \\
  \hline
 \$15,000 & \$3,000 & 11 & 3\%  & \$287,664 \\
 \hline
\end{tabular}
%}
\vspace{-0.05in}
\caption{\label{dronecost} Cost for Drone Network}
\end{table}

%%%%%%%%%%%%%%%%%%%%%%%%%%%%%%%%%
%%%%%%%%%%%%%%%%%%%%%%%%%%%%%%%%%
%%%%%%%%%%%%%%%%%%%%%%%%%%%%%%%%%
%%%%%%%%%%%%%%%%%%%%%%%%%%%%%%%%%
%\newpage
%\setlength\extrarowheight{12.75pt}
%\setlength\extrarowheight{-2.5pt}
\vspace{-0.2in}
\section{Computational Efficiency Study}
\label{APP-CEFF}
%\subsection{Computational Efficiency Tests for Base Scenario Case ($q = 10, p = 11$)}
\newpage
\vspace{-0.15in}
\begin{table}[H]
\centering
\resizebox{\columnwidth}{11.5cm}{%
\begin{tabular}{P{2cm} | P{2cm} | P{2.5cm} | P{2.5cm} | P{2.5cm} | P{2.5cm} }
 \hline
\multirow{2}{*}{$|I|$} & \multirow{2}{*}{Instance} & \multicolumn{4}{c}{ Solution Time (sec.)} \\ \cline{3-6} & & {\tt REFO} & {\tt OA} & {\tt OA-B\&C} & {\tt R-OA-B\&C} \\
 \hline
\multirow{6}{*}{50} 
& 1 & 7 & 3 & 3 & 2 \\
& 2 & 25 & 3 & 3 & 2\\
& 3 & 17 & 2 & 2 & 2\\
& 4 & 17 & 3 & 3 & 2\\
& 5 & 6 & 3 & 2 & 2\\
& Average & 14 & 3 & 3 & 2 \\
\hline
\multirow{6}{*}{100} 
& 1 & 3600 & 12 & 11 & 8   \\
& 2 & 3600 & 13 & 12 & 9\\
& 3 & 3600 & 11 & 10 & 8\\
& 4 & 3600 & 12 & 12 & 9\\
& 5 & 3600 & 12 & 12 & 9\\
& Average & 3600 & 12 & 11 & 9 \\
\hline
\multirow{6}{*}{150} 
& 1 & 3600 & 46 &  60 & 34 \\
& 2 & 3600 & 47 &  33 & 24\\
& 3 & 3600 & 37 &  29 & 25 \\
& 4 & 3600 & 41 &  43 & 26\\
& 5 & 3600 & 39 &  32 & 26\\
& Average & 3600 & 42 & 39 & 27 \\
 \hline
\multirow{6}{*}{200} 
& 1 & 3600 & 170 &  151 & 99  \\
& 2 & 3600 & 109 &  105 & 67\\
& 3 & 3600 & 161 &  122 & 96\\
& 4 & 3600 & 138 &  170 & 112\\
& 5 & 3600 & 221 &  169 & 129\\
& Average & 3600 & 160 & 143 & 101 \\
 \hline
\multirow{6}{*}{250} 
& 1 & 3600 & 706 &  358 & 250 \\
& 2 & 3600 & 517 &  302 & 200\\
& 3 & 3600 & 408 &  227 & 232\\
& 4 & 3600 & 431 &  440 & 265 \\
& 5 & 3600 & 233 &  218 & 237\\
& Average & 3600 & 459 & 309 & 237\\
 \hline
 \multirow{6}{*}{300} 
& 1 & 3600 & 831 & 252 & 406 \\
& 2 & 3600 & 608 & 374 & 290\\
& 3 & 3600 & 935 & 706 & 428 \\
& 4 & 3600 & 656 & 813 & 456 \\
& 5 & 3600 & 288 & 446 & 146\\
& Average & 3600 & 664 & 518 & 345 \\
 \hline
 \multirow{6}{*}{350} 
& 1 & 3600 & 595 & 699 & 475 \\
& 2 & 3600 & 1160 & 540 & 490 \\
& 3 & 3600 & 3172 & 893 & 497 \\
& 4 & 3600 & 710 & 620 & 429 \\
& 5 & 3600 & 1872 & 337 & 492\\
& Average & 3600 & 1502 & 618 & 477 \\
 \hline
 \multirow{6}{*}{400} 
& 1 & 3600 & 1515  & 778 & 545  \\
& 2 & 3600 & 1442 & 560 & 803 \\
& 3 & 3600 & 2824 & 1249 & 532 \\
& 4 & 3600 & 596 & 1295 & 365 \\
& 5 & 3600 & 847 & 471 & 655 \\
& Average & 3600 & 1445 & 871 & 580 \\
 \hline
 \multirow{6}{*}{450} 
& 1 & 3600 & 965 & 669 & 1901 \\
& 2 & 3600 & 2155 & 934 & 775\\
& 3 & 3600 & 3600 & 1815 & 601\\
& 4 & 3600 & 810 & 833 & 508\\
& 5 & 3600 & 1966 & 1033 & 2265\\
& Average & 3600 & 1899 & 1057 & 1210 \\
 \hline
\multirow{6}{*}{500} 
& 1 & 3600 & 3600 & 2943 & 1059 \\
& 2 & 3600 & 3600 & 1814 & 1245\\
& 3 & 3600 & 3600 & 3600 & 975\\
& 4 & 3600 & 2865 & 1208  & 2359\\
& 5 & 3600 & 1728 & 1872 & 637\\
& Average & 3600 & 3079 & 2287 & 1255 \\
 \hline
\end{tabular}%
}
\vspace{-0.135in}
\caption{\label{table_compute} Computational Efficiency for Base Case Scenario: $q = 10$ and $p = 11$}
\end{table}

\begin{table}[h]
\centering
% \setlength\extrarowheight{-3pt}
%\resizebox{0.8\columnwidth}{2.5cm}
%{
\begin{tabular}{P{4cm} |P{2.5cm}|P{2.5cm}|P{2.5cm} | P{2.5cm}}
\hline
 Solution Time (seconds) & {\tt REFO} & {\tt OA} & {\tt OA-B\&C} & {\tt R-OA-B\&C}  \\
  \hline
 600 & 10\% & 54\% & 64\% 
 & 78\% \\ 
 \hline
 1200 & 10\% & 74\% & 84\%& 92\%  \\ 
 \hline
  1800 & 10\% & 80\% & 90\% & 94\% \\ 
 \hline
  2400 & 10\% & 86\% & 96\% & 100\% \\
 \hline
  3000 & 10\% & 90\% & 98\% & 100\% \\
 \hline
   3600 & 10\% & 92\% & 98\% & 100\% \\
  \hline
\end{tabular}
%}
\caption{\label{Num_solved} Percentage of Problem Instances Solved to Optimality}
\end{table}
%%%%%%%%%%%%%%%%%%%%%%%%%%%%%%%%%%%%%%%%%%%%%%%%%%%%
%%%%%%%%%%%%%%%%%%%%%%%%%%%%%%%%%%%%%%%%%%%%%%%%%%%%

\end{appendices}

\end{document}